\documentclass{article}
\usepackage{tikz-cd}
\usetikzlibrary{cd}

\usepackage[all,pdf]{xy}
\usepackage{graphicx} 
\usepackage{titlesec}
\usepackage[numbers,sort&compress]{natbib}
\usepackage{float}
\usepackage{setspace}
\usepackage{authblk}
\usepackage{tocloft}
\usepackage{xpatch}
\usepackage{mathptmx}

\xpatchcmd{\author}{\relax#1\relax}{\relax\detokenize{#1}\relax}{}{}
\usepackage[colorlinks,linkcolor = blue,anchorcolor = red,urlcolor = blue,citecolor = blue]{hyperref}
\usepackage{amsthm,amsmath,amssymb}
\usepackage[utf8]{inputenc}
\usepackage{geometry}
\bibpunct[, ]{[}{]}{,}{n}{,}{,}
\makeatletter
\def\NAT@def@citea{\def\@citea{\NAT@separator}}
\makeatother
\theoremstyle{plain}
\newtheorem{theorem}{Theorem}[section]
\newtheorem{lemma}[theorem]{Lemma}

\newtheorem{proposition}[theorem]{Proposition}
\theoremstyle{definition}
\newtheorem{definition}[theorem]{Definition}
\newtheorem{example}[theorem]{Example}
\newtheorem*{question}{Question}
\newtheorem{remark}[theorem]{Remark}

\author[a,b]{Liang Guo}
\author[a,c]{Qin Wang}
\author[\empty]{Chen Zhang\textsuperscript{a,}\thanks{Corresponding author.}}
\affil[a]{\small\it{Research Center for Operator Algebras, School of Mathematical Sciences, East China Normal University, Shanghai, 200241, P.~R.~China.}}
\affil[b]{\small\it{Shanghai Institute for Mathematics
and Interdisciplinary Sciences, Shanghai, 200433, P.~R.~China}}
\affil[c]{\small\it{Key Laboratory of MEA, Ministry of Education, and Shanghai Key Laboratory of PMMP, East China Normal University, Shanghai 200241, P.~R.~China}}

\title{The maximal coarse Baum-Connes conjecture for spaces that admit an A-by-FCE coarse fibration structure\footnotetext{\emph{ Email address: 
liangguo@simis.cn (L.~Guo),
qwang@math.ecnu.edu.cn (Q.~Wang),
52275500018@stu.ecnu.edu.cn (C.~Zhang)}}}
\date{\today}

\begin{document}
\maketitle
\noindent\textbf{Abstract.} In this paper, we introduce a concept of A-by-FCE coarse fibration structure for metric spaces, which serves as a generalization of the A-by-CE structure for a sequence of group extensions proposed by Deng, Wang, and Yu. We prove that the maximal coarse Baum-Connes conjecture holds for metric spaces with bounded geometry that admit an A-by-FCE coarse fibration structure. As an application, the relative expanders constructed by Arzhantseva and Tessera, as well as the box spaces derived from an ``amenable-by-Haagerup'' group extension, admit the A-by-FCE coarse fibration structure. 
Consequently, the maximal coarse Baum-Connes conjecture holds for these spaces, which may not admit an FCE structure, i.e. fibred coarse embedding into Hilbert space.


\noindent\textbf{2020 MR Subject Classification.} 46L80

\begin{spacing}{1.2}
\tableofcontents
\end{spacing}

\section{Introduction}

Let $X$ be a discrete proper metric space. The (maximal) coarse Baum-Connes conjecture, initially proposed in \cite{Roe93,HR95,Yu95} and extended to its maximal version in \cite{GWY08}, is a coarse geometric analogue of the original Baum-Connes conjecture for groups \cite{BCH94}. It states that the following coarse assembly map
\[\mu:\lim_{d\rightarrow\infty}K_{*}(P_{d}(X))\rightarrow K_{*}(C^{*}(X))(\text{or}\ K_{*}(C^{*}_{\max}(X))\]
is an isomorphism, where $K_{*}(P_{d}(X))$ is a topological object involving $K$-homology groups of the Rips complex of $X$, and $K_{*}(C^{*}(X))(\text{or}\ K_{*}(C^{*}_{\max}(X)))$
is the $K$-theory groups of the (maximal) Roe algebra of $X$. A positive answer to this conjecture transforms the problem of calculating the higher indices of generalized elliptic operators on non-compact spaces to the computation of locally finite $K$-homology group, thus establishing a bridge between Connes' theory of noncommutative geometry and classical commutative geometry. In particular, it implies the Novikov conjecture and the non-existence of positive scalar curvature metrics on uniformly contractible complete Riemannian manifolds, as well as Gromov’s zero-in-the-spectrum conjecture stating that the Laplacian operator acting on the space of all $L^{2}$-forms of a uniformly contractible Riemannian manifold has zero in
its spectrum.

A significant contribution to proving this conjecture was presented by G. Yu \cite{Yu00}, who demonstrated that the coarse assembly map $\mu$ is an isomorphism for metric spaces with bounded geometry that coarsely embed into Hilbert space. Here, a metric space $X$ with \emph{bounded geometry} means that for any $r>0$ there exists $N_{r}>0$ such that any ball of radius $r$ in $X$ contains at most $N_{r}$ elements. A map $f:X\rightarrow H$ from a metric space $X$ to a Hilbert space $H$ is said to be a \emph{coarse embedding}\cite{Gro93} if there exist two non-decreasing functions $\rho_{1}$ and $\rho_{2}$ from $[0,\infty)$ to $[0,\infty)$ with $\lim_{r\rightarrow \infty} \rho_{i}(r)=\infty, i=1,2,$ such that
\[\rho_{1}(d_{X}(x,x'))\leqslant \|f(x)-f(x')\|\leqslant \rho_{2}(d_{X}(x,x')),\]
for all $x,x'\in X$. Also in \cite{Yu00}, G. Yu introduced the concept of Property A as a coarse analogue of amenability and showed that bounded geometry spaces with Property A can all be coarsely embedded into Hilbert space, see also \cite{Wil09}. Then the following natural question arises.

\begin{question}
    Can all spaces be coarsely embedded into Hilbert space? If not, what about their coarse Baum-Connes conjecture?
\end{question}

$\bullet$ \textbf{Negative side:} The works of \cite{Enf69,DGLY02,Gro03} inspired that a sequence of expander graphs cannot be coarsely embeded into Hilbert space. Specifically, motivated by a graph sequence constructed by Enflo \cite{Enf69}, M. Gromov discovered that the Pinski-Margulis-Selberg expander graph may be coarsely embedded into a certain finitely generated group, and he used a method of random groups to construct a concrete example which is known as Gromov's monster groups \cite{Gro03}. For a long time, expander graphs were the only known obstruction for a bounded geometry space to admit a coarse embedding into Hilbert space.
As for the second part of the question, expander graphs also bring us bad news. In \cite{HLS02}, N. Higson, V. Lafforgue, and G. Skandalis provided a counterexample to the coarse Baum-Connes conjecture for certain Margulis-type expanders. Furthermore, R. Willett and G. Yu \cite{WY121} showed that a larger class of expander graphs beyond Margulis-type expander graphs also make the surjectivity of the coarse assembly map fail.

$\bullet$ \textbf{Positive side:} Although the coarse Baum-Connes conjecture fails for expander graphs, its maximal version might be true. A representative work in this direction was given by G. Gong, Q. Wang, and G. Yu in \cite{GWY08}. They introduced the maximal Roe algebra $C^{*}_{\max}(X)$ of $X$ and proved the maximal coarse Novikov conjecture (i.e. injectivity of the maximal coarse assembly map), for the box space of a class of residually finite groups, including V. Lafforgue's sequences of expanders in \cite{Laf08}. Meanwhile, H. Oyono-Oyono and G. Yu \cite{OOY09} 
showed the maximal coarse Baum-Connes assembly map is an isomorphism for certain expander graphs constructed from
spaces with isometric actions by residually finite groups. Afterward, R. Willett and G. Yu \cite{WY12} considered a class of expanders with large girth and proved the maximal coarse Baum-Connes conjecture for its coarse disjoint union. The expander graphs studied above cannot be globally coarsely embedded into Hilbert space. However,  X. Chen, Q. Wang, and G. Yu discovered that large bounded subsets of $X$ can be coarsely embedded into Hilbert space within a common distortion as long as these subsets are far away toward infinity. This feature is referred to as ``fibred coarse embedding into Hilbert spaces" in \cite{CWY13}, which is a generalization of Gromov's notion of coarse embedding \cite{Gro93} into Hilbert space. It turns out that the maximal coarse Baum-Connes conjecture holds for expander graphs that admit such an embedding, including the major results of \cite{OOY09} and \cite{WY12}.

Most of the work we mentioned above was conducted from the perspective that expanders are counterexamples to the coarse Baum-Connes conjecture. However, metric spaces that cannot be closely embeded into Hilbert space do not always contain expander graphs.

In \cite{AT15}, G. Arzhantseva and R. Tessera introduced the concept of ``relative expander" and constructed a metric space with bounded geometry that cannot be coarsely embedded into any $L^{p}$-space for $1\leqslant p<\infty$. After that, T. Delabie and A. Khukhro \cite{DK18} utilized the representation theory of linear groups over finite fields and Ramanujan graphs to create a box space for the free group $\mathbb{F}_{3}$. This space cannot be coarsely embedded into Hilbert space and yet does not contain any expanders as well. Motivated by D. Osajda's work \cite{Osa20} on specific formulations for Gromov monster groups and Haagerup monster groups, G. Arzhantseva and R. Tessera \cite{AT19} employed the wreath product of groups to build two examples $\mathbb{Z}_{2}\wr_{G}H$ and $\mathbb{Z}_{2}\wr_{G}(H\times \mathbb{F}_{n})$, where $G$ is a Gromov monster group and $H$ is a Haagerup monster group. These two groups contain relative expanders so that they cannot be coarsely embedded into $L^{p} (1\leqslant p<\infty)$ space and do not contain any expanders. 

Recently, Deng, Wang, and Yu \cite{DWY23} observed that these three examples can all be regarded as group extensions that admit a ``CE-by-CE" structure. Precisely, for a sequence of group extensions
$(1\rightarrow N_{n}\rightarrow G_{n}\rightarrow Q_{n}\rightarrow 1)_{n\in\mathbb{N}}$ of finitely generated groups with uniformly finitely generated subsets, if the coarse disjoint unions of $\{N_{n}\}_{n\in\mathbb{N}}$ and $\{Q_{n}\}_{n\in\mathbb{N}}$ are coarsely embeddable into Hilbert space, then the coarse Baum-Connes conjecture holds for the coarse disjoint union
of $\{G_{n}\}_{n\in\mathbb{N}}$. This result implies the coarse Baum-Connes conjecture for the relative expanders constructed by G. Arzhantseva and R. Tessera, and the special box spaces of free groups discovered by T. Delabie
and A. Khukhro. At this stage, one might naturally explore the coarse Baum-Connes conjecture for the case of ``FCE-by-FCE". In \cite{DGWY25}, Deng, Guo, Wang, and Yu have shown that the injectivity of the (maximal) coarse assembly map holds for such spaces. But the surjectivity is still unknown.

Inspired by the above developments, in this paper, we introduce a notion of ``A-by-FCE coarse fibration structure" for metric spaces with bounded geometry. It is a generalization of the A-by-CE structure for a sequence of group extensions proposed in \cite{DWY23}. We shall prove that the maximal coarse Baum-Connes conjecture holds for metric spaces that admit such a coarse structure. The following theorem is our main result.
\begin{theorem}\label{the. 1.1}
 Let $X$ be a discrete metric space with bounded geometry. If $X$ admits an A-by-FCE coarse fibration structure, then the maximal coarse Baum-Connes conjecture holds for $X$.
\end{theorem}

Typical metric spaces with an A-by-FCE coarse fibration structure can also be constructed through a sequence of group extensions. In particular, the relative expanders established by Arzhantseva and Tessera in \cite{AT15}, which admit an A-by-CE structure \cite{DWY23}, also support our A-by-FCE coarse fibration structure, however, they may not support FCE (i.e., fibred coarse embedding into Hilbert space). 
Since $K$-amenability is generally false for A-by-FCE coarse fibration structures, the original coarse Baum-Connes conjecture cannot be deduced directly from Theorem \ref{the. 1.1}. Moreover, motivated by \cite{CWW13}, the box spaces constructed from an 
``amenable-by-Haagerup'' group extension also allows for the A-by-FCE coarse fibration structure. It turns out that a large class of relative expanders, even expanders, admit an A-by-FCE coarse fibration structure, and Theorem \ref{the. 1.1} implies that their maximal coarse Baum-Connes conjectures are valid.

Here is an outline of the paper. In Section \ref{section 2}, we introduce the notion of ``A-by-FCE coarse fibration structure" and show that spaces with such a structure can be created through a sequence of group extensions. Furthermore, the box space, derived from an 
``amenable-by-Haagerup'' group extension, also exhibits an A-by-FCE coarse fibration structure. In Section \ref{section 3}, we provide a brief overview of Roe algebras, maximal Roe algebras, and the formulation of the maximal coarse Baum-Connes conjecture. In Section \ref{section 4}, we explain the strategy to prove Theorem \ref{the. 1.1}, that is, the problem of proving the maximal coarse Baum-Connes conjecture for a metric space $X$ can be reduced to verifying that the evaluation homomorphism from the $K$-theory of localization algebras at infinity to the $K$-theory of the maximal Roe algebra at infinity for the coarse disjoint union of a sequence of subspaces of $X$ is an isomorphism. In Section \ref{section 5}, we introduce the maximal twisted Roe algebra at infinity and its localization algebras for a sequence of metric spaces with coefficients coming from the fibred coarse embedding of the base space into Hilbert space. On this basis, we construct the Bott map as asymptotic morphisms. In Section \ref{section 6}, we employ the cutting and pasting technique to explore different ideals of the twisted algebras and prove that the evaluation map for the twisted algebras supported on certain coherent systems is an isomorphism. In Section \ref{section 7}, we define the Dirac map as asymptotic morphisms and subsequently establish a geometric analogue of Bott periodicity in finite dimensions. This fact shows that the evaluation map required in Section \ref{section 4} is an isomorphism, which in turn implies our main result, Theorem \ref{the. 1.1}.

\section{A-by-FCE coarse fibration structure}\label{section 2}

In this section, we introduce the concept of an A-by-FCE coarse fibration structure for a discrete metric space with bounded geometry, which generalizes the A-by-CE group extension structure studied by Deng, Wang, and Yu in \cite{DWY23}. It is observed that a large class of relative expanders and certain group extensions admit such a coarse structure. Subsequently, the maximal coarse Baum-Connes conjecture holds by Theorem \ref{the. 1.1}. To start, we review several useful notions and results.

\begin{definition}\label{def. 2.1}
Let $(X,d)$ be a metric space. For any $r > 0, x\in X$, let $B(x,r):=\{x'\in X\mid d(x, x')<r\}$
denote the open ball of radius $r$ about $x$. A \emph{$C$-net} of a metric space $X$ is a countable subset $\varGamma\subset X$ such that there exists $C > 0$ satisfying: (1)
$d(\gamma,\gamma')>C$ for all distinct elements $\gamma,\gamma'\in\varGamma$; (2) for any $x\in X$ there exists $\gamma\in\varGamma$ such that $d(x,\gamma)<C$.
A metric space $X$ is said to have \emph{bounded geometry} if for any $r>0$ there is
$N_{r}>0$ such that any ball of radius $r$ in $X$ contains at most $N_{r}$ elements. \par
Let $(X, d_{X})$ and $(Y, d_{Y})$ be two metric spaces. A map $f:X\rightarrow Y$ is said to be a \emph{coarse embedding} (or \emph{uniform embedding}) if there
exist non-decreasing functions $\rho_{1}$ and $\rho_{2}$ from $[0,\infty)$ to $[0,\infty)$ with $\lim_{r\rightarrow \infty} \rho_{i}(r)=\infty, i=1,2,$ such that
\[\rho_{1}(d_{X}(x,x'))\leqslant d_{Y}(f(x),f(x'))\leqslant \rho_{2}(d_{X}(x,x')),\]
for all $x,x'\in X$.
A metric space $X$ is said to be \emph{coarsely equivalent} to another metric space $Y$ if there
exists a coarse embedding $f:X\rightarrow Y$ such that the image $f(X)$ is a $C$-net of $Y$ for some $C > 0$.\par 
Suppose that $\{(X_{i}, d_{X_{i}})\}_{i\in I}$ and $\{(Y_{i}, d_{Y_{i}} )\}_{i\in I}$ be families of metric spaces with \emph{uniform
bounded geometry} in the sense that for any $r>0$ there exists $N_{r}>0$ such that $B(x,r)$ contains at most $N_{r}$ elements for $x\in X_{i},i\in I$. 
The sequence $\{X_{i}\}_{i\in I}$ is said to be \emph{uniformly coarsely equivalent} (or \emph{equi-coarsely equivalent})
to $\{Y_{i}\}_{i\in I}$ if there exists a sequence of maps $\{f_{i}:X_{i}\rightarrow Y_{i}\}_{i\in I}$, a constant $C>0$
and two non-decreasing functions $\rho_{1}$ and $\rho_{2}$ from $[0,\infty)$ to $[0,\infty)$ with $\lim_{r\rightarrow \infty} \rho_{i}(r)=\infty,i=1,2,$ such that
\begin{enumerate}
\item [(1)] \hspace{0pt}$\rho_{1}(d_{X_{i}}(x,x'))\leqslant d_{Y_{i}}(f_{i}(x),f_{i}(x'))\leqslant \rho_{2}(d_{X_{i}}(x,x'))$ for any $x,x'\in X_{i}$ and $i\in I$;
\item [(2)] \hspace{0pt}the image $f_{i}(X_{i})$ is the $C$-net of
$Y_{i}$ for each $i\in I$.
\end{enumerate}
\end{definition}

The concept of fibred coarse embedding into Hilbert space for metric spaces was first proposed by Chen, Wang, and Yu in \cite{CWY13}. It is a generalization of Gromov’s notion of coarse embedding \cite{Gro93} into Hilbert space.

\begin{definition}[\cite{CWY13}]\label{def. 2.2}
A  metric space $(X,d)$ is said to admit a \emph{fibred coarse embedding into Hilbert space} if there exist\par 
$\bullet$ a field of Hilbert spaces $(H_{x})_{x\in X}$ over $X$;\par
$\bullet$ a section $s:X\rightarrow\bigsqcup_{x\in X}H_{x}$ ($\text{i.e.}\ s(x)\in H_{x}$);\par 
$\bullet$ two non-decreasing functions $\rho_{1}$ and $\rho_{2}$ from $[0,\infty)$ to $[0, \infty)$ with $\lim\limits_{r\rightarrow\infty}\rho_{i}(r)=\infty\ (i=1,2)$\\
such that for any $r>0$ there exists a bounded subset $K\subset X$ for which there exists a ``trivialization"
$t_{C}:(H_{x})_{x\in C}\rightarrow 
C\times H$ for each subset $C\subset X\setminus K$ of diameter less than $r$, i.e. a map from $(H_{x})(x\in C)$ to the constant field $C\times H$ over $C$ such that the restriction of $t_{C}$ to the fiber $H_{x}(x\in C)$ is an affine isometry $t_{C}(x):H_{x}\rightarrow H$, satisfying
\begin{enumerate}
    \item [(1)] \hspace{0pt}for any $x,y\in C$, $\rho_{1}(d(x,y))\leqslant\|t_{C}(x)(s(x))-t_{C}(y)(s(y))\|\leqslant \rho_{2}(d(x,y))$;
    \item [(2)] \hspace{0pt}for any subsets $C_{1},C_{2}\subset X\setminus K$ of diameter less than $r$
    with $C_{1}\cap C_{2}\neq \emptyset$, there exists an affine isometry $t_{C_{1}C_{2}}: H\rightarrow H$ such that $t_{C_{1}}(x)\circ t^{-1}_{C_{2}}(x)=t_{C_{1}C_{2}}$ for all $x\in C_{1}\cap C_{2}$.
\end{enumerate}
\end{definition}
This generalization allows the major results for expander graphs studied in \cite{OOY09} and \cite{WY12} to be the special case that admits a fibred coarse embedding into Hilbert space, ensuring the fulfillment of their maximal coarse Baum-Connes conjectures. Especially, results in \cite{OOY09,OO01} indicate that the maximal coarse Baum-Connes conjecture holds for the box space derived from an 
``amenable-by-Haagerup'' group extension, which will be explained later that they admit an A-by-FCE coarse fibration structure. Next, we review the notions of Property A and equi-Property A, which will be essential in section \ref{section 6} of this paper. For further information, readers can consult \cite{Yu00,HR00,Tu01,BNW07}.

\begin{definition}\label{def. 2.3}
A metric space $X$ is said to have \emph{Property $A$} if for any $R>0$ and $\epsilon>0$, there
exists $S>0$ and a map
$\xi:X\rightarrow \ell^{2}(X)$ such that: (1) for all $x\in X$, $\xi _{x}(y)\in[0,1]$; (2) for all $x,y\in X$, $\|\xi _{x}\|_{2}=1$; (3) $\sup\{\|\xi _{x}-\xi _{y}\|:d(x,y)\leqslant R, x,y\in X\}<\epsilon$; (4) for all $x\in X$, $\text{Supp}(\xi _{x})\subset B(x,S)$.\par 
A family $\{X_{i}\}_{i\in I}$ of discrete metric spaces with
bounded geometry is said to have \emph{equi-Property $A$} if for any $R>0,\epsilon>0$, there
exists $S>0$ and a family of maps 
$\xi^{(i)}:X_{i}\rightarrow \ell^{2}(X_{i}), i\in I$ such that: (1) for all $x, y \in X_{i}$, $\xi ^{(i)}_{x}(y)\in[0,1]$; (2) for all $x,y\in X_{i}$, $\|\xi ^{(i)}_{x}(y)\|=1$; (3) $\sup\{\|\xi ^{(i)}_{x}-\xi ^{(i)}_{y}\|:d(x,y)\leqslant R, x,y\in X_{i}\}<\epsilon$; (4) for all $x\in X_{i}$, $\text{Supp}(\xi ^{(i)}_{x})\subset B(x,S), i\in I$.
\end{definition}


In \cite{GLWZ23}, a notion of ``A-by-CE coarse fibration structure" was defined. In the following, we introduce a generalization of this concept, namely ``A-by-FCE coarse fibration structure".
\begin{definition}[A-by-FCE coarse fibration structure]\label{def. 2.4} A discrete metric space $X$ with bounded geometry is said to admits an \emph{A-by-FCE coarse fibration structure} if there exist a discrete metric space $Y$ (called \emph{the base space}) with bounded geometry and a surjective map $p:X\rightarrow Y$ satisfying
\begin{enumerate}
\item [(1)] \hspace{0pt}the map $p$ is \emph{bornologous} (or \emph{uniformly expansive}), i.e. for any $R>0$, there exists $S>0$ such that \[d_X(x,x')\leqslant R \Rightarrow d_Y(p(x),p(x'))\leqslant S ;\]
\item [(2)] \hspace{0pt}the base space $Y$ admits a fibred coarse embedding into Hilbert space;
\item [(3)] \hspace{0pt}the family of fiber spaces $\{p^{-1}(y)\}_{y\in Y}$ of $X$ has equi-Property A;
\item [(4)] \hspace{0pt}for any $R>0$, the collection $\{p^{-1}(B_Y(y,R))\}_{y\in Y}$ is uniformly coarsely equivalent to $\{p^{-1}(y)\}_{y\in Y}$.
\end{enumerate}

Moreover, a sequence of discrete metric spaces $(X_{n})_{n\in\mathbb{N}}$ with uniform bounded geometry is said to admit an A-by-FCE coarse fibration structure if each $X_n$ admits an A-by-FCE structure with uniform constant as above independent on $n$.
\end{definition}

In the rest of this paper, we shall prove the maximal coarse Baum-Connes conjecture for a discrete metric space $X$ with bounded geometry that admits an A-by-FCE coarse fibration structure. The approach is to reduce the problem to the case of the coarse disjoint union of a sequence of metric spaces $(X_{n})_{n\in\mathbb{N}}$. Here, a coarse disjoint union of $(X_{n})_{n\in\mathbb{N}}$ refers to the disjoint union $X=\bigsqcup_{n\in\mathbb{N}} X_{n}$ equipped with a metric $d$ such that: (1) the restriction of $d$ to each $X_{n}$ is the original metric of $X_{n}$;
(2) $d(X_{n}, X_{m})\rightarrow \infty$ as $n+m\rightarrow \infty$ and $n\neq m$. 
Therefore, the reader will frequently encounter the uniform version of the A-by-FCE coarse fibration structure in the sequel.\par 


We also want to point out that Definition \ref{def. 2.4} given above is slightly stronger from the notion of coarse fibration structure proposed in \cite{GLWZ23}. In this study, condition (4) is utilized to ensure the validity of Proposition \ref{prop. 4.8}  in Section \ref{section 4}. Specifically, if the base space $Y$ permits an $\omega$-excisive decomposition $Y=Y^{(0)}\cup Y^{(1)}$ (Cf. \cite{HRY93}), then the corresponding fiber spaces $p^{-1}(Y^{(0)})$ and $p^{-1}(Y^{(1)})$ forms a partition of $X$, and the pair $(p^{-1}(Y^{(0)}),p^{-1}(Y^{(1)}))$ remains $\omega$-excisive. To ensure the validity of this property, it is sufficient for the fibers to be in approximate proximity without the necessity of precise arrangement. Further elaboration on this matter can be found in the proof of Proposition \ref{prop. 4.8}.

To get a bit more intuition for the A-by-FCE coarse fibration structure, our next examples give more
concrete pictures in the group extension case.

\begin{example}\label{exa. 2.5} Let $(1\rightarrow N_{n}\rightarrow G_{n}\rightarrow Q_{n}\rightarrow 1)_{n\in\mathbb{N}}$ be a sequence of group extensions of countable discrete groups. If the coarse disjoint union of $(N_{n})_{n\in\mathbb{N}}$ has Property A and the coarse disjoint union of $(Q_{n})_{n\in\mathbb{N}}$ admits a fibred coarse embedding into Hilbert space, then the coarse disjoint union of $(G_{n})_{n\in\mathbb{N}}$ admits an A-by-FCE coarse fibration structure. 

For each $n\in\mathbb{N}$, assume that the normal subgroup $N_{n}\lhd G_{n}$ is endowed with the subspace metric of $G_{n}$, and the quotient group $Q_{n}\cong G_{n}/N_{n}$ is endowed with the quotient metric. Denote by $p_{n}: G_{n} \rightarrow Q_{n}$ the canonical quotient map. Obviously, it is a contractive map.

For each $q_{n}\in Q_{n}$, choose $g_{n}\in G_{n}$ such that $p_{n}(g_{n})=q_{n}$. The fiber space of $q_{n}$, denoted by $p_{n}^{-1}(q_{n})$, is exactly the coset $g_{n}N_{n}$. The left-invariant metric of $G_{n}$ implies that
the fiber space $p_{n}^{-1}(q_{n})$ is isometric to another $p_{n}^{-1}(q'_{n})$ for any $q_{n}, q'_{n}\in Q_{n}$ by
\[g_{n}N_{n}\rightarrow g'_{n}N_{n},\quad g_{n}x_{n}\mapsto g'_{n}x_{n}.\]
If $q_{n}, q'_{n}\in Q_{n}$ satisfy $d_{Q_{n}}(q_{n},q'_{n})<R$ for some $R> 0$, then for any $x\in p_{n}^{-1}(q_{n})$,
there exists $x'\in p_{n}^{-1}(q'_{n})$ such that $d_{G_{n}}(x,x')<R$. The definition of quotient metrics guarantees the existence of $x'$. Therefore, there exists $C\geqslant R$
such that the family of fiber spaces $\{p_{n}^{-1}(q_{n})\}_{q_{n}\in Q_{n}}$ becomes a $R$-net of $\{p_{n}^{-1}(B_{Q_{n}} (q_{n},R))\}_{q_{n}\in Q_{n}}$. From
Definition \ref{def. 2.1} we conclude that the
fiber spaces $\{p_{n}^{-1}(q_{n})\}_{q_{n}\in Q_{n},n\in\mathbb{N}}$ is uniformly coarsely equivalent
to $\{p_{n}^{-1}(B_{Q_{n}}(q_{n},R))\}_{q_{n}\in Q_{n},n\in\mathbb{N}}$. 
Since the coarse disjoint union of normal subgroups $(N_{n})_{n\in\mathbb{N}}$ has Property A, we have that the sequence $\{p_{n}^{-1}(q_{n})\}_{q_{n}\in Q_{n},n\in\mathbb{N}}
=\{q_{n}N_{n}\}_{q_{n}\in Q_{n},n\in\mathbb{N}
}$ has equi-Property A. Consequently, the coarse disjoint union of $(G_{n})_{n\in\mathbb{N}}$ admits an A-by-FCE coarse fibration structure.

A specific illustration of spaces exhibiting an A-by-FCE coarse fibration structure as a sequence of group extensions was constructed by Arzhantseva and 
Tessera in \cite{AT15}. In particular, these spaces admit an A-by-FCE coarse fibration structure but not FCE (i.e. fibred coarse embedding into Hilbert space).
\end{example}

\begin{example}\label{exa. 2.6}
Let $\varDelta$ be an amenable group and $\varOmega$ be a group with the Haagerup property.  An extension
of $\varOmega$ by $\varDelta$ is an exact sequence $1\rightarrow \varDelta\rightarrow \varGamma\xrightarrow{\pi} \varOmega\rightarrow 1$ of groups. Suppose that $(\varGamma_{n})_{n\in\mathbb{N}}$ is a sequence of finite index normal subgroups of $\varGamma$ satisfying
for any $r > 0$ there exists $N\in\mathbb{N}$ such that $\varGamma_{n}\cap B_{\varGamma}(e,r)=\{e\}$ for all $n\geqslant{N}$, where $B(e,r)$ is the ball in $\varGamma$ of radius $r$ about the identity $e$. It follows from \cite{CWW13} that 
\begin{align*}
    \varDelta \ \text{amenable}&\Longleftrightarrow \text{Box}_{\{\varDelta_{n}\}}(\varDelta)\ \text{Yu's Property A};\\
    \varOmega \ \text{Haagerup}&\Longleftrightarrow \text{Box}_{\{\varOmega_{n}\}}(\varOmega)\ \text{ fibred coarsely embeddable into Hilbert space},
\end{align*}
where $\varDelta_{n}=\varDelta\cap\varGamma_{n}$ and $\varOmega_{n}=\pi(\varGamma_{n})$ for each $n\in\mathbb{N}$.
From Example \ref{exa. 2.5} above we can conclude that the box space $X(\varGamma)=\bigsqcup_{n\in\mathbb{N}}\varGamma/
\varGamma_{n}$ possesses an A-by-FCE coarse fibration structure. Theorem \ref{the. 1.1} asserts the validity of the maximal coarse Baum-Connes conjecture for the box space $X(\varGamma)$, a result that can also be demonstrated through the approach outlined in \cite{OO01, OOY09} by H. Oyono-Oyono and G. Yu.
\end{example}

\section{The maximal coarse Baum-Connes conjecture}\label{section 3}

The aim of this section is to review some notions and results concerning the Roe algebra and the maximal Roe algebra of a proper metric space with bounded geometry. Furthermore, a brief formulation of the maximal coarse Baum-Connes conjecture is also presented. Readers are encouraged to consult \cite{WY20,GWY08, OOY09} for more detailed information.

Let $X$ be a proper metric space (a metric space is called \textit{proper} if every closed ball is
compact). An $X$-module $H_{X}$ is a separable Hilbert space equipped with a $*$-representation $\pi$ of $C_{0}(X)$ (the algebra of all continuous functions on $X$ which vanish at infinity). 
An $X$-module
is called \emph{non-degenerate} if the $*$-representation of $C_{0}(X)$ is non-degenerate. An $X$-module is
said to be \emph{ample} if no nonzero function in $C_{0}(X)$ acts as a compact operator. When $H_{X}$ is an
$X$-module, for each $f\in C_{0}(X)$ and $h\in H_{X}$, we abbreviate $(\pi(f))h$ by $fh$.

\begin{definition}(Cf. \cite{Roe93}.)\label{def. 3.1}
Let $H_{X}$ be an ample non-degenerate $X$-module.\par
\begin{enumerate}
	\item [(1)] \hspace{0pt}The \emph{support} of a bounded linear operator $T:H_{X}\rightarrow H_{X}$, denoted by $\text{supp}(T)$, is defined to be the complement of the set of all points $(x,x')\in X\times X$ for which there exist $f,g\in C_{0}(X)$ such that
	$gTf=0$ but $f(x)\neq0$, $g(x')\neq0$.
	\item [(2)] \hspace{0pt} The \emph{propagation} of 
 a bounded operator $T:H_{X}\rightarrow H_{X}$ is defined by \[\text{propagation}(T):=\sup\{d(x,x'):(x,x')\in \text{supp}(T)\}.\]
$T$ is said to have \textit{finite propagation} if
	the number is finite.
	\item [(3)] \hspace{0pt}A bounded operator $T:H_{X}\rightarrow H_{X}$ is said to be $locally$ $compact$ if the operators $fT$ and $Tf$
	are compact for all $f\in C_{0}(X)$.
\end{enumerate}
\end{definition}
\begin{definition}(Cf. \cite{Roe93}.)\label{def. 3.2} Let $H_{X}$ be an ample non-degenerate $X$-module.\par
\begin{enumerate}
	\item [(1)] \hspace{0pt}The algebraic Roe algebra $\mathbb{C}[X, H_{X}]$, or denoted by $\mathbb{C}[X]$, is the $*$-algebra of all locally compact and finite propagation operators on the ample non-degenerate $X$-module $H_{X}$. 
\item [(2)]  \hspace{0pt}The Roe algebra (also known as the reduced Roe algebra), denoted by $C^{*}(X,H_{X})$, or simply $C^{*}(X)$, is defined to be the operator norm closure of $\mathbb{C}[X,H_{X}]$ in the bounded linear operators on $H_{X}$. 
\end{enumerate}
\end{definition}
Note that the algebraic Roe algebra $\mathbb{C}[X,H_{X}]$, up to a non-canonical isomorphism, does not depend on the choice of ample non-degenerate $X$-module (see \cite{Yu97}). In particular, we choose $H_{X}$ to be the Hilbert space $\ell^{2}(Z)\otimes H_{0}$, where $Z$ is a countable dense subset of $X$ and $H_{0}$ is a separable infinite-dimensional Hilbert space. A function $f\in C_{0}(X)$ acts on 
$\ell^{2}(Z)\otimes H_{0}$ by
pointwise multiplication \[f(\xi\otimes h)=f\xi\otimes h\]
for all $\xi\in \ell^{2}(Z)$ and $h\in H_{0}$.
Since $\ell^{2}(Z)\otimes H_{0}=\oplus_{x\in Z}\mathbb{C}\delta _{x}\otimes H_{0}$, where $\delta_{x}$ is the Dirac function, we can express each bounded linear operator $T\in \mathcal{B}(\ell^{2}(Z)\otimes H_{0})$ as a $Z$-by-$Z$ matrix
\[T=(T(x,x'))_{x,x'\in{Z}},\] where $T(x,x')$ is a bounded linear operator from $\mathbb{C}\delta_{x'}\otimes H_{0}$ to $\mathbb{C}\delta_{x}\otimes H_{0}$. Then we have \par 
$\bullet$ the propagation of $T$ is $\sup\{d(x,x')\mid T(x,x')\neq 0\};$\par 
$\bullet$ if $T$ is locally compact, then $T(x,x')$ is a compact operator on $H_{0}$ for all $x, x'\in Z$.

\begin{definition}\label{def. 3.3}
    Define $\mathbb{C}_{f}[X]$ to be the $*$-algebra of all bounded functions $T:Z\times Z\rightarrow\mathcal{K}:=\mathcal{K}(H_{0})$ such that 
\begin{enumerate}
	\item[(1)] \hspace{0pt}for any bounded subset $B\subset X$, the set
	$\{(x,x')\in (B\times B)\cap (Z\times Z)\mid T(x,x')\neq 0 \}$
	is finite;
	\item[(2)] \hspace{0pt}there exists $L>0$ such that 
	\[\#\{x'\in Z\mid T(x,x')\neq 0\}<L \quad\text{and}\quad 
	\#\{x'\in Z\mid T(x',x)\neq 0\}<L\]
	for all $x\in Z$ (here, $\#A$ denotes the number of elements in a set $A$);
	\item[(3)] \hspace{0pt}there exists $R>0$ such that $T(x,x')=0$ whenever $d(x,x')>R$ for any $x,x'\in Z.$
\end{enumerate}
\end{definition}
The $*$-algebra $\mathbb{C}_{f}[X]$ is naturally equal to the algebraic Roe algebra $\mathbb{C}[X]$ when $X$ is a discrete metric space with bounded geometry (in general case $\mathbb{C}_{f}[X]$ is a dense $*$-subalgebra of $\mathbb{C}[X]$ within $C^{*}(X)$, but not equal). We shall use $\mathbb{C}_{f}[X]$ to replace $\mathbb{C}[X]$ as a generating subalgebra of $C^{*}(X)$ in the 
sequel.
\begin{lemma}\label{lem. 3.4}(Cf. \cite{GWY08}.)  Let $X$ be a proper metric space with bounded geometry, and let $H_{X}$ be an ample non-degenerate $X$-module. For any $r>0$ there exists a constant $c>0$ such that for any
$*$-representation $\phi$ of $\mathbb{C}[X]$ on a Hilbert space $H_{\phi}$ and any $T\in \mathbb{C}[X]$ with propagation less than $r$, we have that
\[\|\phi(T)\|_{B(H_{\phi})}\leqslant c \|T\|_{B(H_{X})}.
\]
\end{lemma}
This result is essentially proved in \cite{GWY08} and allows us to define the maximal Roe algebra. 
\begin{definition}\label{def. 3.5}(Cf.\cite{GWY08}.)  Let $X$ be a proper metric space with bounded geometry. The maximal
Roe algebra of $X$, denoted by $C^{*}_{\max}(X)$, is the completion of $\mathbb{C}[X]$ with respect to the $C^{*}$-norm:
\[\|T\|_{\max}:=\sup\{\|\phi(T)\|_{B(H_{\phi})}\mid\phi:\mathbb{C}[X]\rightarrow B(H_{\phi}),\ \text{a}\ *\text{-representation}\}.\]
\end{definition}
With the assumption that $H_{X}=\ell^{2}(Z)\otimes H_{0}$, the maximal Roe algebra $C^{*}_{\max}(X)$ does not depend on the choice of countable dense subset $Z$ up to a non-canonical isomorphism. Moreover, if the metric spaces $X$ and $Y$ are coarsely equivalent,
their corresponding maximal Roe algebras $C^{*}_{\max}(X)$ and $C^{*}_{\max}(Y)$ are isomorphic via a non-canonical isomorphism. This result is still valid at the level of $K$-theory. One can refer to \cite{HRY93} for more details.\par 
In the rest of this subsection, we shall define the index map $\text{Ind}_{\max}$.\par 
Recall first that the $K$-homology groups $K_{i}(X) =
KK^{i}(C_{0}(X),\mathbb{C}) (i = 0, 1)$ for a proper metric space $X$ are generated by certain cycles (abstract elliptic operators) modulo
certain equivalence relations \cite{Kas75,Kas88,Bla98}:\par 
\begin{enumerate}
	\item [(1)] \hspace{0pt}a cycle for $K_{0}(X)$ is a pair $(H_{X},F)$, where $H_{X}$ is an $X$-module and $F$ is a bounded linear
	operator acting on $H_{X}$ such that $F^{*}F-I$ and $FF^{*}-I$ are locally compact, and $f F-Ff$
	is compact for all $f\in C_{0}(X)$;
	\item [(2)] \hspace{0pt} a cycle for $K_{1}(X)$ is a pair $(H_{X},F)$, where $H_{X}$ is an $X$-module and $F$ is a self-adjoint
	operator acting on $H_{X}$ such that $F^{2}-I$ is locally compact, and $fF-Ff$ is compact for all
	$f\in C_{0}(X)$.
\end{enumerate}\par 
In the above description of cycles for $K_{i}(X)$, the $X$-module $H_{X}$ can always be chosen to be ample and non-degenerate, and the equivalence relations on cycles are given by homotopy of the operators $F$,
unitary equivalence, and direct sum with ``degenerate" cycles, i.e. those cycles for which $Ff-fF$,
$f(F^{*}F-I)$ and so on, are not merely compact but actually zero \cite{Kas75, Kas88}.
\begin{definition}\label{def. 3.6} 
 Let $(H_{X},F)$ represent a cycle for $K_{0}(X)$ such that $H_{X}$ is a standard non-degenerate $X$-module. Let $\{U_{i}\}_{i}$ be a
locally finite (an open cover $\{U_{i}\}_{i}$ of $X$ is \emph{locally finite} if every compact subset $K\subseteq X$ is contained in only finitely many elements of the cover) and uniformly bounded (i.e. for any $R>0$, $diameter(U_{i})<R$ for all $i$). 
Let $\{f_{i}\}_{i}$ be a continuous partition of unity subordinate to the open cover $\{U_{i}\}_{i}$. Define \[F'=\sum_{i}f_{i}^{\frac{1}{2}}Ff_{i}^{\frac{1}{2}},\]
where the infinite sum converges in the strong operator topology. Note that $(H_{X},F')$ is a cycle in $K_{0}(X)$, and equivalent to  $(H_{X},F)$ via the homotopy $(H_{X},(1-t)F+tF')$, where $t\in[0,1]$.
Since $F'$ has finite propagation $R>0$, we have that $F'$
is a multiplier of $C^{*}_{\max}(X)$ and easy to verify that ${F'}^{2}-1\in C^{*}_{\max}(X)$. Thus, $F'$ is invertible modulo $C^{*}_{\max}(X)$. Hence, $F'$ give rise to an element, denoted by $\partial[F']$, in $K_{0}(C^{*}_{\max}(X))$, where 
\[\partial:K_{1}(M(C^{*}_{\max}(X))/C^{*}_{\max}(X))\rightarrow K_{0}(C^{*}_{\max}(X))\]
induced by the following six-term exact sequence
$$\begin{tikzcd}
		K_{0}(C^{*}_{\max}(X)) \arrow[r] & K_{0}(M(C^{*}_{\max}(X))) \arrow[r] & K_{0}(M(C^{*}_{\max}(X))/C^{*}_{\max}(X)) \arrow[d] \\
		K_{1}(M(C^{*}_{\max}(X))/C^{*}_{\max}(X)) \arrow[u,"\partial"] & K_{1}(M(C^{*}_{\max}(X))) \arrow[l] & K_{1}(C^{*}_{\max}(X)) \arrow[l]
\end{tikzcd}$$
where $M(C^{*}_{\max}(X))$ represents the multiplier algebra of $C^{*}_{\max}(X)$.
We define the index map $\text{Ind}_{\max}$ from $K_{0}(X)$ to $K_{0}(C^{*}_{\max}(X))$ by
\[\text{Ind}_{\max}([(H_{X},F)]):=\partial[F'].\]
More specifically, consider the product (Cf. \cite{Mil71})
\[W=\left(
\begin{array}{cc}
	I &F'\\
	0&I\\
\end{array}\right)\left(
\begin{array}{cc}
	I &0\\
	-(F')^{*}&I\\
\end{array}\right)\left(
\begin{array}{cc}
	I &F'\\
	0&I\\
\end{array}\right)\left(
\begin{array}{cc}
	0 &-I\\
	I&0\\
\end{array}\right)\]
in $M(C^{*}_{\max}(X))\otimes M_{2}(\mathbb{C})$, which is an invertible element that agrees with 
\[\left(\begin{array}{cc}
	0&F'\\
	-(F')^{*}&0\\
\end{array}\right)\]
modulo $C^{*}_{\max}(X)\otimes M_{2}(\mathbb{C})$. We then define
\[\partial[F']:=\left[W\left(
\begin{array}{cc}
	I &0\\
	0&0\\
\end{array}\right)W^{-1}\right]-\left[
\left(
\begin{array}{cc}
	I &0\\
	0&0\\
\end{array}\right)
\right]\in K_{0}(C^{*}_{\max}(X)).\]
Similarly,
one can define the index map from $K_{1}(X)$ to $K_{1}(C^{*}_{\max}(X))$ (see \cite{Yu95}).\end{definition}
\begin{remark}\label{rem. 3.7} 
 Let $X$ be a proper metric space with bounded geometry as above, and $Z$ a countable
dense subset of $X$ used to define $\mathbb{C}_{f}[X]$ as in Definition \ref{def. 3.3}. For any natural number $n>0$, let
$Z_{n}$ be a subset of $Z$ such that $d(x,x')>\frac{1}{2n}$
for distinct $x,x'\in Z_{n}$, and $d(x,Z_{n})\leqslant\frac{1}{n}$ for all
$x\in X$. Without loss of generality, we may assume $Z=\sum^{\infty}_{n=1}Z_{n}$. Let $\mathbb{C}[Z_{n}]$ be the $*$-algebra
of all bounded functions $T:Z_{n}\times Z_{n}\rightarrow K$ with finite propagation, i.e. there exists $R>0$ such
that $T(x,x')=0$ whenever $d(x,x')>R$ for all $x,x'\in Z_{n}$. Then $\mathbb{C}[Z_{n}]\subseteq \mathbb{C}_{f}[X]$. Moreover, there
exists a non-canonical $*$-isomorphism \cite{HRY93}
$\text{Ad}(U):\mathbb{C}[X] \rightarrow \mathbb{C}[Z_{n}]$ such that,
if $T\in \mathbb{C}[X]$ has propagation less than $R$, then the propagation of $(\text{Ad}(U))(T)$ is less than $R+\frac{2}{n}$.
For any $R>0$, let $\mu([(H_{X},T)])\in K_{0}(\mathbb{C}[X])$ be as in Definition \ref{def. 3.6}. Then
\[\text{Ad}(U)_{*}(\mu([(H_{X},T)]))\in K_{0}(\mathbb{C}[Z_{n}]),\]
which also defines an element in $K_{0}(\mathbb{C}_{f}[X])$ via the inclusion $\mathbb{C}[Z_{n}]\hookrightarrow \mathbb{C}_{f}[X]$. Note that the propagation of the element
\[(\text{Ad}(U))(W\left(
\begin{array}{cc}
	I &0\\
	0&0\\
\end{array}\right)W^{-1}-\left(
\begin{array}{cc}
	I &0\\
	0&0\\
\end{array}\right))\in \mathbb{C}_{f}[X]\otimes M_{2}(\mathbb{C})\]
is less that $6R+\frac{2}{n}$. Note also that $\text{Ad}(U)_{*}$ canonically induces the identity on $K_{*}(C^{*}_{\max}(X))
$.\end{remark}

\begin{definition}\label{def. 3.8}
 Let $X$ be a discrete metric space with bounded geometry. For each $d\geqslant 0$, the \textit{Rips
complex} $P_{d}(X)$ at scale $d$ is defined to be the simplicial polyhedron in which the set of vertices
is $X$, and a finite subset $\{x_{0},x_{1},\dots,x_{q}\}\subseteq X$ spans a simplex if and only if $d(x_{i},x_{j})\leqslant d$ for all
$0\leqslant i,j\leqslant q$.
The Rips complex $P_{d}(X)$ is equipped with the \textit{spherical metric} that is the maximal metric whose restriction to each simplex 
 $\{\sum^{q}_{i=0}t_{i}x_{i}\mid t_{i}\geqslant 0,\sum^{q}_{i=0}t_{i}=1\}$ is the metric obtained by identifying the simplex with $S^{q}_{+}$ via the map
\[\sum^{q}_{i=0}t_{i}x_{i}\mapsto(\frac{t_{0}}{\sqrt{\sum^{q}_{i=0}t^{2}_{i}}},\frac{t_{1}}{\sqrt{\sum^{q}_{i=0}t^{2}_{i}}},\dots,\frac{t_{q}}{\sqrt{\sum^{q}_{i=0}t^{2}_{i}}}
)\]
where $S^{q}_{+}:=\{(s_{0},s_{1},\dots,s_{q})\in\mathbb{R}^{q+1}\mid s_{i}\geqslant 0,\sum^{q}_{i=0}s_{i}=1\}$ is endowed with the standard Riemannian metric, defined as follows. For any $x,y\in P_{d}(X)$ and any semi-simplicial path (see \cite{WY20}):
\[\xi=(x_{0},y_{0},x_{1},y_{1},\dots,x_{n},y_{n})\]
connecting $x$ and $y$,
where $x_{0},\dots,x_{n},y_{0},\dots,y_{n}\in X$ with $x=x_{0}, y=y_{0}$, and $x_{i},y_{i}$ in the same simplex. The length of the path $\xi$ is defined by the formula
\[l(\xi):=\sum^{n}_{i=0}d_{S}(x_{i},y_{i})+\sum^{n-1}_{i=0}d_{X}(y_{i},x_{i+1}).\]
Then the spherical metric of $P_{d}(X)$ is defined by
\[d(x,y)=\inf\{l(\xi)\mid\xi\ \text{is a semi-simplicial path connecting}\ x \ \text{and} \ y \},\]
which induces the same topology as the weak topology of the simplicial
complex: a subset $S\subseteq P_{d}(X)$ is closed if and only if the intersection of $S$ with each simplex is
closed.
In addition, for any $d\geqslant 0$, $P_{d}(X)$ is coarsely equivalent to $X$ via the inclusion map. If $d\leqslant d'$,
then $P_{d}(X)$ is included in $P_{d'}(X)$ as a subcomplex via a simplicial map $i_{d'd}$, and induces a $*$-homomorphism from $C^{*}_{\max}(P_{d}(X))$ to $C^{*}_{\max}(P_{d'}(X))$. Furthermore, for $d\leqslant d'\leqslant d''$, the inclusion map $i_{d''d}$
and the composition $i_{d''d'}$ and $i_{d'd}$ induce the same homomorphism from $K_{*}(C^{*}_{\max}(P_{d}(X)))$ to $K_{*}(C^{*}_{\max}(P_{d''}(X)))$.
Passing to the inductive
limit, we obtain the assembly map
\[\mu_{\max}:\lim_{d\rightarrow\infty}K_{*}(P_{d}(X))\rightarrow \lim_{d\rightarrow\infty}K_{*}(C^{*}_{\max}(P_{d}(X)))\cong K_{*}(C^{*}_{\max}(X)).\]\end{definition}
\noindent\textbf{The maximal coarse Baum-Connes conjecture.} \textit{If $X$ is a discrete metric space with bounded
geometry, then the assembly map $\mu_{\max}$ is an isomorphism.}
 
Finally, we recall the definition of localization algebras \cite{Yu97}  and the relation between its $K$-theory
and the $K$-homology groups. 
\begin{definition}\label{def. 3.9}
 Let $X$ be a proper metric space.
\begin{enumerate}
\item [(1)]\hspace{0pt}The algebraic localization algebra, denoted by $\mathbb{C}_{L}[X]$, is defined to be the $*$-algebra of all bounded and
uniformly norm-continuous functions $g:\mathbb{R}_{+}\rightarrow \mathbb{C}[X]$
such that 
\[propagation(g(t))\rightarrow 0,\ \text{as}\ t\rightarrow \infty.\]
	\item [(2)]\hspace{0pt}
The maximal localization algebra $C^{*}_{L,\max}(X)$ is the
completion of $\mathbb{C}_{L}[X]$ 
with respect to the norm \[\|g\|_{\max}
:=\sup_{t\in\mathbb{R}_{+}}\|g(t)\|_{\max}.\]
\end{enumerate}
\end{definition}
We can define
the evaluation-at-zero homomorphism $e$ from
$C^{*}_{L,\max}(X)$ to $C^{*}_{\max}(X)$ by $e(g)=g(0)$ for any $g\in C^{*}_{L,\max}(X)$. There exists a local assembly map \cite{Yu97}: 
\[\mu_{L,\max}:\lim_{d\rightarrow\infty}K_{*}(P_{d}(X))\rightarrow \lim_{d\rightarrow\infty}K_{*}(C^{*}_{L,\max}(P_{d}(X))).\]\par 
The following result establishes the relation between the $K$-homology groups and the $K$-theory of localization algebras. 
\begin{theorem}\label{the. 3.10}(Cf. \cite{Yu97}.) For every finite-dimensional simplicial complex $X$ endowed with the
spherical metric, the local index map $\text{Ind}_{L,\max}:K_{*}(X)\rightarrow K_{*}(C^{*}_{L,\max}(X))$ is an isomorphism.
\end{theorem}
Consequently, if $X$ is a discrete metric space with bounded geometry, we have the following commutative diagram

\begin{center}
	\begin{tikzcd}
	&  & {\lim_{d\rightarrow\infty}K_{*}(C^{*}_{L,\max}(P_{d}(X)))} \arrow[d,"e_{*}"] \\
	{\lim_{d\rightarrow\infty}K_{*}(P_{d}(X))} \arrow[rr,"\mu_{\max}"] \arrow[rru,"\mu_{L,\max}"] &  & {\lim_{d\rightarrow\infty}K_{*}(C^{*}_{\max}(P_{d}(X)))}      \end{tikzcd}
\end{center}
and the maximal coarse Baum-Connes conjecture is a consequence of the result that the map
\[e_{*}:\lim_{d\rightarrow\infty}K_{*}(C^{*}_{L,\max}(P_{d}(X)))\rightarrow
\lim_{d\rightarrow\infty}K_{*}(C^{*}_{\max}(P_{d}(X)))\]
induced by the evaluation map on $K$-theory is an isomorphism.

\section{Reduction to the coarse disjoint union of metric spaces}\label{section 4}
As discussed previously in Section \ref{section 2}, the approach to establish Theorem \ref{the. 1.1} involves reducing the maximal coarse Baum-Connes conjecture for a discrete metric space with bounded geometry to the maximal coarse Baum-Connes conjecture at infinity for the separate disjoint union of a sequence of metric spaces with uniform bounded geometry. The whole process is divided into the following two parts.\par 
$\bullet$ \emph{Cutting the entire space into a coarse disjoint union of a sequence of subspaces.} Let $X$ denote a discrete metric space with bounded geometry that admits an A-by-FCE coarse fibration structure. We shall first prove that if the base space
$Y$ permits an $\omega$-excisive decomposition $Y=Y^{(0)}\cup Y^{(1)}$, where $Y^{(0)}$, $Y^{(1)}$ and $Y^{(0)}\cap Y^{(1)}$ are coarse disjoint unions of sequences of finite subspaces of $Y$, then the corresponding fiber spaces $p^{-1}(Y^{(0)})$ and $p^{-1}(Y^{(1)})$ forms 
a partition of $X$, and the pair $(p^{-1}(Y^{(0)}),p^{-1}(Y^{(1)}))$ remains $\omega$-excisive. Consequently, it suffices to prove Theorem \ref{the. 1.1} for $X=\bigsqcup_{n\in\mathbb{N}} X_{n}$ being a coarse disjoint union of a sequence of metric spaces with uniform bounded geometry. 

$\bullet$ \emph{Reducing the maximal coarse Baum-Connes conjecture to its counterpart ``at infinity".} We begin by introducing the maximal Roe algebra at infinity and its associated localization algebras to construct the assembly map at infinity. Subsequently, the maximal coarse Baum-Connes conjecture for a sequence of metric spaces $(X_{n})_{n\in\mathbb{N}}$ can be reduced to its counterpart at infinity. This case can be further simplified by confirming that the evaluation homomorphism from the $K$-theory groups of certain localization algebras to that of the corresponding maximal Roe algebra at infinity is an isomorphism, a fact which will be proved in Sections \ref{section 5}, \ref{section 6} and \ref{section 7}. This section serves to provide an overview of the reduction process.

\begin{definition}(Cf. \cite{HRY93}.)\label{lem. 4.7}
Let $X$ be a proper metric space, and let $A$ and $B$ be closed subspaces
with $X=A\cup B$. We say that $(A,B)$ is an $\omega$-excisive couple, or that $X = A\cup B$ is an $\omega$-excisive decomposition, if for each $R > 0$ there is some $S > 0$ such that
\[\text{Pen}(A;R)\cap \text{Pen}(B;R)\subseteq \text{Pen}(A\cap B ;S).\]
\end{definition}

 \begin{proposition}\label{prop. 4.8}
 Let $X$ be a discrete metric space with bounded geometry that admits an A-by-FCE coarse fibration structure. If the base space $Y$ can be partitioned as $Y=Y^{(0)}\cup Y^{(1)}$ such that each of $Y^{(0)}$, $Y^{(1)}$ and $Y^{(0)}\cap Y^{(1)}$ is a coarse disjoint union of a sequence of finite subspaces of $Y$, and the pair $(Y^{(0)}, Y^{(1)})$ is $\omega$-excisive, then their corresponding fiber spaces $p^{-1}(Y^{(0)})$, $p^{-1}(Y^{(1)})$ and $p^{-1}(Y^{(0)})\cap p^{-1}(Y^{(1)})$ remain a coarse disjoint union of sequences of subspaces of $X$. Moreover, the pair $(p^{-1}(Y^{(0)}),p^{-1}(Y^{(1)}))$ is $\omega$-excisive again.
 \end{proposition}
\begin{proof}[\rm\textbf{Proof.}] Since the metric space $X$ admits an A-by-FCE coarse fibration structure, there exist a discrete metric space $Y$ with bounded geometry and a surjective map $p:X\rightarrow Y$ satisfying conditions (1)-(4) in Definition \ref{def. 2.4}. We fix a point $y_{0}\in Y$ and define 
\[Y_{n}:=\{y\in Y\mid n^{3}-n\leqslant d(y,y_{0})\leqslant (n+1)^{3}+(n+1)\}\]for all $n\in\mathbb{N}$. Then the base space $Y$ can be partitioned as a union of 
\[Y^{(0)}=\bigsqcup_{n:\text{even}}Y_{n}\quad \text{and} \quad Y^{(1)}=\bigsqcup_{n:\text{odd}}Y_{n}. \]
Denote by $X^{(0)}$ and $ X^{(1)}$ their corresponding family of fiber spaces, i.e.
\[X^{(0)}=p^{-1}(Y^{(0)})\quad \text{and}\quad X^{(1)}=p^{-1}(Y^{(1)}),\]
then we have that $X=X^{(0)}\cup X^{(1)}$, and each of $X^{(0)}$, $X^{(1)}$ and $X^{(0)}\cap X^{(1)}$ is a coarse disjoint union of fiber spaces, i.e.
\[X^{(0)}=\bigsqcup_{n:\text{even}}p^{-1}(Y_{n});
\quad X^{(1)}=\bigsqcup_{n:\text{odd}}p^{-1}(Y_{n});\quad X^{(0)}\cap X^{(1)}=\bigsqcup_{n\in\mathbb{N}}(p^{-1}(Y_{n})\cap p^{-1}(Y_{n+1})).\]
Moreover, for any
$R > 0$, choose an integer $n_{R}>R$ and suppose that $y\in \text{Pen}(Y^{(0)}
;R)\cap \text{Pen}(Y^{(1)};R)$. It follows from \cite{CWY13} that there exists an interger $S=(n_{R}+1)^{3}$ such that $y\in\text{Pen}(Y^{(0)}\cap Y^{(1)};S)$. Therefore, the pair $(Y^{(0)},Y^{(1)})$ is ``$\omega$-excisive". Next, we consider the associated fiber spaces.


Take an arbitrary point $x\in p^{-1}(y) \subseteq \text{Pen}(X^{(0)};R)\cap \text{Pen}(X^{(1)};R) $, then \[y=p(x)\in \text{Pen}(Y^{(0)};\tilde{R})\cap \text{Pen}(Y^{(1)};\tilde{R}).\] Since the pair $(Y^{(0)},Y^{(1)})$ is $\omega$-excisive, we have that $y\in \text{Pen}(Y^{(0)}\cap Y^{(1)};S)$. That is, there exists $b\in Y^{(0)}\cap Y^{(1)}$ such that $d(y,b)\leqslant S$. Note that the collection of fiber spaces $\{p^{-1}(b)\}_{b\in Y^{(0)}\cap Y^{(1)}} $ is uniformly coarsely equivalent to the family of fiber spaces $\{p^{-1}(B_{Y}(b,S))\}_{b\in Y^{(0)}\cap Y^{(1)}}$ via the embedding maps \[\{f_{b}:p^{-1}(b)\rightarrow p^{-1}(B_{Y}(b,S))\}_{b\in Y^{(0)}\cap Y^{(1)}}.\] Hence, there exists a constant $C>0$ such that the image of $f_b$ is exactly a $C$-net of the family of fibers $p^{-1}(B_{Y}(b,S))$ for any $b\in Y$. As a result, we have $d(x, p^{-1}(b))<C$, where $p^{-1}(b)\subseteq X^{(0)}\cap X^{(1)}$. That is, $x\in \text{Pen}(X^{(0)}\cap X^{(1)};C)$, and
the proof is completed.
\end{proof}

\begin{proposition}\label{prop. 4.22} To prove Theorem \ref{the. 1.1}, it suffices to prove it for 
$X=\bigsqcup_{n\in\mathbb{N}} X_{n}$ being a coarse disjoint union of a sequence of metric spaces with uniform bounded geometry and uniform A-by-FCE coarse fibration structure, where the base spaces are finite.
\end{proposition}
\begin{proof}[\rm\textbf{Proof.}] 
With the notations and results presented in Proposition \ref{prop. 4.8} and its proof, the metric space $X$ admits an $\omega$-excisive decomposition $X=X^{(0)}\cup X^{(1)}$, where
\[X^{(0)}=\bigsqcup_{n:\text{even}}p^{-1}(Y_{n})\quad\text{and}\quad X^{(1)}=\bigsqcup_{n:\text{odd}}p^{-1}(Y_{n}).\]
Applying the Mayer-Vietoris argument in \cite{HRY93} we obtain the following six-term exact sequence 
\begin{center}

\begin{tikzcd}
{K_{0}(C^{*}_{\max}(X^{(0)}\cap X^{(1)}))} \arrow[r] & {{K_{0}(C^{*}_{\max}(X^{(0)}))\oplus {K_{0}(C^{*}_{\max}(X^{(1)}))}}} \arrow[r] & {{K_{0}(C^{*}_{\max}(X))}} \arrow[d] \\
{{K_{1}(C^{*}_{\max}(X))}} \arrow[u] & {{K_{1}(C^{*}_{\max}(X^{(0)}))\oplus {K_{1}(C^{*}_{\max}(X^{(1)}))}}} \arrow[l] & {{K_{1}(C^{*}_{\max}(X^{(0)}\cap X^{(1)})})} \arrow[l]
\end{tikzcd}
\end{center}
\normalsize
On the other hand, for each $d\geqslant 0$, there exists a sufficent large $l_{d}>0$ such that, for any fiber $x\in X\setminus B_{X}(p^{-1}(y_{0}),l_{d})$, the ball $B_{X}(x,d)$ is contained in either $X^{(0)}$ or $X^{(1)}$. Denote $K_{d}:=B_{X}(p^{-1}(y_{0}),l_{d})$. Then the Rips complex $P_{d}(X)$ can be partitioned as 
\[P_{d}(X)=P_{d}(K_{d}\cup X^{(0)})\cup P_{d}(K_{d}\cup X^{(1)}).\]
This is again an ``$\omega$-excisive" decomposition, into subspaces that are coarsely equivalent to
$X^{(0)}$ and $X^{(1)}$ respectively. 
Note that if the scale $d$ is not large enough, there exists a fiber space $p^{-1}(Y_{n})$ for some $n\in\mathbb{N}$ such that $P_{d}(X_{n}\cap X^{(0)})\cup P_{d}(X_{n}\cap X^{(1)})$ is not agree with $P_{d}(X_{n})$. We take $l_{d}>0$ sufficient large to get around this and obtain another exact Mayer-Vietoris sequence 

\begin{center}
\small
\begin{tikzcd}
K_{0}(P_{d}(K_{d}\cup X^{(0)})\cap P_{d}(K_{d}\cup X^{(1)})) \arrow[r] & K_{0}(P_{d}(K_{d}\cup X^{(0)}))\oplus K_{0}(P_{d}(K_{d}\cup X^{(1)}))\arrow[r] & {K_{0}(P_{d}(X))} \arrow[d] \\
{K_{1}(P_{d}(X))} \arrow[u] & K_{1}(P_{d}(K_{d}\cup X^{(0)}))\oplus K_{1}(P_{d}(K_{d}\cup X^{(1)})) \arrow[l] & K_{1}(P_{d}(K_{d}\cup X^{(0)})\cap P_{d}(K_{d}\cup X^{(1)})) \arrow[l]
\end{tikzcd}
\end{center}
\normalsize
Connecting these two Mayer-Vietoris 
sequences with assembly maps, and passing to the inductive limit as $d\rightarrow \infty$, the proposition holds.\end{proof}

Now we proceed to the second part mentioned at the beginning of this section, and introduce the definition of the maximal Roe algebra at infinity. Let
$(X_{n})_{n\in\mathbb{N}}$ be a sequence of metric spaces with uniform bounded geometry and uniform A-by-FCE coarse fibration structure, where the base spaces are finite.
For each $n\in\mathbb{N}$, let $P_{d}(X_{n})$ be the Rips complex of $X_{n}$ at scale $d\geqslant 0$ endowed with the spherical metric. Choose a countable dense subset $Z_{d}\subset P_{d}(X_{n})$ with $Z_{d}\subseteq Z_{d'}$ for $d\leqslant d'$. Denote $Z_{d,n}=Z_{d}\cap P_{d}(X_{n})$ for all $d\geqslant 0$ and $n\in \mathbb{N}$.
\begin{definition}\label{def. 4.1}
 For each $d\geqslant 0$, define the algebraic Roe algebra at infinity $\mathbb{C}_{u,\infty}[(P_{d}(X_{n}))_{n\in\mathbb{N}}]$ to be the set of all equivalence classes
$T = [(T^{(0)}
,\dots,T^{(n)},\dots)]$ of sequences $(T^{(0)}
,\dots,T^{(n)},\dots)$ described as follows

\begin{enumerate}
	\item [(1)]\hspace{0pt}$T^{(n)}$ is a bounded functions from $Z_{d,n}\times Z_{d,n}$ to $\mathcal{K}$ for any $n\in\mathbb{N}$ such that 
	\[\sup_{n\in\mathbb{N}} \sup_{x,x'\in Z_{d,n}}\|T^{(n)}(x,x')\|_{\mathcal{K}}<\infty; \]
	\item [(2)]\hspace{0pt}for each $n\in\mathbb{N}$ and any bounded subset $B\subset P_{d}(X_{n})$, the set
	\[\{(x,x')\in (B\times B)\cap (Z_{d,n}\times Z_{d,n})\mid T^{(n)}(x,x')\neq 0 \}\]
	is finite;
	\item [(3)]\hspace{0pt}there exists $L>0$ such that 
	\[\#\{x'\in Z_{d,n}\mid T^{(n)}(x,x')\neq 0\}<L \quad\text{and}\quad 
	\#\{x'\in Z_{d,n}\mid T^{(n)}(x',x)\neq 0\}<L\]
	for all $x\in Z_{d,n},n\in\mathbb{N};$
	\item [(4)]\hspace{0pt}there exists $R>0$ such that $T^{(n)}(x,x')=0$ whenever $d(x,x')>R$ for any $x,x'\in Z_{d,n},n\in\mathbb{N}.$
\end{enumerate}
\end{definition}
The equivalence relation $\thicksim$ on these sequences is defined by $(T^{(0)},\dots,T^{(n)},\dots)\thicksim(S^{(0)},\dots,S^{(n)},\dots)$ if and only if
$\lim_{n\rightarrow \infty}\sup_{x,x'\in Z_{d,n}}{\|T^{(n)}(x,x')-S^{(n)}(x,x')\|}_{\mathcal{K}}=0$.
Viewing each $T^{(n)}$ as $Z_{d,n}\times Z_{d,n}$ matrices, 
then
$\mathbb{C}_{u,\infty}[(P_{d}(X_{n}))_{n\in\mathbb{N}}]$ is made into a $*$-algebra by
using the usual matrix operations.
 The maximal Roe algebra at infinity, denoted by \[C^{*}_{u,\max,\infty}((P_{d}(X_{n}))_{n\in\mathbb{N}}),\] is defined to be the completion of $\mathbb{C}_{u,\infty}[(P_{d}(X_{n}))_{n\in\mathbb{N}}]$ with respect to
the maximal norm
\[\|T\|_{\max}:=\sup\{\|\phi(T)\|_{B(H_{\phi})}\mid\phi:\mathbb{C}_{u,\infty}[(P_{d}(X_{n}))_{n\in\mathbb{N}}]\rightarrow B(H_{\phi}), \text{a}\ *\text{-representation}\}.\]
This norm is well-defined since the sequence of metric spaces $(X_{n})_{n\in\mathbb{N}}$ has uniform bounded geometry. Moreover, the sequence $(X_{n})_{n\in\mathbb{N}}$ is uniformly coarsely equivalent to the sequence of associated Rips complexes $(P_{d}(X_{n}))_{n\in\mathbb{N}}$, so that
$C^{*}_{u,\max,\infty}((P_{d}(X_{n}))_{n\in\mathbb{N}})$ is isomorphic to $C^{*}_{u,\max,\infty}((X_{n})_{n\in\mathbb{N}})$ via a non-canonical isomorphism. 
\begin{definition}
Let $X=\bigsqcup_{n\in\mathbb{N}}X_{n}$ be the coarse disjoint union of a sequence of metric spaces with uniform bounded geometry. The geometric ideal of $\mathbb{C}[P_{d}(X)]$ generated by $X_{0}\in\{X_{n}\}_{n\in\mathbb{N}}$, denoted by $\mathbb{C}[P_{d}(X),X_{0}]$, is defined to be the set of all locally compact, finite propagation operators $T$ on geometric module $H_{X}$ whose support is contained in
$\text{Pen}(X_{0}; R_{T})\times \text{Pen}(X_{0}; R_{T})$ for some $R_{T} > 0$.
\end{definition}
We define $C^{*}_{\max}(P_{d}(X),X_{0})$ to be the completion of $\mathbb{C}[P_{d}(X),X_{0}]$ with respect to the norm induced from $C^{*}_{\max}(P_{d}(X))$, which is exactly the maximal norm of $\mathbb{C}[P_{d}(X),X_{0}]$ since $X_0$ has Property A. Note that $C^{*}_{\max}(P_{d}(X),X_{0})$ is a closed two-sided ideal in $C^{*}_{\max}(P_{d}(X))$, then using the universal property of the maximal norm, we have the following result. The reader
is referred to \cite{OOY09} for some relative discussion.

\begin{lemma}\label{lem. 4.2}
 Let $X=\bigsqcup_{n\in\mathbb{N}}X_{n}$ be the coarse disjoint union of a sequence of metric spaces
with uniform bounded geometry. For each $d\geqslant 0$, there is a short exact sequence
\[0\rightarrow C^{*}_{\max}(P_{d}(X),X_{0})\rightarrow C^{*}_{\max}(P_{d}(X))\rightarrow C^{*}_{u,\max,\infty}((P_{d}(X_{n}))_{n\in\mathbb{N}})\rightarrow 0.
\]
such that the inclusion $C^{*}_{\max}(P_{d}(X),X_{0})\hookrightarrow C^{*}_{\max}(P_{d}(X))$ induces an injection on $K$-theory.\qed
\end{lemma}
\begin{remark} \label{rem. 4.3}
Note that the above result is false if we consider the reduced Roe algebra at infinity, which is denoted by $C^{*}_{u,\infty}((P_{d}(X_{n}))_{n\in\mathbb{N}})$, and defined to be the completion of $\mathbb{C}_{u,\infty}[(P_{d}(X_{n}))_{n\in\mathbb{N}}]$ concerning 
the reduced norm (Cf. \cite{DGWY25})
\[\|T\|_{\text{red}}=\inf\{\|(T^{(n)}+S^{(n)})_{n\in\mathbb{N}}\|:(S^{(n)})_{n\in\mathbb{N}}\ \text{is a ghost operator}\}.\]
It follows that $C^{*}_{u,\infty}((P_{d}(X_{n}))_{n\in\mathbb{N}})$ is the quotient algebra of the Roe algebra over the ghost ideal (see \cite{Roe03}), which is contradictory with Lemma \ref{lem. 4.2}, and also the main reason for considering the maximal version of the coarse Baum-Connes conjecture in this study.
\end{remark}

For each $d\geqslant 0$, we recall that the Rips complex $P_{d}(X)$
at scale $d$ is the simplicial polyhedron in which the set of vertices is $X$, and a finite subset $\{x_{0},x_{1},\dots,x_{q}\}\subseteq X$ spans a simplex if and only if $d(x_{i},x_{j})\leqslant d$ for all
$0\leqslant i,j\leqslant q$. Thus, there exists a sufficiently large $N_{d}\in\mathbb{N}$ such that $d(X_{n},X_{m})>d$ provided $n,m\geqslant N_{d}$. We have the following direct sum decomposition
\[P_{d}(X)=P_{d}(\cup^{N_{d}-1}_{n=0}X_{n})\oplus \cup^{\infty}_{n=N_{d}}P_{d}(X_{n}),\]while at $K$-homology level
\[K_{*}(P_{d}(X))=K_{*}(P_{d}(\cup^{N_{d}-1}_{n=0}X_{n}))\oplus \prod^{\infty}_{n=N_{d}}K_{*}(P_{d}(X_{n})),\]
which canonically induces the following short exact sequence

\begin{center}
\begin{tikzcd}
{0} \arrow[r] & {{K_{*}(P_{d}(\cup^{N_{d}-1}_{n=0}X_{n}))\oplus \oplus^{\infty}_{n=N_{d}}K_{*}(P_{d}(X_{n}))}} \arrow[r] & {K_{*}(P_{d}(X))} \arrow[r] & {{\frac{\prod^{\infty}_{n=0}K_{*}(P_{d}(X_{n}))}{\oplus^{\infty}_{n=0}K_{*}(P_{d}(X_{n}))}}}\arrow[r] 
& {0} .
\end{tikzcd}
\end{center}
In addition, it follows from Definition \ref{def. 3.6} that the individual assembly map 
\[\mu_{\max}:K_{*}(P_{d}(X_{n}))\rightarrow K_{*}(C^{*}_{\max}(P_{d}(X_{n})))\]
for each $n\in\mathbb{N}$ can be defined by $\mu_{\max}([(H_{X},F)])=\partial ([(H_{X},F)])$ such that the propagation of $\partial ([(H_{X},F)])$ is
less than any given small $R > 0$ which is independent of $n\in\mathbb{N}$, and that $\|F\|\leqslant 1$. Consequently,
we obtain the following assembly map at infinity
\[\mu_{\max,\infty}:{\frac{\prod^{\infty}_{n=0}K_{*}(P_{d}(X_{n}))}{\oplus^{\infty}_{n=0}K_{*}(P_{d}(X_{n}))}}\rightarrow {K_{*}(C^{*}_{u,\max,\infty}((P_{d}(X_{n}))_{n\in\mathbb{N}}))} \]
defined by the formula \[\mu_{\max,\infty}[([(H_{P_{d}(X_{0})},F^{(0)})],[(H_{P_{d}(X_{1})},F^{(1)})],\dots)]=[(\partial[(H_{P_{d}(X_{0})},F^{(0)})],\partial[(H_{P_{d}(X_{1})},F^{(1)})],\dots)].\]
Combining Definition \ref{def. 3.6} and Lemma \ref{lem. 4.2}, we now have the following commutative diagram.

\begin{center}
	\begin{tikzcd}
	{0} \arrow[d]            &  & {0} \arrow[d] \\
	{K_{*}(P_{d}(\cup^{N_{d}-1}_{n=0}X_{n}))\oplus \oplus^{\infty}_{n=N_{d}}K_{*}(P_{d}(X_{n}))} \arrow[d] \arrow[rr] &  & {K_{*}(C^{*}_{\max}(P_{d}(X),X_{0}))} \arrow[d] \\
	K_{*}(P_{d}(X)) \arrow[d] \arrow[rr,"\mu_{\max}"] &  & {K_{*}(C^{*}_{\max}(P_{d}(X)))} \arrow[d] \\
	{\frac{\prod^{\infty}_{n=0}K_{*}(P_{d}(X_{n}))}{\oplus^{\infty}_{n=0}K_{*}(P_{d}(X_{n}))}} \arrow[d] \arrow[rr,"\mu_{\max,\infty}"] &  & {K_{*}(C^{*}_{u,\max,\infty}((P_{d}(X_{n}))_{n\in\mathbb{N}}))} \arrow[d] \\
	{0}                      &  & {0}   \end{tikzcd}
\end{center}
\begin{remark}
Passing to the inductive limit as $d$ tends to $\infty$, the top horizontal map is an isomorphism. Actually, any element in the direct sum, as a finite sequence, is supported on a sum below some fixed $k$ and, as $d$ tends to $\infty$, will eventually be absorbed into the first term on a single simplex. That is,
\[\lim_{d\rightarrow \infty}{K_{*}(P_{d}(\cup^{N_{d}-1}_{n=0}X_{n}))\oplus \oplus^{\infty}_{n=N_{d}}K_{*}(P_{d}(X_{n}))}=\lim_{d\rightarrow \infty}\lim_{k\rightarrow \infty}{K_{*}(P_{d}(\cup^{k}_{n=0}X_{n}))}.\]
Since $X_{0}$ is coarsely equivalent to $\cup^{k}_{n=0}X_{n}$ for some $k\in\mathbb{N}$, we have that
\[K_{*}(C^{*}_{\max}(P_{d}(X),X_{0}))=
\lim_{k\rightarrow\infty}K_{*}(C^{*}_{\max}(\cup^{k}_{n=0}X_{n})).\]
The top horizontal map is precisely equivalent to the maximal assembly map for the finite union $\cup^{k}_{n=0}X_{n}$ as $d$ goes to infinity. By conditions (3) and (4) of Definition \ref{def. 2.4} and the assumption that base spaces are finite, $\cup^{k}_{n=0}X_{n}$ for any $k\in\mathbb{N}$ admits Property A, we conclude that the top horizontal map is an isomorphism as $d$ tends to $\infty$( Cf.\cite{Yu00}).
Hence, to prove the maximal coarse Baum-Connes conjecture holds for $X$, that is, the assembly map $\mu_{\max}$ is an isomorphism, it suffices to prove that the assembly map at infinity $\mu_{\max,\infty}$ is an isomorphism by the five lemma. 
\end{remark}

\noindent\textbf{The maximal coarse Baum-Connes conjecture at infinity.} \textit{Let $X=\bigsqcup_{n\in\mathbb{N}}X_{n}$ be the coarse disjoint union of a sequence of metric spaces with uniform bounded geometry. Then the assembly map at infinity $\mu_{\max,\infty}$ is an isomorphism.}

\begin{definition}\label{def. 4.4}
 Let $d\geqslant 0$. Define the algebraic localization algebra at infinity 
$\mathbb{C}_{u, L, \infty}
[(P_{d}(X_{n}))_{n\in\mathbb{N}}]$ to be the
$\ast$-algebra of all bounded and uniformly norm-continuous functions
\[f:\mathbb{R}_{+}\rightarrow\mathbb{C}_{u, \infty}
[(P_{d}(X_{n}))_{n\in\mathbb{N}}]\]
such that $f(t)$ is of the form $[({f}^{(0)}(t),\dots,{f}^{(n)}(t),\dots)]$ for all $t\in\mathbb{R}_{+}$, where the family of functions $({f}^{(n)}(t))_{n\in\mathbb{N}, t\geqslant 0}$ satisfies the conditions in Definition \ref{def. 4.1} with uniform constants,
and there exists a bounded function $R(t):\mathbb{R}_{+}\rightarrow\mathbb{R}_{+}$ with $\lim_{t\rightarrow\infty} R(t)=0$ such that
$({f}^{(n)}(t))(x,x')=0$ whenever $d(x,x')>R(t)$
for all $x,x'\in Z_{d,n},n\in\mathbb{N}, t\in\mathbb{R}_{+}$.
\end{definition}

\begin{definition}\label{def. 4.5}
Let $d\geqslant 0$. The maximal localization algebra at infinity, denoted by $C^{*}_{u,L,\max,\infty}((P_{d}(X_{n}))_{n\in\mathbb{N}})$, is defined to be the completion of $\mathbb{C}_{u, L, \infty}
[(P_{d}(X_{n}))_{n\in\mathbb{N}}]$ with respect to the norm 
\[\|f\|_{\max}:=\sup_{t\in\mathbb{R}_{+}}\|f(t)\|_{\max}.\]
\end{definition}
With a similar proof of the arguments in \cite{Yu97}, the local assembly map at infinity 
\[\mu_{L,\max,\infty}:\lim_{d\rightarrow\infty}\frac{\prod^{\infty}_{n=0}K_{*}(P_{d}(X_{n}))}{\bigoplus^{\infty}_{n=0}K_{*}(P_{d}(X_{n}))}\rightarrow \lim_{d\rightarrow\infty}K_{*}(C^{*}_{u,L,	\max,\infty}((P_{d}(X_{n}))_{n\in\mathbb{N}}))
\]is an isomorphism. 
We can also define a natural evaluation-at-zero homomorphism
\[e:C^{*}_{u, L, \max,\infty}((P_{d}(X_{n}))_{n\in\mathbb{N}})\rightarrow C^{*}_{u, \max,\infty}((P_{d}(X_{n}))_{n\in\mathbb{N}})\]
by $e(f)=f(0)$, and obtain the following commutative diagram

\begin{center}
	\begin{tikzcd}
			&  & {\lim_{d\rightarrow\infty}K_{*}(C^{*}_{u,L,\max,\infty}((P_{d}(X_{n}))_{n\in\mathbb{N}}))} \arrow[d,"e_{*}"] \\
		{\lim_{d\rightarrow\infty}\frac{\prod^{\infty}_{n=0}K_{*}(P_{d}(X_{n}))}{\bigoplus^{\infty}_{n=0}K_{*}(P_{d}(X_{n}))}} \arrow[rr,"\mu_{\max,\infty}"] \arrow[rru,"\mu_{L,\max,\infty}"] &  & {\lim_{d\rightarrow\infty}K_{*}(C^{*}_{u,\max,\infty}((P_{d}(X_{n}))_{n\in\mathbb{N}}))} .         
	\end{tikzcd}
		\end{center}
As a result, proving the assembly map $\mu_{\max,\infty}$ is isomorphism can be reduced to verify the map
\begin{equation}\label{equ. 1}
e_{*}:\lim_{d\rightarrow\infty}K_{*}(C^{*}_{u, L, \max,\infty}((P_{d}(X_{n}))_{n\in\mathbb{N}}))\rightarrow \lim_{d\rightarrow\infty}K_{*}(C^{*}_{u \max,\infty}((P_{d}(X_{n}))_{n\in\mathbb{N}}))
\end{equation}
induced by the evaluation-at-zero map on $K$-theory is an isomorphism, which is restated as follows.

\begin{theorem}\label{the. 4.6}
Let $(X_{n})_{n\in\mathbb{N}}$ be a sequence of metric spaces with uniform bounded geometry. If the coarse disjoint union $X=\bigsqcup_{n\in\mathbb{N}}X_{n}$ admits an A-by-FCE coarse fibration structure, then the evaluation homomorphism (\ref{equ. 1}) is an isomorphism.
\end{theorem}

We will devote the second half of this paper, specifically Sections \ref{section 5}, \ref{section 6}, and \ref{section 7}, to proving this result. With Theorem \ref{the. 4.6} established, Theorem \ref{the. 1.1} follows directly from Proposition \ref{prop. 4.22}.

\section{Maximal twisted Roe algebras at infinity}\label{section 5}

In this section, we shall introduce the maximal twisted Roe algebra at infinity $C^{*}_{u,\max,\infty}((P_{d}(X_{n}),\mathcal{A}(V_{n}))_{n\in\mathbb{N}})$ 
 and their localized counterpart $C^{*}_{u,L,\max,\infty}((P_{d}(X_{n}),\mathcal{A}(V_{n}))_{n\in\mathbb{N}})$ for a sequence of metric spaces $(X_{n})_{n\in\mathbb{N}}$ with uniform bounded geometry such that the coarse disjoint $X=\bigsqcup_{n\in\mathbb{N}}X_{n}$ admits an A-by-FCE coarse fibration structure. Subsequently, a Bott map $\beta$ will be constructed from the $K$-theory of the maximal Roe algebra at infinity to the $K$-theory of the maximal twisted Roe algebra at infinity, along with another Bott map $\beta_{L}$ between the $K$-theory of the corresponding localization algebras. These constructions lead to the following commutative diagram
 
\begin{center}
\begin{tikzcd}
\lim_{d\rightarrow \infty}{K_{*}(C^{*}_{u, L, \max,\infty}((P_{d}(X_{n}))_{n\in\mathbb{N}}))} \arrow[rr,"e_{*}"] \arrow[d,"(\beta_{L})_{*}"] &  & \lim_{d\rightarrow \infty}{K_{*}(C^{*}_{u, \max,\infty}((P_{d}(X_{n}))_{n\in\mathbb{N}}))} \arrow[d,"\beta_{*}"] \\
\lim_{d\rightarrow \infty}{K_{*}(C^{*}_{u, L, \max,\infty}((P_{d}(X_{n}),\mathcal{A}(V_{n}))_{n\in\mathbb{N}}))} \arrow[rr,"e_{*}^{\mathcal{A}}"]           &  & \lim_{d\rightarrow \infty}{K_{*}(C^{*}_{u, \max,\infty}((P_{d}(X_{n}),\mathcal{A}(V_{n}))_{n\in\mathbb{N}}))}  ,      
\end{tikzcd}
\end{center}
where $e_{*}^{\mathcal{A}}$ is the $*$-homomorphism induced by the evaluation-at-zero map from the maximal twisted localization algebra at infinity to the maximal twisted Roe algebra at infinity. We shall decompose these twisted algebras into various smaller ideals to establish an isomorphism of $e_{*}^{\mathcal{A}}$ in Section \ref{section 6}. Finally, in Section \ref{section 7}, we will develop the Dirac maps $\alpha$ and $\alpha_{L}$ and prove a geometric analogue of the Bott periodicity in finite dimensions. Consequently, the top evaluation map $e_{*}$ is an isomorphism, as stated in Theorem \ref{the. 4.6}. The concept behind these constructions can be traced back to \cite{CWY13}\cite{Yu00}.

\subsection{Preliminary}
\quad Let $H$ be a separable infinite-dimensional Hilbert space. Denote by $V_{a}, V_{b}$ the finite-dimensional affine subspaces of $H$. Let $V^{0}_{a}$ be the linear subspace of $H$ consisting of differences of elements of $V_{a}$. Let $\text{Cliff}(V^{0}_{a})$ be the complexified Clifford algebra of $V^{0}_{a}$ and $\mathcal{C}(V_{a})$ the graded $C^{*}$-algebra of continuous functions vanishing at infinity from $V_{a}$ into $\text{Cliff}(V^{0}_{a})$. Let $\mathcal{S}$ be the $C^{*}$-algebra of all continuous functions on $\mathbb{R}$ vanishing at infinity. Then $\mathcal{S}$ is graded according to the even and odd functions. Define the graded tensor product
\[\mathcal{A}(V_{a})=\mathcal{S}\hat\otimes\mathcal{C}(V_{a})
\]
If $V_{a}\subseteq V_{b}$, then we have a decomposition $V_{b}=V^{0}_{ba}\oplus V_{a}$, where $V^{0}_{ba}$ is the orthogonal complement of $V^{0}_{a}$ in $V^{0}_{b}$. For each $v_{b}\in V_{b}$, we have a corresponding decomposition $v_{b}=v_{ba}+v_{a}$,
where $v_{ba}\in V_{ba}^{0}$ and $v_{a}\in V_{a}$. Every function $h$ on $V_{a}$ can be extended to a function $\tilde{h}$ on $V_{b}$ by
the formula $\tilde{{h}}(v_{ba}+v_{a})=h(v_{a})$.
\begin{definition}\label{def. 5.1}
For affine subspaces $V_{a}\subseteq V_{b}$, denote by $C_{ba}:V_{b}\rightarrow\text{Cliff}
(V^{0}_{b})$ the function
$v_{b}\mapsto v_{ba}\in\text{Cliff}(V^{0}_{b})$, where $v_{ba}$ is considered as an element of 
$\text{Cliff}(V^{0}_{b})$ via the inclusion
$V^{0}_{ba}\subset \text{Cliff}(V^{0}_{b})$. Let $X$ be the unbounded multiplier of $\mathcal{S}$ given by the function $x\mapsto x$. Define a
$*$-homomorphism 
\[\beta_{V_{b},V_{a}}:\mathcal{A}(V_{a})\rightarrow \mathcal{A}(V_{b}),\] or simply denoted by $\beta_{ba}$, by the formula
\[\beta_{ba}(g\hat\otimes h)=g(X\hat\otimes 1+1\hat\otimes C_{ba})(1\hat\otimes \tilde{h})\]
for all $g \in\mathcal{S}$, $h\in\mathcal{C}(V_{a})$, where $g(X\hat\otimes 1 +1 \hat\otimes C_{ba})$ is defined by functional calculus of $g$ on the unbounded, essentially self-adjoint operator $X\hat\otimes 1+1\hat\otimes C_{ba}$.
\end{definition} 

\begin{remark}\label{rem. 5.2} Some explanations for the above formula are given here. Since the $*$-homomorphism $\beta_{ba}$ is a tensor product with the identity on $\mathcal{C}(V_{a})$, we shall ignore the factor $\mathcal{C}(V_{a})$ in what follows and thus it suffices to
consider the case where $g$ is one of the generators 
$\exp(-x^{2})$ or $x\exp(-x^{2})$ of the $C^{*}$-algebra $\mathcal{S}$. On the first generator,
\[g(X\hat\otimes 1+1\hat\otimes C_{ba})=\exp(-x^{2})\hat\otimes \exp(-C_{ba}^{2}).\]
It is a continuous function with domain $\mathbb{R}_{+}\times V_{ba}^{0}$, for any pair $(t,v)\in\mathbb{R}_{+}\times V_{ba}^{0}$, we calculate that 
\[(\exp(-x^{2})\hat\otimes \exp(-C_{ba}^{2}))(t,v)=\exp(-t^{2})\hat\otimes \exp(-\|v_{ba}\|^{2})=\exp\Big(-\Big(\sqrt{t^{2}+\|v_{ba}\|^{2}}\Big)^{2}\Big),\]
which happens to be the value of function $g$ at point $t^{2}+\|v_{ba}\|^{2}$. One can verify that this result also applies to the generator $x\exp(-x^{2})$. Consequently, the function $g$ and subsequently, the $*$-homomorphism $\beta_{ba}$ can be interpreted as a rotation around the $V_{a}$-space, as illustrated in Figure \ref{fig.1}. Through this rotation, functions defined on the lower dimensional affine subspace $\mathbb{R_{+}}\times V_{a}$ are elevated to the higher dimensional affine subspace $\mathbb{R_{+}}\times V_{b}$. Continuing this iterative process, we eventually ascend to the infinite-dimensional space $\mathbb{R}_{+}\times H$.

From this perspective, the metric topology of space $\mathbb{R}_{+}\times H$ is no longer applicable. We instead endow with it a topology under which $\mathbb{R}_{+}\times H$ is a locally
compact topological space in such a way that a net $\{(t_{i},v_{i})\}$ in $\mathbb{R}_{+}\times H$ converges to
a point $(t,v)\in\mathbb{R}_{+}\times H$ if and only if
\begin{enumerate}
    \item [(1)] \hspace{0pt}$t_{i}^{2}+\|v_{i}\|^{2}\rightarrow\ t^{2}+\|v\|^{2}$, as $i\rightarrow\infty$;
    \item [(2)] \hspace{0pt}$\langle v_{i},u\rangle\rightarrow\langle v,u\rangle$ for any $u\in H$, as $i\rightarrow\infty$.
\end{enumerate}
Note that for each $v\in H$ and each $r>0$,
$B(v,r)=\{(t,w)\in\mathbb{R}_{+}\times H\mid t^{2}+\|v-w\|^{2}<r^{2}\}$ is an open subset of $\mathbb{R}_{+}\times H$. For finite-dimensional subspaces $V_{a}\subseteq V_{b}\subseteq H$, since $\beta_{ba}$ takes $C_{0}(\mathbb{R}_{+}\times V_{a})$ into $C_{0}(\mathbb{R}_{+}\times V_{b})$, the $C^{*}$-algebra $\lim\limits_{\rightarrow}C_{0}(\mathbb{R}_{+}\times V_{a})$ is $*$-isomorphic to $C_{0}(\mathbb{R}_{+}\times H)$, where the direct limit is over the directed set of all finite-dimensional affine subspaces $V_{a}\subseteq V$.
\end{remark}

\begin{figure}[H]
    \centering
    \includegraphics[width=6.7in]{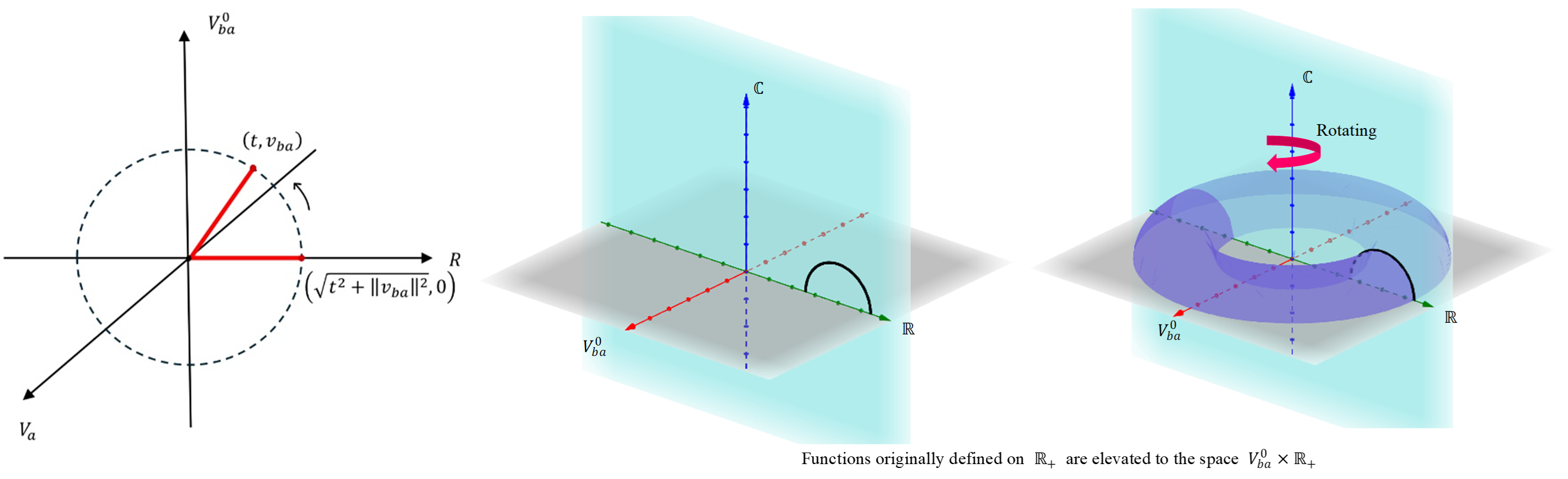}
    \caption{Diagram of the rotation process.}
    \label{fig.1}
\end{figure}

\begin{definition}\label{def. 5.3}
    If $V_{a}\subseteq V_{b}$, for any subset 
    $O\subset \mathbb{R}_{+}\times V_{a}$, define
    \[\overline{O}^{\beta_{ba}}=\{(t,v_{ba}+v_{a})\in\mathbb{R}_{+}\times V_{b}\mid (\sqrt{t^{2}+\|v_{ba}\|^{2}},v_{a})\in O\}.\]
\end{definition}
For any finite dimensional affine subspace $V_{a}$ of $H$, the algebra $C_{0}(\mathbb{R}_{+}\times V_{a})$ is included in
$\mathcal{A}(V_{a})$ as its center. For any function $a\in\mathcal{A}(V_{a})$, the support of $a$, denoted by $\text{supp}(a)$, is the
complement of all points $(t,v)\in \mathbb{R}_{+}\times V_{a}$ such that there exists 
$g\in C_{0}(\mathbb{R}_{+}\times V_{a})$ such that
$g(t,v)\neq0 $ but $g\cdot a=0$. Note that if $V_{a}\subseteq V_{b}$ and $a\in\mathcal{A}(V_{a})$, then \[\text{Supp}(\beta_{ba}(a))=\overline{\text{Supp}(a)}^{\beta_{ba}}.\]

The following concept pertains to the compatibility between $*$-homomorphisms defined above and isometries. This compatibility will be instrumental in establishing the product structure of the maximal twisted Roe algebra at infinity in subsequent discussions.

\begin{definition}\label{def. 5.4}
Let $W_{a},W_{b},V_{a},V_{b},V_{c}$ be finite dimensional affine subspaces of $H$ with $W_{a}\subseteq
W_{b}\subseteq V_{c}$ and $V_{a}\subseteq V_{b}\subseteq V_{c}$. Let $t$ be an affine isometry from $W_{b}$ onto $V_{b}$ mapping $W_{a}$ onto 
$V_{a}$. There exists a canonical $*$-isomorphism $t_{*}:\mathcal{C}(W_{a})\rightarrow \mathcal{C}(V_{a})$ defined by $(t_{*}(h))(w)=h(t^{-1}(w))$ for any $h\in\mathcal{C}(W_{a})$ and $w\in V_{a}$. Subsequently, we obtain a $*$-isomorphism $1\hat\otimes t_{*}$ from $\mathcal{A}(W_{a})$ onto $\mathcal{A}(V_{a})$. Note that $\mathcal{A}(V_{a})$ is included in $\mathcal{A}(V_{c})$ via $\beta_{ca}$, without causing confusion, we still denote by $t_{*}$ the composition of $*$-isomorphism $1\hat\otimes t_{*}$ and $*$-homomorphism $\beta_{ca}$:
\begin{center}
\begin{tikzcd}
{t_{*}:\mathcal{A}(W_{a})} \arrow[r,"\cong"] & {\mathcal{A}(V_{a})} \arrow[r,"\beta_{ca}"] & {\mathcal{A}(V_{c})}.
\end{tikzcd}
\end{center}
Consequently, we have the following commutative diagram.
\begin{center}
\begin{tikzcd}
{\mathcal{A}(W_{a})} \arrow[r,"\cong"] \arrow[d,"\beta_{ba}"] & {\mathcal{A}(V_{a})} \arrow[r,"\beta_{ca}"] \arrow[d,"\beta_{ba}"] & {\mathcal{A}(V_{c})} \arrow[d,"="] \\
{\mathcal{A}(W_{b})} \arrow[r,"\cong"]           & {\mathcal{A}(V_{b})} \arrow[r,"\beta_{cb}"]           & {\mathcal{A}(V_{c})}      
\end{tikzcd}
\end{center}
When $t$ is not bijective, the composition is adjusted as
\begin{center}
\begin{tikzcd}
{t_{*}:\mathcal{A}(W_{a})} \arrow[r,"t_{*}"] & {\mathcal{A}(t(W_{a}))}\arrow[r,"\beta_{V_{a},t(W_{a})}"] &\mathcal{A}(V_{a})\arrow[r,"\beta_{ca}"] & {\mathcal{A}(V_{c})},
\end{tikzcd}
\end{center}
which is the form utilized later in discussing the product structure for maximal twisted Roe algebras at infinity.
\end{definition}

\subsection{Maximal twisted Roe algebras at infinity}
Let $(X_{n})_{n\in\mathbb{N}}$ be a sequence of metric spaces with uniform bounded geometry that admits an A-by-FCE coarse fibration structure. Then there exist a family of discrete metric spaces $(Y_{n})_{n\in\mathbb{N}}$ with uniform bounded geometry, and for each $n\in\mathbb{N}$ there exists a surjective map $p_{n}:X_{n}\rightarrow Y_{n}$ satisfying conditions (1)-(4) in Definition \ref{def. 2.4}. 
In this subsection, we shall define the maximal twisted Roe algebras at infinity and their associated localization algebras for the sequence $(X_{n})_{n\in\mathbb{N}}$ by using the fibred coarse embeddability of the sequence $(Y_{n})_{n\in\mathbb{N}}$. \par Recall that the sequence of finite metric spaces $(Y_{n})_{n\in\mathbb{N}}$ with uniform bounded geometry is said to admit a fibred coarse embedding into Hilbert space $H$ if there exist\par 
$\bullet$ a field of Hilbert spaces $(H_{y})_{y\in Y_{n},n\in\mathbb{N}}$;\par
$\bullet$ a section $s:Y_{n}\rightarrow \bigsqcup_{y\in Y_{n}}H_{y}$ for all $n\in\mathbb{N}$;\par 
$\bullet$ two non-decreasing functions $\rho_{1},\rho_{2}$ from $[0,\infty)$ to  $[0,\infty)$ with $\lim\limits_{r\rightarrow\infty}\rho_{i}(r)=\infty, i=1,2$;\par   
$\bullet$ a non-decreasing sequence of numbers $0\leqslant l_{0}\leqslant l_{1}\leqslant ...\leqslant l_{n}\leqslant ...$ with $\lim\limits_{n\rightarrow\infty}l_{n}=\infty$,\\
such that for each $y\in Y_{n},n\in\mathbb{N}$, there exists a trivialization
\[t_{y}:(H_{z})_{z\in B_{Y_{n}}(y,l_{n})}\rightarrow 
B_{Y_{n}}(y,l_{n})\times H\]
such that the restriction of $t_{y}$ to the fiber 
$H_{z},z\in B_{Y_{n}}(y,l_{n}),$ is an affine isometry $t_{y}(z):H_{z}\rightarrow H$ satisfying
\begin{enumerate}
   \item [(1)] \hspace{0pt} $\rho_{1}(d(z,z'))\leqslant\|t_{y}(z)(s(z))-t_{y}(z')(s(z'))\|\leqslant \rho_{2}(d(z,z'))$ for any $z,z'\in B_{Y_{n}}(y,l_{n}),y\in Y_{n},n\in\mathbb{N};$
    \item [(2)] \hspace{0pt} for any $y,y'\in Y_{n}$ with $B_{Y_{n}}(y,l_{n})\cap B_{Y_{n}}(y',l_{n})\neq\emptyset$, there exists an affine isometry $t_{yy'}: H\rightarrow H$ such that $t_{y}(z)\circ t^{-1}_{y'}(z)=t_{yy'}$ for all $z\in B_{Y_{n}}(y,l_{n})\cap B_{Y_{n}}(y',l_{n})$.
\end{enumerate}

 For any $d\geqslant 0$ and $n\in\mathbb{N}$, let $P_{d}(X_{n})$ and $P_{d}(Y_{n})$ be the Rips complexes of $X_{n}$ and $Y_{n}$ at scale $d$ endowed with the spherical metric, respectively. For each $y\in P_{d}(Y_{n})$, denote by Star$(y)$ the open star of $y$ in the barycentric subdivision of $P_{d}(Y_{n})$. The surjective map $p_{n}:X_{n}\rightarrow Y_{n}$ induces $p_{n}:P_{d}(X_{n})\rightarrow P_{d}(Y_{n})$ by the formula
\[p_{n}(\sum^{k}_{i=0}c_{i}x_{i})=\sum^{k}_{i=0}c_{i}p_{n}(x_{i}),\]
where $c_{i}\geqslant 0$ and $\sum^{k}_{i=0}c_{i}=1$. 
Take a countable dense subset 
$F_{d,n}\subset P_{d}(Y_{n})$ for each $d\geqslant{0}$ and $n\in\mathbb{N}$ in such a
way that
\[(1)\ F_{d,n}\subset\bigsqcup_{y\in Y_{n}} \text{Star}(y);\quad\quad (2)\ F_{d,n}\subset F_{d',n}\ \text{when}\ d<d'.\]
Then we can choose the corresponding countable dense subset $Z_{d,n}$ of $P_{d}(X_{n})$ such that $p_{n}(Z_{d,n})=F_{d,n}$ so that $Z_{d,n}\subset Z_{d',n}$ whenever $d<d'$ for all $n\in\mathbb{N}$. 
Moreover, 
for any $y\in F_{d,n}$, there exists a unique point $\bar{y}\in Y_{n}$ such that $y\in$ Star$(\bar{y})$. We define 
\[H_{y}=H_{\bar{y}}\quad \text{and}\quad s(y)=s(\bar{y}),\]
for all $y\in F_{d,n}\cap \text{Star}(\bar{y})$.
The trivialization of the single point $y$ will  represent that of open star Star$(y)$ of $y$, i.e.
$t_{y}(z)=t_{\bar{y}}(\bar{z})$ for all $y,z\in F_{d,n},n\in\mathbb{N}$ with $\bar{z}\in B(\bar{y},l_{n})$. 

For any $n\in\mathbb{N}$, define $V_{n}$ to be the finite dimensional affine subspace of $H$ spanned by
$t_{y}(z)(s(z))$ for all $z\in B_{Y_{n}}(y,l_{n})$ with $y\in Y_{n}$, i.e.
\[V_{n}:=\text{affine-span}\{t_{y}(z)(s(z))\mid z\in B_{Y_{n}}(y,l_{n}), y\in Y_{n} \}.\]
For any $y\in Y_{n}$ and $k\geqslant 0$, define
\[W_{k}(y):=\text{affine-span}\{t_{y}(z)(s(z))\mid
z\in F_{d,n}\cap
B_{P_{d}(Y_{n})}(y,k)\}.\]
Note that for each $k\geqslant 0$, there exists $N\in\mathbb{N}$ such that $W_{k}(y)$ is well defined for $y\in F_{d,n}$ with $n\geqslant N$, and is an affine subspaces of $V_{n}$. By Definition \ref{def. 5.1}, the inclusion $W_{k}(y)\subseteq V_{n}$ induces the following $*$-homomorphism
\[\beta_{V_{n},W_{k}(y)}:\mathcal{A}(W_{k}(y))\rightarrow \mathcal{A}({V}_{n}).\]

\begin{definition}\label{def. 5.5}
 For any $d\geqslant 0$, the algebraic twisted Roe algebra at infinity
$\mathbb{C}_{u,\infty}
[(P_{d}(X_{n}),\mathcal{A}(V_{n}))_{n\in\mathbb{N}}]$ is defined to be the
set of all equivalence classes $T=[(T^{(0)},\dots,T^{(n)},\dots)]$ of sequences $(T^{(0)},\dots,T^{(n)},\dots)$ satisfying
\begin{enumerate}
   \item [(1)] \hspace{0pt}for each $n\in\mathbb{N}$, 
$T^{(n)}$ is a bounded function from $Z_{d,n}\times Z_{d,n}$ to $\mathcal{A}(V_{n})\hat{\otimes}\mathcal{K}$ such that 
\[\sup_{n\in\mathbb{N}} \sup_{x,x'\in Z_{d,n}}\|T^{(n)}(x,x')\|_{\mathcal{A}(V_{n})\hat{\otimes}\mathcal{K}}<\infty; \]
\item [(2)] \hspace{0pt}for each $n\in\mathbb{N}$ and any bounded subset $B\subset P_{d}(X_{n})$, the set
\[\{(x,x')\in (B\times B)\cap (Z_{d,n}\times Z_{d,n})\mid T^{(n)}(x,x')\neq 0 \}\]
is finite;
\item [(3)] \hspace{0pt}there exists $L>0$ such that 
\[\#\{x'\in Z_{d,n}\mid T^{(n)}(x,x')\neq 0\}<L \quad\text{and}\quad 
\#\{x'\in Z_{d,n}\mid T^{(n)}(x',x)\neq 0\}<L\]
for all $x\in Z_{d,n},n\in\mathbb{N};$
\item [(4)] \hspace{0pt}there exists $R>0$ such that $T^{(n)}(x,x')=0$ whenever $d(x,x')>R$ for any $x,x'\in Z_{d,n},n\in\mathbb{N}$ (the least such $R$ is called the propagation of the sequence $(T^{(0)},\dots,T^{(n)},\dots)$);
\item [(5)] \hspace{0pt}there exists $r>0$ such that 
\begin{align*}
	\text{Supp}(T^{(n)}(x,x'))
	&\subseteq B_{{\mathbb{R}}_{+}\times V_{n}}(t_{p_{n}(x)}(p_{n}(x))(s(p_{n}(x))),r)\\
	&=\{(\tau,v)\in{{\mathbb{R}}_{+}\times V_{n}}\mid \tau^{2}+\|v-t_{p_{n}(x)}(p_{n}(x))(s(p_{n}(x)))\|^{2}<r^{2}\}
\end{align*}
for all $x,x'\in Z_{d,n}, n\in\mathbb{N}$;
\item [(6)] \hspace{0pt}there exist $k>0$ and $i_{0}>0$ depending only on the sequence $(T^{(0)},\dots,T^{(n)},\dots)$(not on $n$) 
such that for each $x,x'\in Z_{d,n}$, there exists
\[{T_{1}}^{(n)}(x,x')\in\mathcal{A}(W_{k}(p_{n}(x)))\hat{\otimes}\mathcal{K}\cong\mathcal{S}\hat{\otimes}\mathcal{C}(W_{k}(p_{n}(x)))\hat{\otimes}\mathcal{K}
\]
of the form $\sum_{i=1}^{i_{0}} g_{i}\hat{\otimes}h_{i}\hat{\otimes}{k_{i}}$, where $g_{i}\in\mathcal{S}$, $h_{i}\in\mathcal{C}(W_{k}(p_{n}(x)))$, $k_{i}\in\mathcal{K}$, for $1\leqslant i\leqslant i_{0}$ such that 
\[T^{(n)}(x,x')=(\beta_{V_{n},W_{k}(p_{n}(x))}\hat{\otimes}1)({T_{1}}^{(n)}(x,x'));\]
\item [(7)] \hspace{0pt}there exists $c>0$ such that if ${T_{1}}^{(n)}(x,x')\in\mathcal{A}(W_{k}(p_{n}(x)))\hat{\otimes}\mathcal{K}$
as above and $w\in{\mathbb{R}}_{+}\times {W_{k}}^{0}(p_{n}(x))$
with $\|w\|\leqslant 1$, then $\bigtriangledown_{w}({T_{1}}^{(n)}(x,x'))$, the derivative of ${T_{1}}^{(n)}(x,x')$ in the direction $w$, exists in 
$\mathcal{A}(W_{k}(p_{n}(x)))\hat{\otimes}\mathcal{K}$ and is of norm at most $c$ for all $x,x'\in Z _{d,n},n\in\mathbb{N}$.
\end{enumerate}
\end{definition}

The equivalence relation $\thicksim$ on these sequences is defined by 
\[(T^{(0)},\dots,T^{(n)},\dots)\thicksim(S^{(0)},\dots,S^{(n)},\dots)\] if and only if
\[\lim_{n\rightarrow \infty}\sup_{x,x'\in Z_{d,n}}{\|T^{(n)}(x,x')-S^{(n)}(x,x')\|}_{\mathcal{A}(V_{n})\hat{\otimes}\mathcal{K}}=0.\]

The product of two arbitrary elements ${[(T^{(0)},\dots,T^{(n)},\dots)]}$ and $[(S^{(0)},\dots,S^{(n)},\dots)]$ in the algebraic twisted Roe algebra at infinity
$\mathbb{C}_{u,\infty}
[(P_{d}(X_{n}),\mathcal{A}(V_{n}))_{n\in\mathbb{N}}]$ is defined to be
\[[(T^{(0)},\dots,T^{(n)},\dots)][(S^{(0)},\dots,S^{(n)},\dots)]=[((TS)^{(0)},\dots,
(TS)^{(n)},\dots)],\]
where there exists a sufficiently large $N\in\mathbb{N}$ depending on the propagation of the two sequences satisfying that $(TS)^{(n)}=0$ for all $n<N$, while
\[(TS)^{(n)}(x,x')=\sum_{z\in Z_{d,n}}
(T^{(n)}(x,z))\cdot ((t_{p_{n}(x)p_{n}(z)})_{*}(S^{(n)}(z,x')))\]
for all $x,x'\in Z_{d,n}, n\geqslant N$.

Note that for any $z,x'\in Z_{d,n}, n\in\mathbb{N}$ there exists $k\geqslant 0$ such that
\[S^{(n)}(z,x')=(\beta_{V_{n},W_{k}(p_{n}(x))}\hat{\otimes}1)({S_{1}}^{(n)}(z,x')),\] 
where ${S_{1}}^{(n)}(z,x')\in\mathcal{A}(W_{k}(p_{n}(x)))\hat{\otimes}\mathcal{K}
$. Let $R$ be the propagation of the sequence $(T^{(0)},\dots,T^{(n)},\dots)$. Since the map $p_{n}$ is uniformly expansive for each $n\in\mathbb{N}$, there exists $\tilde{R}>0$ such that \[T^{(n)}(x,x')=0\quad\text {whenever}\quad d_{P_{d}(Y_{n})}(p_{n}(x),p_{n}(x'))>\tilde{R}\] for any $x,x'\in Z_{d,n},n\in\mathbb{N}$.
Take $N\in\mathbb{N}$ large enough such that for any $n\geqslant N$, there exists $\tilde{l_{n}}>0$ such that $k+\tilde{R}<\tilde{l_{n}}$. Then we have the trivialization 
\[t_{p_{n}(x)}:\{H_{z}\}_{z\in B_{P_{d}(Y_{n})}(p_{n}(x),\tilde{l_{n}})}\rightarrow B_{P_{d}(Y_{n})}(p_{n}(x),\tilde{l_{n}})\times H.\]
If $d_{P_{d}(Y_{n})}(p_{n}(x),p_{n}(z))\leqslant \tilde{R}$, then
the affine isometry \[t_{p_{n}(x)p_{n}(z)}=t_{p_{n}(x)}(\omega)\circ {t}^{-1}_{p_{n}(z)}(\omega):H\rightarrow H\]
maps $W_{k}(p_{n}(z))$
onto an affine subspace of $W_{k+\tilde{R}}(p_{n}(x))$ in $V_{n}$ for all $\omega\in B_{P_{d}(Y_{n})}(p_{n}(x),\tilde{l_{n}})\cap B_{P_{d}(Y_{n})}(p_{n}(z),\tilde{l_{n}})$ since $B_{P_{d}(Y_{n})}(p_{n}(z),k)\subseteq B_{P_{d}(Y_{n})}(p_{n}(x),k+\tilde{R})$.
By Definition \ref{def. 5.4}, the isometry
\[t_{p_{n}(x)p_{n}(z)}:W_{k}(p_{n}(z))\rightarrow W_{k+\tilde{R}}(p_{n}(x))\hookrightarrow V_{n}
\]
induces a $\ast$-homomorphism
\[(t_{p_{n}(x)p_{n}(z)})_{*}:\mathcal{A}(W_{k}(p_{n}(z)))\xrightarrow{\cong}\mathcal{A}(W_{k+\tilde{R}}(p_{n}(x)))\xrightarrow{\beta_{V_{n},W_{k+\tilde{R}}(p_{n}(x))}}\mathcal{A}(V_{n}).\]
We define 
\[((t_{p_{n}(x)p_{n}(z)})_{*})(S^{(n)}(z,x')):=((t_{p_{n}(x)p_{n}(z)})_{*})(S_{1}^{(n)}(z,x')).
\]
Observe that for $n\in\mathbb{N}$ large enough, this definition does not depend on the choice of $k$ (Cf. \cite{HG04, HKT98}).

The $\ast$-structure for $\mathbb{C}_{u,\infty}
[(P_{d}(X_{n}),\mathcal{A}(V_{n}))_{n\in\mathbb{N}}]$ is defined by 
\[{[(T^{(0)},\dots,T^{(n)},\dots)]}^{*}=[({(T^{*})}^{(0)},\dots,{(T^{*})}^{(n)},\dots)],\]
where 
\[{(T^{*})}^{(n)}(x,x')=(t_{p_{n}(x)p_{n}(x')})_{*}
((T^{(n)}(x',z))^{*})\]
for all but finitely many $n$, and $0$ otherwise.

So far, the 
algebraic twisted Roe algebra at infinity $\mathbb{C}_{u,\infty}
[(P_{d}(X_{n}),\mathcal{A}(V_{n}))_{n\in\mathbb{N}}]$ is made into a $\ast$-algebra. 

\begin{definition}\label{def. 5.6}
 For any $d\geqslant 0$, the maximal twisted Roe algebra at infinity
$C^{*}_{u,\max,\infty}((P_{d}(X_{n}),\mathcal{A}(V_{n}))_{n\in\mathbb{N}})$ 
is defined to be the completion of $\mathbb{C}_{u,\infty}
[(P_{d}(X_{n}),\mathcal{A}(V_{n}))_{n\in\mathbb{N}}]$ with respect to the maximal norm 
\[\|T\|_{\max}:=\sup\{\|\phi(T)\|_{\mathcal{B}(H_{\phi})}	\mid\phi:	\mathbb{C}_{u,\infty}[(P_{d}(X_{n}),\mathcal{A}(V_{n}))_{n\in\mathbb{N}}]\rightarrow\mathcal{B}(H_{\phi}), \text{a}\ast\text{-representation}\}.\] 
\end{definition}
\noindent We can also define a reduced norm of the algebraic twisted Roe algebra at infinity as in \cite[Remark 6.4]{GLWZ24}. The reduced twisted Roe algebra at infinity is defined to be the reduced completion of $*$-algebra $\mathbb{C}_{u,\infty}
[(P_{d}(X_{n}),\mathcal{A}(V_{n}))_{n\in\mathbb{N}}]$, and denoted by $C^{*}_{u,\infty}((P_{d}(X_{n}),\mathcal{A}(V_{n}))_{n\in\mathbb{N}})$.

\begin{definition}\label{def. 5.7}
 Let $d\geqslant 0$. The algebraic twisted localization Roe algebra at infinity
\[\mathbb{C}_{u, L, \infty}
[(P_{d}(X_{n}),\mathcal{A}(V_{n}))_{n\in\mathbb{N}}]\]is defined to be the
$\ast$-algebra of all bounded and uniformly norm-continuous functions
\[f:\mathbb{R}_{+}\rightarrow\mathbb{C}_{u, \infty}
[(P_{d}(X_{n}),\mathcal{A}(V_{n}))_{n\in\mathbb{N}}]\]
such that $f(t)$ is of the form $f(t)=[({f}^{(0)}(t),\dots,{f}^{(n)}(t),\dots)]$ for all $t\in\mathbb{R}_{+}$, where the family of functions $({f}^{(n)}(t))_{n\in\mathbb{N}, t\geqslant 0}$ satisfies the conditions in Definition \ref{def. 5.5} with uniform constants,
and there exists a bounded function $R(t):\mathbb{R}_{+}\rightarrow\mathbb{R}_{+}$ with $\lim_{t\rightarrow\infty} R(t)=0$ such that
\[({f}^{(n)}(t))(x,x')=0\quad \text{whenever}\quad d(x,x')>R(t)\]
for all $x,x'\in Z_{d,n},n\in\mathbb{N}, t\in\mathbb{R}_{+}$.
\end{definition}
\begin{definition}\label{def. 5.8}
Let $d\geqslant 0$. The maximal twisted localization Roe algebra at infinity
\[C^{*}_{u, L, \max,\infty}((P_{d}(X_{n}),\mathcal{A}(V_{n}))_{n\in\mathbb{N}})\]
is defined to be the completion of $\mathbb{C}_{u, L, \infty}
[(P_{d}(X_{n}),\mathcal{A}(V_{n}))_{n\in\mathbb{N}}]$ with respect to the norm 
\[\|f\|_{\max}:=\sup_{t\in\mathbb{R}_{+}}\|f(t)\|_{\max}.\]
\end{definition}

We can define a natural evaluation-at-zero map
\[e^{\mathcal{A}}:C^{*}_{u, L, \max,\infty}((P_{d}(X_{n}),\mathcal{A}(V_{n}))_{n\in\mathbb{N}})\rightarrow
C^{*}_{u, \max,\infty}((P_{d}(X_{n}),\mathcal{A}(V_{n}))_{n\in\mathbb{N}})
\]
by $e(f)=f(0)$, which induces the following $*$-homomorphism on $K$-theory
\[e_{*}^{\mathcal{A}}:K_{*}(C^{*}_{u, L, \max,\infty}((P_{d}(X_{n}),\mathcal{A}(V_{n}))_{n\in\mathbb{N}}))\rightarrow
K_{*}(C^{*}_{u, \max,\infty}((P_{d}(X_{n}),\mathcal{A}(V_{n}))_{n\in\mathbb{N}})).
\]

\subsection{Constructions of the Bott maps $\beta$ and $\beta_{L}$}
In this subsection, we shall define an asymptotic morphism $\beta$, called the Bott map, from $C^{*}_{u, \max,\infty}((P_{d}(X_{n}))_{n\in\mathbb{N}})$ to
$C^{*}_{u, \max,\infty}((P_{d}(X_{n}),\mathcal{A}(V_{n}))_{n\in\mathbb{N}})$, and the localized counterpart $\beta_{L}$ to build the following commutative diagram:

\begin{center}
\begin{tikzcd}
\lim\limits_{d\to\infty}{K_{*+1}(C^{*}_{u, L, \max,\infty}((P_{d}(X_{n}))_{n\in\mathbb{N}}))} \arrow[rr,"e_{*}"] \arrow[d,"(\beta_{L})_{*}"] &  & \lim\limits_{d\to\infty}{K_{*+1}(C^{*}_{u, \max,\infty}((P_{d}(X_{n}))_{n\in\mathbb{N}}))} \arrow[d,"\beta_{*}"] \\
\lim\limits_{d\to\infty}{K_{*}(C^{*}_{u, L, \max,\infty}((P_{d}(X_{n}),\mathcal{A}(V_{n}))_{n\in\mathbb{N}}))} \arrow[rr,"e_{*}^{\mathcal{A}}"]           &  & \lim\limits_{d\to\infty}{K_{*}(C^{*}_{u, \max,\infty}((P_{d}(X_{n}),\mathcal{A}(V_{n}))_{n\in\mathbb{N}}))}  .        
\end{tikzcd}
\end{center}

\quad For each $y\in F_{d,n}$ and $n\in\mathbb{N}$, we define a Clifford operator
\[C_{V_{n},t_{y}(y)(s(y))}:V_{n}\rightarrow \text{Cliff}(V^{0}_{n})\]
 by the function
\[v\mapsto v-t_{y}(y)(s(y))\in\text{Cliff}(V^{0}_{n}),\]
where $v-t_{y}(y)(s(y))$ is considered as an element of 
$\text{Cliff}(V^{0}_{n})$ via the inclusion $V^{0}_{n}\subseteq\text{Cliff}(V^{0}_{n})$. It is a degree one, essentially self-adjoint, unbounded multiplier of $\mathcal{C}(V_{n})$, with domain the compactly supported functions in $\mathcal{C}(V_{n})$. 
Let $X$ be the function given by $x\mapsto x$ on $\mathbb{R}$. 
 By Lemma 3.3 in \cite{HKT98}, the operator $X\hat\otimes 1+1\hat \otimes C_{V_{n},t_{y}(y)(s(y))}$ actually a degree one, essentially self-adjoint, unbounded multiplier of $\mathcal{A}(V_{n})=\mathcal{S}\hat\otimes\mathcal{C}(V_{n})$. Applying the functional calculus, the inclusion of the 0-dimensional affine subspace $\{
t_{y}(y)(s(y))\}$ into $V_{n}$ induces a $\ast$-homomorphism
\[\beta(y):\mathcal{S}\cong\mathcal{A}(\{
t_{y}(y)(s(y))\})\rightarrow
\mathcal{A}(V_{n})\]
by the formula
\[(\beta(y))(g)=g(X\hat\otimes 1+1\hat \otimes C_{V_{n},t_{y}(y)(s(y))}).\]\par

\begin{definition}\label{def. 5.9}
 Let $d\geqslant 0$. For each $t\in [1,\infty)$, the Bott map is defined as the morphism
\[\beta_{t}:\mathcal{S}\hat\otimes\mathbb{C}_{u,\infty}[(P_{d}(X_{n}))_{n\in\mathbb{N}}]\rightarrow \mathbb{C}_{u, \infty}[(P_{d}(X_{n}),\mathcal{A}(V_{n}))_{n\in\mathbb{N}}]\]
by
\[\beta_{t}(g\hat\otimes T)=[((\beta_{t}(g\hat\otimes T))^{(0)},\dots,(\beta_{t}(g\hat\otimes T))^{(n)},\dots)]\]
for all $g\in\mathcal{S}$ and $T=[(T^{(0)},\dots,T^{(n)},\dots)]\in\mathbb{C}_{u,\infty}[(P_{d}(X_{n}))_{n\in\mathbb{N}}],$ where
	\[(\beta_{t}(g\hat\otimes T))^{(n)}(x,x')=(\beta(p_{n}(x)))(g_{t})\hat\otimes T^{(n)}(x,x')\]
for $x,x'\in Z_{d,n}$, $n\in\mathbb{N}$, and $g_{t}(r)=g(t^{-1}r)$ for all $r\in\mathbb{R}$.\end{definition}

\begin{definition}\label{def. 5.10}
Let $d\geqslant 0$. For each $t\in [1,\infty)$, the localized Bott map is defined as the morphism 
\[(\beta_{L})_{t}:\mathcal{S}\hat\otimes\mathbb{C}_{u,L,\infty}[(P_{d}(X_{n}))_{n\in\mathbb{N}}]\rightarrow \mathbb{C}_{u,L,  \infty}[(P_{d}(X_{n}),\mathcal{A}(V_{n}))_{n\in\mathbb{N}}]\]
by 
\[((\beta_{L})_{t}(f))(s)=\beta_{t}(f(s))\] for all $s\in\mathbb{R}_{+}$.
\end{definition}

\begin{lemma}\label{rem. 5.11}
 For each $d\geqslant 0$, the maps $(\beta_{t})_{t\geqslant 1}$ and $((\beta_{L})_{t})_{t\geqslant 1}$ extends to asymptotic morphisms
\begin{align*}
	\beta:\mathcal{S}\hat\otimes C^{*}_{u,\max \infty}((P_{d}(X_{n}))_{n\in\mathbb{N}})&\rightarrow C^{*}_{u,  \max,\infty}((P_{d}(X_{n}),\mathcal{A}(V_{n}))_{n\in\mathbb{N}}),\\	\beta_{L}:\mathcal{S}\hat\otimes C^{*}_{u,L,\max \infty}((P_{d}(X_{n}))_{n\in\mathbb{N}})&\rightarrow C^{*}_{u,L,  \max,\infty}((P_{d}(X_{n}),\mathcal{A}(V_{n}))_{n\in\mathbb{N}}).
\end{align*}
\end{lemma}
\begin{proof}[\rm{\textbf{Proof.}}] We shall prove that the maps $(\beta_{t})_{t\geqslant 1}$ can be extended to the asymptotic morphism $\beta$ as above. The case for $\beta_{L}$ is quite similar. \par 
\textbf{Step 1.} We first show that the maps $(\beta_{t})_{t\geqslant 1}$ is a well-defined asymptotic morphism from $\mathcal{S}\hat\otimes\mathbb{C}_{u,\infty}[(P_{d}(X_{n}))_{n\in\mathbb{N}}]$ to 
$C^{*}_{u,  \max,\infty}((P_{d}(X_{n}),\mathcal{A}(V_{n}))_{n\in\mathbb{N}})$. It suffices to verify that $\beta_{t}(g\hat\otimes T)$ is norm continuous in $t\in [1,\infty)$, and 
\begin{align*}
	\begin{split}
		\lim_{t\rightarrow\infty}\left\{
		\begin{array}{c}
		\beta_{t}((g_{1}\hat\otimes T_{1})(g_{2}\hat\otimes T_{2}))-\beta_{t}(g_{1}\hat\otimes T_{1})\beta_{t}(g_{2}\hat\otimes T_{2})\\
		\beta_{t}((g_{1}\hat\otimes T_{1})+(g_{2}\hat\otimes T_{2}))-\beta_{t}(g_{1}\hat\otimes T_{1})+\beta_{t}(g_{2}\hat\otimes T_{2})\\
		\lambda\beta_{t}(g_{1}\hat\otimes T_{1})-\beta_{t}(\lambda (g_{1}\hat\otimes T_{1}))\\
	\beta_{t}((g_{1}\hat\otimes T_{1})^{*})-(\beta_{t}(g_{1}\hat\otimes T_{1}))^{*}
		\end{array}
		\right\}=0
	\end{split}
\end{align*}
for all $g_{1},g_{2}\in\mathcal{S}$,
$T_{1},T_{2}\in\mathbb{C}_{u,\infty}[(P_{d}(X_{n}))_{n\in\mathbb{N}}]$ and $\lambda\in\mathbb{C}$. It suffices to prove the first equality only, and the rest can be checked similarly. By Definition \ref{def. 5.5}, we calculate that
\begin{equation}\begin{aligned}\label{eq 2}
&\sup_{x,x'\in Z_{d,n},n\in\mathbb{N}}\|(\beta_{t}((g_{1}\hat\otimes T_{1})(g_{2}\hat\otimes T_{2})))^{(n)}(x,x')-(\beta_{t}(g_{1}\hat\otimes T_{1})\cdot\beta_{t}(g_{2}\hat\otimes T_{2}))^{(n)}(x,x')\|\\
	&=\sup_{x,x'\in Z_{d,n},n\in\mathbb{N}}  \|(\beta(y))(g_{1}g_{2})_{t}\hat\otimes (T_{1}T_{2})^{(n)}(x,x')\\
	&\quad-\sum_{\xi\in Z_{d,n}}
	((\beta_{t}(g_{1}\hat\otimes T_{1}))^{(n)}(x,\xi))\cdot ((t_{yy'})_{*}	(\beta_{t}(g_{2}\hat\otimes T_{2}))^{(n)}(\xi,x'))
	\|(\text{where}\ y=p_{n}(x))\\
	&=\sup_{x,x'\in Z_{d,n},n\in\mathbb{N}} \| \sum_{\xi\in Z_{d,n}}(\beta(y)(g_{1})_{t})(\beta(y)(g_{2})_{t})\hat\otimes T^{(n)}_{1}(x,\xi)T^{(n)}_{2}(\xi,x')\\
	&\quad-\sum_{\xi\in Z_{d,n}}(\beta(y)(g_{1})_{t})\hat\otimes T^{(n)}_{1}(x,\xi)\cdot 
	((t_{yy'})_{*}
	((\beta(y'))(g_{2})_{t}\hat\otimes T^{(n)}_{2}(\xi,x')))\|(\text{where}\ y'=p_{n}(\xi))\\
	&=\sup_{x,x'\in Z_{d,n},n\in\mathbb{N}} \| \sum_{\xi\in Z_{d,n}}\beta(y)(g_{1})_{t}
	[\beta(y)-(t_{yy'})_{*}
	\beta(y')](g_{2})_{t}
	\hat\otimes T^{(n)}_{1}(x,\xi)T^{(n)}_{2}(\xi,x')\|.\end{aligned}\end{equation}
where $t_{yy'}=t_{y}(w)\circ t^{-1}_{y'}(w):H\rightarrow H$ for all $w\in B_{Y_{n}}(y,l_{n})\cap B_{Y_{n}}(y'
,l_{n})$, and 
$(t_{yy'})_{*}$ is a map from $\mathcal{A}(W_{C}(y'))$ onto $\mathcal{A}(W_{C}(y))$
mapping $g\hat\otimes h$ to $g\hat\otimes (t_{yy'})_{*}(h)$
for any non-empty subset $C\subseteq B_{Y_{n}}(y,l_{n})\cap B_{Y_{n}}(y'
,l_{n})$.
Then we have that 
\begin{align*}
	W_{C}(y)&=t_{yy'}(W_{C}(y')),\\
	\beta_{W_{C}(y),\{t_{y}(y')(s(y'))\}}(g)&=(t_{yy'})_{*}(\beta_{W_{C}(y'),\{t_{y'}(y')(s(y'))\}}(g)),
\end{align*}
for all $g\in \mathcal{S}$.
By the Stone-Weierstrass theorem, we mainly consider the generators $g(x)=(x\pm i)^{-1}$ of $\mathcal{S}=C_{0}(\mathbb{R})$, together with Definition \ref{def. 5.4}, the formula $\beta(y)(g_{t})-(t_{yy'})_{*}(\beta(y')(g_{t}))$ in the middle of \eqref{eq 2} has the following estimates
\begin{align*}
	&\|\beta(y)(g_{t})-(t_{yy'})_{*}(\beta(y')(g_{t}))\|\\
	&=\|\beta_{V_{n},W_{C}(y)}\circ \beta_{W_{C}(y),\{t_{y}(y)(s(y))\}}(g_{t})-(t_{yy'})_{*}(\beta_{V_{n},W_{C}(y')}\circ \beta_{W_{C}(y'),\{t_{y'}(y')(s(y'))\}}(g_{t}))\|\\
	&=\|\beta_{V_{n},W_{C}(y)}\circ\beta_{W_{C}(y),\{t_{y}(y)(s(y))\}}(g_{t})-\beta_{V_{n},W_{C}(y)}\circ\beta_{W_{C}(y),\{t_{y}(y')(s(y'))\}}(g_{t})\|\\
    &\leqslant\|\beta_{W_{C}(y),\{t_{y}(y)(s(y))\}}(g_{t})-\beta_{W_{C}(y),\{t_{y}(y')(s(y'))\}}(g_{t})\|\\
	&=
	\|g_{t}(X\hat\otimes 1+1\hat\otimes C_{W_{C}(y),\{t_{y}(y)(s(y))\}})-g_{t}(X\hat\otimes 1+1\hat\otimes C_{W_{C}(y),\{t_{y}(y')(s(y'))\}})\|\\
	&\leqslant t^{-1}\|t_{y}(y)(s(y))-t_{y}(y')(s(y'))\|\\
	&\leqslant t^{-1}\rho_{2}(d(y,y'))\rightarrow 0 \quad (t\rightarrow \infty).
\end{align*}
It follows from an approximation argument,
together with \cite{Yu00} and Lemma 7.3 in \cite{HKT98}, that for all $d\geqslant 0$, $R>0$, $r>0$, $c>0$, $\epsilon >0$, there exists $t_{0}>1$ such that for every
$y,y'\in F_{d,n}$, $n\in\mathbb{N}$,  with $d(y,y')\leqslant R$, for all $t\geqslant t_{0}$ and all $g\in\mathcal{S}$ with 
\[\text{supp}(g)\subseteq[-r,r],\quad  \|g'\|_{\infty}\leqslant c,\] we have that
\begin{equation*}
	\|\beta(y)(g_{t})-(t_{yy'})_{*}(\beta(y')(g_{t}))\|<\epsilon.
\end{equation*} 
Combining Lemma 3.4 in \cite{GWY08}, the formula 
$\beta_{t}((g_{1}\hat\otimes T_{1})(g_{2}\hat\otimes T_{2}))-\beta_{t}(g_{1}\hat\otimes T_{1})\cdot\beta_{t}(g_{2}\hat\otimes T_{2})$ converges uniformly to $0$ as $t\rightarrow \infty$ in norm.
Consequently, the maps $(\beta_{t})_{t\geqslant 1}$ becomes a well-defined asymptotic morphism from $*$-algebra $\mathcal{S}\hat\otimes\mathbb{C}_{u,\infty}[(P_{d}(X_{n}))_{n\in\mathbb{N}}]$ to $C^{*}$-algebra
$C^{*}_{u,  \max,\infty}(P_{d}(X_{n}),\mathcal{A}(V_{n}))_{n\in\mathbb{N}}$. Therefore, the maps $(\beta_{t})_{t\geqslant 1}$ define a $*$-homomorphism $\beta$ from $*$-algebra $\mathcal{S}\hat\otimes\mathbb{C}_{u,\infty}[(P_{d}(X_{n}))_{n\in\mathbb{N}}]$ to the asymptotic $C^{*}$-algebra
\[\mathcal{Q}(C^{*}_{u,  \max,\infty}(P_{d}(X_{n}),\mathcal{A}(V_{n}))_{n\in\mathbb{N}}):=\frac{C_{b}([1,\infty),C^{*}_{u, \max,\infty}(P_{d}(X_{n}),\mathcal{A}(V_{n}))_{n\in\mathbb{N}})}{C_{0}([1,\infty),C^{*}_{u,  \max,\infty}(P_{d}(X_{n}),\mathcal{A}(V_{n}))_{n\in\mathbb{N}})}\]
satisfying $\|\beta(g\hat\otimes T)\|\leqslant\|g\|\cdot\|T\|$ for all $g\in\mathcal{S}$ and $T\in \mathbb{C}_{u,\infty}[(P_{d}(X_{n}))_{n\in\mathbb{N}}]$.

\textbf{Step 2.} We shall prove that the maps $(\beta_{t})_{t\geqslant 1}$ extend to an asymptotic morphism from 
$\mathcal{S}\hat\otimes C^{*}_{u,\max \infty}((P_{d}(X_{n}))_{n\in\mathbb{N}})$ to 
$C^{*}_{u,  \max,\infty}((P_{d}(X_{n}),\mathcal{A}(V_{n}))_{n\in\mathbb{N}})$.\par 
 It follows from the universal property of the maximal norm that $\beta$ extends to a linear map from the algebraic tensor product  $\mathcal{S}\hat\otimes_{\text{alg}} C^{*}_{u,\max \infty}((P_{d}(X_{n}))_{n\in\mathbb{N}})$ to 
$\mathcal{Q}(C^{*}_{u,  \max,\infty}(P_{d}(X_{n}),\mathcal{A}(V_{n}))_{n\in\mathbb{N}})$ satisfying 
$\|\beta(g\hat\otimes T)\|\leqslant\|g\|\cdot\|T\|$ for all $g\in\mathcal{S}$ and $T\in \mathbb{C}_{u,\infty}[(P_{d}(X_{n}))_{n\in\mathbb{N}}]$. 
Furthermore, by the
universality of the maximal tensor product, the map $\beta$ can be extended to an ${*}$-homomorphism from $\mathcal{S}\hat\otimes_{\max} C^{*}_{u,\max \infty}((P_{d}(X_{n}))_{n\in\mathbb{N}})$ to 
$\mathcal{Q}(C^{*}_{u,  \max,\infty}(P_{d}(X_{n}),\mathcal{A}(V_{n}))_{n\in\mathbb{N}})$. Since $\mathcal{S}$ is nuclear, we conclude that the maps $(\beta_{t})_{t\geqslant 1}$ extends to an asymptotic morphism from 
$\mathcal{S}\hat\otimes C^{*}_{u,\max \infty}((P_{d}(X_{n}))_{n\in\mathbb{N}})$ to 
$C^{*}_{u,  \max,\infty}((P_{d}(X_{n}),\mathcal{A}(V_{n}))_{n\in\mathbb{N}})$. 
The proof is completed.\end{proof}

\section{Reduction to cases with non-twisted coefficients}\label{section 6}
The aim of this section is to illustrate the following result.
\begin{theorem}\label{the. 6.1}
Let $(X_{n})_{n\in\mathbb{N}}$ be a sequence of discrete metric spaces with uniform bounded geometry which admits an A-by-FCE coarse fibration structure. Then the evaluation map
\[e_{*}^{\mathcal{A}}: \lim_{d\rightarrow\infty}K_{*}(C^{*}_{u, L, \max,\infty}((P_{d}(X_{n}),\mathcal{A}(V_{n}))_{n\in\mathbb{N}}) )     
\rightarrow\lim_{d\rightarrow\infty}K_{*}(C^{*}_{u,  \max,\infty}((P_{d}(X_{n}),\mathcal{A}(V_{n}))_{n\in\mathbb{N}}))\]
is an isomorphism.
\end{theorem}
The proof proceeds by decomposing the twisted algebras at infinity into various smaller ideals supported on certain coherent systems, and then showing that the evaluation maps for those subalgebras are isomorphism on $K$-theory. 
Furthermore, we glue these ideals together, and Theorem \ref{the. 6.1} is established by applying a Mayer-Vietoris sequence argument and the five lemma.

\subsection{Twisted algebras supported on certain coherent systems}

In this subsection, we shall discuss ideals of the twisted Roe algebra at infinity supported on certain open subsets of $\mathbb{R}_{+}\times V_{n}$ for any $n\in \mathbb{N}$. Recall first that for any
finite-dimensional affine subspaces $V_{a},V_{b}\subseteq H$ with
$V_{a}\subseteq V_{b}$, and for any 
subset $O\subset \mathbb{R}_{+}\times V_{a}$, we define (see Figure \ref{fig.1})
    \[\overline{O}^{\beta_{ba}}=\{(t,v_{ba}+v_{a})\in\mathbb{R}_{+}\times V_{b}\mid (\sqrt{t^{2}+\|v_{ba}\|^{2}},v_{a})\in O\}.\]
The support of any function $a\in\mathcal{A}(V_{a})$ is the
complement of all points $(t,v)\in \mathbb{R}_{+}\times V_{a}$ such that there exists 
$g\in C_{0}(\mathbb{R}_{+}\times V_{a})$ such that
$g(t,v)\neq0 $ but $g\cdot a=0$. Moreover, we have that \[\text{Supp}(\beta_{ba}(a))=\overline{\text{Supp}(a)}^{\beta_{ba}}.\]

\begin{definition}\label{def. 6.2}
 A collection $O={(O_{n,y})}_{y\in Y_{n},n\in\mathbb{N}}$ of open subsets of $\mathbb{R}_{+}\times V_{n}$ is said to be a coherent system if, for all but finitely many $n\in\mathbb{N}$, the following conditions hold
 \begin{enumerate}
     \item [(1)]\hspace{0pt}
 for any non-empty subset $C\subseteq B_{Y_{n}}(y,\frac{{l_{n}}}{2})\cap B_{Y_{n}}(y',\frac{{l_{n}}}{2})$
with $y,y'\in Y_{n}, n\in\mathbb{N}$, we have that
\[O_{n,y}\cap (\mathbb{R}_{+}\times W_{C}(y))=t_{yy'}(O_{n,y'}\cap (\mathbb{R}_{+}\times W_{C}(y'))),\]
where
\begin{align*}
	W_{C}(y)&=\text{affine-span}\{t_{y}(z)s(z)\mid z\in C\}=t_{yy'}(W_{C}(y')),\\
	W_{C}(y')&=\text{affine-span}\{t_{y'}(z)s(z)\mid z\in C\}=t_{y'y}(W_{C}(y)),
\end{align*}
and 
$t_{yy'}=t_{y}(z)\circ t^{-1}_{y'}(z)$ is an affine isometry 
for all $z\in B_{Y_{n}}(y,{l_{n}})\cap B_{Y_{n}}(y',l_{n})$.
\item [(2)]\hspace{0pt}for any non-empty subset $C\subseteq B_{Y_{n}}(y,\frac{{l_{n}}}{2})$ with $y\in Y_{n}, n\in\mathbb{N}$, and any affine subspace $W$ satisfying $W_{C}(y)\subseteq W\subseteq V_{n}$, we have that
\[{\overline{O_{n,y}\cap (\mathbb{R}_{+}\times W_{C}(y))}}^{\beta_{W,W_{C}(y)}}\subseteq O_{n,y}\cap (\mathbb{R}_{+}\times W).
\]\end{enumerate}\end{definition}

\begin{example}\label{exm. 6.3}Below we provide a typical example of coherent system, which will be used in the sequel.

Fix $r>0$. Define
\[O_{n,y}:=\bigcup_{z\in B_{Y_{n}}(y,{l_{n}})}B_{\mathbb{R}_{+}\times V_{n}}(t_{y}(z)(s(z)),r)\]
for all $y\in Y_{n}$ and $n\in\mathbb{N}$. Then the collection $O={(O_{n,y})}_{y\in Y_{n},n\in\mathbb{N}}$ 
is a coherent systems of open subsets of $\mathbb{R_{+}}\times V_{n}$. Indeed,
let $C$ be any non-empty subset contained in $ B_{Y_{n}}(y,\frac{{l_{n}}}{2})\cap B_{Y_{n}}(y',\frac{{l_{n}}}{2})$ 
with $y,y'\in Y_{n}$ for sufficient large $n$. If there exists $z\in Y_{n}$ such that $B_{\mathbb{R}_{+}\times V_{n}}(t_{y}(z)(s(z)),r)\cap(\mathbb{R}_{+}\times W_{C}(y))\neq \emptyset.$
Then $z\in  B_{Y_{n}}(y,{l_{n}})\cap B_{Y_{n}}(y',{l_{n}})$. Since
\begin{align*}
	B_{\mathbb{R}_{+}\times V_{n}}(t_{y}(z)(s(z)),r)&=t_{yy'}(B_{\mathbb{R}_{+}\times V_{n}}(t_{y'}(z)(s(z)),r)),\\
	\mathbb{R}_{+}\times W_{C}(y)&= t_{yy'}(\mathbb{R}_{+}\times W_{C}(y')),
\end{align*}
we have that
$O_{n,y}\cap(\mathbb{R}_{+}\times W_{C}(y))=t_{yy'}(O_{n,y'}\cap(\mathbb{R}_{+}\times W_{C}(y'))).$
Moreover, if $C$ is denoted to be a non-empty subset of $B_{Y_{n}}(y,\frac{{l_{n}}}{2})$ and $W$ is an affine subspace satisfying $W_{C}(y)\subseteq W\subseteq V_{n}$. Using the  above formulas, we calculate that
\[{\overline{B_{\mathbb{R}_{+}\times V_{n}}(t_{y}{(z)(s(z))},r)\cap (\mathbb{R}_{+}\times W_{C}(y))}}^{\beta_{W,W_{C}(y)}}\subseteq B_{\mathbb{R}_{+}\times V_{n}}(t_{y}{(z)(s(z))},r).\]
It follows that
\[{\overline{O_{n,y}\cap (\mathbb{R}_{+}\times W_{C}(y))}}^{\beta_{W,W_{C}(y)}}\subseteq O_{n,y}\cap (\mathbb{R}_{+}\times W).\]
Hence, the collection $O:=(O_{n,y})_{y\in Y_{n},n\in\mathbb{N}}$ becomes a coherent system of open subsets of $\mathbb{R}_{+}\times V_{n}, n\in\mathbb{N}$.
Since the map $p_{n}$ is surjective for each 
$n\in\mathbb{N}$,
there always exists $x\in X_{n}$ such that $y=p_{n}(x)$. The above collection is sometimes written in the following form
\[O_{n,p_{n}(x)}:=\bigcup_{z\in B_{Y_{n}}(p_{n}(x),{l_{n}})}B_{\mathbb{R}_{+}\times V_{n}}(t_{p_{n}(x)}(z)(s(z)),r).\]
\end{example}

\begin{definition}\label{def. 6.4}
 Let $O:=(O_{n,y})_{y\in Y_{n},n\in\mathbb{N}}$ be a coherent system of open subsets of $\mathbb{R}_{+}\times V_{n},n\in \mathbb{N}$. For any $d\geqslant 0 $, define 
$ \mathbb{C}_{u,\infty}
[(P_{d}(X_{n}),\mathcal{A}(V_{n}))_{n\in\mathbb{N}}]_{O}$ to be the $\ast$-subalgebra of $\mathbb{C}_{u,\infty}
[(P_{d}(X_{n}),\mathcal{A}(V_{n}))_{n\in\mathbb{N}}]$ consisting of all the equivalent classes 
$T=[(T^{(0)},\dots,T^{(n)},\dots)]$ of sequences $(T^{(0)},\dots,T^{(n)},\dots)$
satisfying 
\[\text{Supp}(T^{(n)}(x,x'))\subseteq O_{n,p_{n}(\bar{x})}\] for all $x,x'\in Z_{d,n}$ with $p_{n}(x)\in \text{Star}(p_{n}(\bar{x}))$ for some $\bar{x}\in Z_{d,n}$ and $n\in\mathbb{N}$ large enough depending only on the sequence
$(T^{(0)},\dots,T^{(n)},\dots)$.
\end{definition}

\begin{remark}\label{rem. 6.5} Note that the above definition is well-defined. Indeed, as mentioned in Section \ref{section 3}, by choosing $H_{X}$ as the Hilbert space $\ell^{2}(Z)\hat\otimes H_{0}$, any bounded operator $T\in \mathcal{B}(\ell^{2}(Z)\hat\otimes H_{0})$ can be interpreted as a $Z$-by-$Z$ matrix. To ensure that $\mathbb{C}_{u,\infty}
[(P_{d}(X_{n}),\mathcal{A}(V_{n}))_{n\in\mathbb{N}}]_{O}$ is closed under $*$-operations, the open subset $O_{n,p_{n}(x')}\cap (\mathbb{R}_{+}\times W_{C}(p_{n}(x')))$, obtained by trivializing the operator $T^{(n)}(x,x')$ for each $n\in\mathbb{N}$ along the trivialization of the neighborhood of $x'$, must be able to be transformed through an isometry $t_{p_{n}(x)p_{n}(x')}$ with the open subset $O_{n,p_{n}(x)}\cap (\mathbb{R}_{+}\times W_{C}(p_{n}(x)))$
obtained by trivializing the operator along the trivalization of the neighborhood of $x$, which is exactly required in condition (1) of Definition \ref{def. 6.2}.

In addition, as shown in the following commutative diagram, the isometry $(t_{p_{n}(x)p_{n}(x')})_{*}$ only applies to the finite-dimensional subspace $W_{C}(p_{n}(x'))$. From Definition \ref{def. 5.4} there exists a natural $*$-isomorphism, here we abuse notations slightly and still denote this map as $(t_{p_{n}(x)p_{n}(x')})_{*}$,
from $\mathcal{A}(W_{C}(p_{n}(x')))$ onto $\mathcal{A}(t_{p_{n}(x)p_{n}(x')}(W_{C}(p_{n}(x'))))=\mathcal{A}(W_{C}(p_{n}(x)))$. Condition (2) of Definition \ref{def. 6.2} indicates that the extension of open subsets of this subspace should be the same as that of the open subset obtained by rotating the space to $V_ {n} $ through the $*$-homomorphism $\beta_{V_{n}, W_{C}(p_{n}(x))}$. As a result, the two conditions outlined in Definition \ref{def. 6.2} guarantee the set 
$\mathbb{C}_{u,\infty}
[(P_{d}(X_{n}),\mathcal{A}(V_{n}))_{n\in\mathbb{N}}]_{O}$
constructed above is closed under $*$-operations. Similarly, it can be verified that it is also closed for multiplication operations, thereby forming an $*$-algebra.

\begin{center}
\begin{tikzcd}
{\mathcal{A}(V_{n})}  \arrow[rr,"(t_{p_{n}(x)p_{n}(x')})_{*}"] &      &       {\mathcal{A}(V_{n})}           \\
{\mathcal{A}(W_{C}(p_{n}(x'))))} \arrow[rr,"(t_{p_{n}(x)p_{n}(x')})_{*}"] \arrow[u]      &     & \mathcal{A}(W_{C}(p_{n}(x)))  \arrow[u,"\beta_{V_{n}, W_{C}(p_{n}(x))}"']
\end{tikzcd}
\end{center}
\end{remark}

\begin{definition}\label{def. 6.6}
    Let $O:=(O_{n,y})_{y\in Y_{n},n\in\mathbb{N}}$ be a coherent system of open subsets of $\mathbb{R}_{+}\times V_{n},n\in \mathbb{N}$. 
    \begin{enumerate}
        \item [(1)]\hspace{0pt}We define the $C^{*}$-algebra 
 \[C^{*}_{u,\infty}
 ((P_{d}(X_{n}),\mathcal{A}(V_{n}))_{n\in\mathbb{N}})_{O}\] to be the completion of $\mathbb{C}_{u,\infty}
 [(P_{d}(X_{n}),\mathcal{A}(V_{n}))_{n\in\mathbb{N}}]_{O}$ with respect to the reduced norm, as in \cite[Remark 6.4]{GLWZ24}.
\item [(2)]\hspace{0pt}Viewing $\mathbb{C}_{u,\infty}
[(P_{d}(X_{n}),\mathcal{A}(V_{n}))_{n\in\mathbb{N}}]_{O}$
as a $\ast$-subalgebra of $C^{*}_{u,\max,\infty}
((P_{d}(X_{n}),\mathcal{A}(V_{n}))_{n\in\mathbb{N}})$, we define the $C^{*}$-algebra 
\[C^{*}_{u,\phi, \infty}
((P_{d}(X_{n}),\mathcal{A}(V_{n}))_{n\in\mathbb{N}})_{O}\] to be
the completion of $\mathbb{C}_{u,\infty}
[(P_{d}(X_{n}),\mathcal{A}(V_{n}))_{n\in\mathbb{N}}]_{O}$ under the norm in
$C^{*}_{u,\max,\infty}((P_{d}(X_{n}),\mathcal{A}(V_{n}))_{n\in\mathbb{N}})$.
\item [(3)]\hspace{0pt}We define the $C^{*}$-algebra 
 \[C^{*}_{u,\max,\infty}
 ((P_{d}(X_{n}),\mathcal{A}(V_{n}))_{n\in\mathbb{N}})_{O}\] to be the completion of $\mathbb{C}_{u,\infty}
 [(P_{d}(X_{n}),\mathcal{A}(V_{n}))_{n\in\mathbb{N}}]_{O}$ concerning the maximal norm.
    \end{enumerate}
\end{definition}
Until now, we do not know whether the norm induced from $C^{*}_{u,\max,\infty}((P_{d}(X_{n}),\mathcal{A}(V_{n}))_{n\in\mathbb{N}})$ is equal to the supremum norm of all representations of the $*$-algebra $\mathbb{C}_{u,\infty}
 [(P_{d}(X_{n}),\mathcal{A}(V_{n}))_{n\in\mathbb{N}}]_{O}$ (After some technical preparation, we will prove they are equal in the next section). 
 Note that there exists a canonical quotient map 
 \[\lambda_{\text{max,red}}:C^{*}_{u,\max, \infty}
 ((P_{d}(X_{n}),\mathcal{A}(V_{n}))_{n\in\mathbb{N}})_{O} \rightarrow C^{*}_{u,\infty}
 ((P_{d}(X_{n}),\mathcal{A}(V_{n}))_{n\in\mathbb{N}})_{O}.
 \]
Since the reduced representation of $*$-algebras $ \mathbb{C}_{u,\infty}
[(P_{d}(X_{n}),\mathcal{A}(V_{n}))_{n\in\mathbb{N}}]$ is contained in the maximal one, there exists a canonical quotient map that restricts to the ideals
\[\lambda_{\phi\text{,red}}:C^{*}_{u,\phi, \infty}
 ((P_{d}(X_{n}),\mathcal{A}(V_{n}))_{n\in\mathbb{N}})_{O} \rightarrow C^{*}_{u,\infty}
 ((P_{d}(X_{n}),\mathcal{A}(V_{n}))_{n\in\mathbb{N}})_{O}.
 \]
Moreover, by the universal property of the maximal norm, the canonical quotient map $\lambda_{\phi,\text{red}}$ can be lifted to the map $\lambda_{\text{max, red}}$ via a quotient map $\lambda_{\max,\phi}$, i.e.,
\begin{center}
\begin{tikzcd}
{C^{*}_{u,\max, \infty}
 ((P_{d}(X_{n}),\mathcal{A}(V_{n}))_{n\in\mathbb{N}})_{O}} \arrow[r,"\lambda_{\text{max, red}}"] \arrow[d,"\lambda_{\max,\phi}"'] & {C^{*}_{u, \infty}
 ((P_{d}(X_{n}),\mathcal{A}(V_{n}))_{n\in\mathbb{N}})_{O}}  \\
{C^{*}_{u,\phi, \infty}
 ((P_{d}(X_{n}),\mathcal{A}(V_{n}))_{n\in\mathbb{N}})_{O}} \arrow[r,"\lambda_{\phi, \text{red}}"]           & {C^{*}_{u, \infty}
 ((P_{d}(X_{n}),\mathcal{A}(V_{n}))_{n\in\mathbb{N}})_{O}} \arrow[u,Rightarrow,no head]                          \end{tikzcd}
\end{center}
Especially, J\'{a}n \v{S}pakula
and Rufus Willett \cite{SW13} proved that if a metric space $X$ has Property $A$, then the canonical map between the maximal and reduced (uniform) Roe algebras is an isomorphism.

\begin{definition}\label{def. 6.7}
Let $r>0$ and $\varGamma_{n}$ be a subset of $Y_{n}$ for each $n\in\mathbb{N}$, we denote $\varGamma:=(\varGamma_{n})_{n\in\mathbb{N}}$. 
A coherent system $O:=(O_{n,y})_{y\in Y_{n},n\in\mathbb{N}}$ of open subsets of $\mathbb{R}_{+}\times V_{n},n\in \mathbb{N}$, is said to be $(\varGamma,r)$-separate if there exist open subsets 
$(O_{n,y,\gamma})_{\gamma\in \varGamma_{n}\cap B_{Y_{n}}(y, l_{n})}$ of $\mathbb{R}_{+}\times V_{n}$ for all $y\in Y_{n},n\in\mathbb{N},$ such that
\begin{enumerate}
   \item [(1)] \hspace{0pt}$O_{n,y}=\bigcup_{\gamma\in \varGamma_{n}\cap B_{Y_{n}}(y,l_{n})}O_{n,y,\gamma};$ 
\item [(2)] \hspace{0pt}$O_{n,y,\gamma}\cap O_{n,y,\gamma'}=\emptyset$ for distinct $\gamma,\gamma'\in\varGamma_{n}\cap B_{Y_{n}}(y,l_{n})$;
\item [(3)] \hspace{0pt}$O_{n,y,\gamma}\subseteq B_{\mathbb{R}_{+}\times V_{n}}(t_{y}(\gamma)(s(\gamma)),r)$ for each  $\gamma\in\varGamma_{n}\cap B_{Y_{n}}(y,l_{n})$.
\end{enumerate}
\end{definition}
We have the following result.
\begin{lemma}\label{lem. 6.8}
 Suppose that a coherent system $O:=(O_{n,y})_{y\in Y_{n},n\in\mathbb{N}}$ is  $(\varGamma,r)$-separate for some
$\varGamma:=(\varGamma_{n})_{n\in\mathbb{N}}$ and $r>0$.
Then the $*$-homomorphism
\[\lambda_{\phi,\text{red}}:C^{*}_{u,\phi, \infty}
((P_{d}(X_{n}),\mathcal{A}(V_{n}))_{n\in\mathbb{N}})_{O} \rightarrow C^{*}_{u,\infty}
((P_{d}(X_{n}),\mathcal{A}(V_{n}))_{n\in\mathbb{N}})_{O}
\]
is an isomorphism.\end{lemma}

We shall postpone the proof to the next subsection. For the remainder of this part, we define the localization counterpart of twisted algebras at infinity supported on certain coherent systems.

\begin{definition}\label{def. 6.9}
Let $O:=(O_{n,y})_{y\in Y_{n},n\in\mathbb{N}}$ be a coherent system of open subsets of $\mathbb{R}_{+}\times V_{n},n\in \mathbb{N}$. For any $d\geqslant 0 $, define 
$\mathbb{C}_{u,L,\infty}
[(P_{d}(X_{n}),\mathcal{A}(V_{n}))_{n\in\mathbb{N}}]_{O}$ to be the $\ast$-subalgebra of $\mathbb{C}_{u,L,\infty}
[(P_{d}(X_{n}),\mathcal{A}(V_{n}))_{n\in\mathbb{N}}]$ consisting of all functions
\[f:\mathbb{R}_{+}\rightarrow\mathbb{C}_{u, \infty}
[(P_{d}(X_{n}),\mathcal{A}(V_{n}))_{n\in\mathbb{N}}]_{O}.\]\end{definition} 
\begin{definition}\label{def. 6.10}
    Let $O:=(O_{n,y})_{y\in Y_{n},n\in\mathbb{N}}$ be a coherent system of open subsets of $\mathbb{R}_{+}\times V_{n},n\in \mathbb{N}$. 
    \begin{enumerate}
        \item [(1)]\hspace{0pt}The $C^{*}$-algebra
 $C^{*}_{u,L,\infty}
 ((P_{d}(X_{n}),\mathcal{A}(V_{n}))_{n\in\mathbb{N}})_{O}$ is defined to be the completion of $\mathbb{C}_{u,L,\infty}
 [(P_{d}(X_{n}),\mathcal{A}(V_{n}))_{n\in\mathbb{N}}]_{O}$ with respect to the reduced norm.
\item [(2)]\hspace{0pt}Viewing $\mathbb{C}_{u,L,\infty}
[(P_{d}(X_{n}),\mathcal{A}(V_{n}))_{n\in\mathbb{N}}]_{O}$
as a $\ast$-subalgebra of $C^{*}_{u,L,\max,\infty}
((P_{d}(X_{n}),\mathcal{A}(V_{n}))_{n\in\mathbb{N}})$, we define the $C^{*}$-algebra
$C^{*}_{u,L,\phi, \infty}
((P_{d}(X_{n}),\mathcal{A}(V_{n}))_{n\in\mathbb{N}})_{O}$ to be
the completion of $\mathbb{C}_{u,L,\infty}
[(P_{d}(X_{n}),\mathcal{A}(V_{n}))_{n\in\mathbb{N}}]_{O}$ under the norm in
$C^{*}_{u,L,\max,\infty}((P_{d}(X_{n}),\mathcal{A}(V_{n}))_{n\in\mathbb{N}})$.
\item [(3)]\hspace{0pt}The $C^{*}$-algebra 
 $C^{*}_{u,L,\max,\infty}
 ((P_{d}(X_{n}),\mathcal{A}(V_{n}))_{n\in\mathbb{N}})_{O}$ is defined to be the completion of $\mathbb{C}_{u,L,\infty}
 [(P_{d}(X_{n}),\mathcal{A}(V_{n}))_{n\in\mathbb{N}}]_{O}$ concerning the maximal norm.
    \end{enumerate}
\end{definition}

Analogue to the proof of Lemma 6.5 in \cite{CWY13}, one can verifies the $C^{*}$-algebras
 $C^{*}_{u,\phi, \infty}
((P_{d}(X_{n}),\mathcal{A}(V_{n}))_{n\in\mathbb{N}})_{O}$ and $C^{*}_{u,L,\phi, \infty}
((P_{d}(X_{n}),\mathcal{A}(V_{n}))_{n\in\mathbb{N}})_{O}$ 
are closed two-sided ideals of the $C^{*}$-algebras
$C^{*}_{u,\max,\infty}
((P_{d}(X_{n}),\mathcal{A}(V_{n}))_{n\in\mathbb{N}})$ and  $C^{*}_{u,L,\max, \infty}
((P_{d}(X_{n}),\mathcal{A}(V_{n}))_{n\in\mathbb{N}})$, respectively.

\begin{lemma}\label{lem. 6.11}
If $O:=(O_{n,y})_{y\in Y_{n},n\in\mathbb{N}}$ is a coherent system of open subsets of $\mathbb{R}_{+}\times V_{n},n\in \mathbb{N}$,  
then the $*$-homomorphism
\[(\lambda_{L,\phi,\text{red}})_{*}:K_{*}(C^{*}_{u,L,\phi, \infty}
((P_{d}(X_{n}),\mathcal{A}(V_{n}))_{n\in\mathbb{N}})_{O}) \rightarrow K_{*}(C^{*}_{u,L,\infty}
((P_{d}(X_{n}),\mathcal{A}(V_{n}))_{n\in\mathbb{N}})_{O})
\]restricted by the canonical quotient map on $K$-theory is an isomorphism.
\end{lemma} 
\begin{proof}[\rm{\textbf{Proof.}}]Since the $K$-theory of localization algebras $C^{*}_{u,L,\phi, \infty}
((P_{d}(X_{n}),\mathcal{A}(V_{n}))_{n\in\mathbb{N}})_{O}$ and $C^{*}_{u,L, \infty}
((P_{d}(X_{n}),\mathcal{A}(V_{n}))_{n\in\mathbb{N}})_{O}$ both
satisfy the homotopy invariance and the Mayer-Vietoris argument on $P_{d}(X_{n})$ for each $n\in\mathbb{N}$, it is sufficient to prove the Lemma \ref{lem. 6.11} for the 0-skeleton of each $P_{d}(X_{n})$. In this case, we can assume the metric of any two distinct points to be infinite as the topology induced by this metric remains the discrete topology. With this metric, the algebraic twisted Roe algebra is isomorphic to the direct product of the twisted Roe algebra of each vertex of $P_{d}(X_{n})$. Since the (maximal or reduced) twisted Roe algebra of single point is isomorphic to $\mathbb{C}\otimes\mathcal{K}\otimes \mathcal{A}(V_{n})\cong\mathcal{K}\otimes\mathcal{A}(V_{n})$, 
the quotient map $\lambda_{L,\phi,\text{red}}$ is an isomorphism for the 0-skeleton of $P_{d}(X_{n})$. The proof is completed.
\end{proof}
We can define an evaluation-at-zero map
\[e:C^{*}_{u, L, \max,\infty}((P_{d}(X_{n}),\mathcal{A}(V_{n}))_{n\in\mathbb{N}})_{O}\rightarrow
C^{*}_{u, \max,\infty}((P_{d}(X_{n}),\mathcal{A}(V_{n}))_{n\in\mathbb{N}})_{O}
\]
by $e(f)=f(0)$, which induces a $*$-homomorphism on $K$-theory
\[e_{*}:\lim_{d\rightarrow\infty}K_{*}(C^{*}_{u, L, \max,\infty}((P_{d}(X_{n}),\mathcal{A}(V_{n}))_{n\in\mathbb{N}})_{O})
\rightarrow\lim_{d\rightarrow\infty}K_{*}(C^{*}_{u,  \max,\infty}((P_{d}(X_{n}),\mathcal{A}(V_{n}))_{n\in\mathbb{N}})_{O}).\]

\subsection{$K$-theory of ideals of the maximal twisted Roe algebras at infinity}
In this section, we shall characterize ideals of the twisted algebra at infinity supported on certain coherent systems which are separated by subsets of $X_{n},n\in\mathbb{N}$, and prove that the evaluation map restricted to each ideal is an isomorphism. Our main result of this subsection is as follows.

\begin{theorem}\label{the. 6.12}
Let $O:=(O_{n,y})_{y\in Y_{n},n\in\mathbb{N}}$ be a coherent system of open subsets of $\mathbb{R}_{+}\times V_{n}, n\in\mathbb{N}$. If $O$ is  $(\varGamma,r)$-separate,
then the evaluation homomorphism on $K$-theory
\[e_{*}:\lim_{d\rightarrow\infty}K_{*}(C^{*}_{u, L, \max,\infty}((P_{d}(X_{n}),\mathcal{A}(V_{n}))_{n\in\mathbb{N}})_{O})
\rightarrow\lim_{d\rightarrow\infty}K_{*}(C^{*}_{u,  \max,\infty}((P_{d}(X_{n}),\mathcal{A}(V_{n}))_{n\in\mathbb{N}})_{O})\]is an isomorphism.
\end{theorem}

For each $n\in\mathbb{N}$, let $\varGamma_{n}$ be a subset of $Y_{n}$  and denote $\varGamma:=(\varGamma_{n})_{n\in\mathbb{N}}$. 
Let 
$(I_{\gamma})_{\gamma\in\varGamma_{n}}$ be a family of subsets
of $P_{d}(X_{n})$ for some $d\geqslant 0$ and $n\in\mathbb{N}$ such that: (1) $\gamma\in p_{n}(I_{\gamma})$ for each $\gamma\in\varGamma_{n},n\in\mathbb{N}$; (2)
$(p_{n}(I_{\gamma}))_{\gamma\in\varGamma_{n},n\in\mathbb{N}}$ is uniformly bounded, i.e. there exists $K>0$ such that $diameter(p_{n}(I_{\gamma}))\leqslant K$ for all $\gamma\in\varGamma_{n},n\in\mathbb{N}$. In the following discussion, we mainly consider the case where $I_{\gamma}$ is a \emph{cylinder}, i.e.
\[I_{\gamma}=B_{P_{d}(X_{n})}(p^{-1}_{n}(\gamma),S):=\{x\in P_{d}(X_{n})\mid d(x,p^{-1}_{n}(\gamma))\leqslant S\}\] for some common $S\geqslant 0$ for all $\gamma\in\varGamma_{n}, n\in\mathbb{N}$.
To prove Theorem \ref{the. 6.12}, we construct the following algebras.

\begin{definition}\label{def. 6.13}
 For any $\gamma\in\varGamma_{n},n\in\mathbb{N}$,
denote by $\mathcal{A}(O_{n,\gamma,\gamma})$ the C*-subalgebra
of $\mathcal{A}(V_{n})$ generated by the functions whose supports are contained in $O_{n,\gamma,\gamma}$.
Define $A_{\infty}[(I_{\gamma};\gamma\in\varGamma_{n})_{n\in\mathbb{N}}]$ to be the $\ast$-subalgebra of 
\begin{equation}\label{equ. 3}
	\frac{\prod_{n\in\mathbb{N}}(\bigoplus_{\gamma\in\varGamma_{n}}\mathbb{C}[I_{\gamma}]\hat\otimes\mathcal{A}(O_{n,\gamma,\gamma}))
	}{\bigoplus_{n\in\mathbb{N}}(\bigoplus_{\gamma\in\varGamma_{n}}\mathbb{C}[I_{\gamma}]\hat\otimes\mathcal{A}(O_{n,\gamma,\gamma}))}
\end{equation}
consisting of all 
elements of the form
$[(T^{(0)},\dots,T^{(n)},\dots)]$,
where
\[T^{(n)}=\bigoplus_{\gamma\in\varGamma_{n}}
T^{(n)}_{\gamma}\] with 
$T^{(n)}_{\gamma}\in\mathbb{C}[I_{\gamma}]\hat\otimes\mathcal{A}(O_{n,\gamma,\gamma})$ and when viewed as functions 
\[T^{(n)}_{\gamma}:(Z_{d,n}\times Z_{d,n})\cap (I_{\gamma}\times I_{\gamma}) \rightarrow \mathcal{A}(O_{n,\gamma,\gamma})\hat\otimes\mathcal{K}\]
the collection $(T^{(n)}_{\gamma})_{\gamma\in\varGamma_{n},n\in\mathbb{N}}$ satisfies the conditions in Definition \ref{def. 6.13} with uniform constants.\par 
By completing each algebraic Roe algebra $\mathbb{C}[I_{\gamma}]$ in (\ref{equ. 3}) into the reduced Roe algebra $C^{*}(I_{\gamma})$, we obtain a reduced completion of $*$-subalgebra $A_{\infty}[(I_{\gamma};\gamma\in\varGamma_{n})_{n\in\mathbb{N}}]$ and denoted by 
\[A^{*}_{\infty}((I_{\gamma};\gamma\in\varGamma_{n})_{n\in\mathbb{N}}).\]Since the collection
$(I_{\gamma})_{\gamma\in\varGamma_{n},n\in\mathbb{N}}$ has bounded geometry, we can also define the maximal completion of $*$-subalgebra $A_{\infty}[(I_{\gamma};\gamma\in\varGamma_{n})_{n\in\mathbb{N}}]$, denoted by \[A^{*}_{\max,\infty}((I_{\gamma};\gamma\in\varGamma_{n})_{n\in\mathbb{N}}),\] as the completion with respect to the maximal norm
\[\|T\|_{\max}:=\sup\{\|\phi(T)\|_{\mathcal{B}(H_{\phi})}	\mid\phi:	A_{\infty}[(I_{\gamma};\gamma\in\varGamma_{n})_{n\in\mathbb{N}}]\rightarrow\mathcal{B}(H_{\phi}), \text{a}\ \ast\text{-representation}\}.\]
\end{definition}

Note that the $C^{*}$-algebra $A^{*}_{\max,\infty}((I_{\gamma};\gamma\in\varGamma_{n})_{n\in\mathbb{N}})$ and $ A^{*}_{\infty}((I_{\gamma};\gamma\in\varGamma_{n})_{n\in\mathbb{N}})$ are not equal in general for the maximal norm may not equal to the supremum norm of all representations of $*$-subalgebra $A_{\infty}[(I_{\gamma};\gamma\in\varGamma_{n})_{n\in\mathbb{N}}]$. However, the collection $(p_{n}(I_{\gamma}))_{\gamma\in\varGamma_{n},n\in\mathbb{N}}$ is required uniformly bounded for all $\gamma\in\varGamma_{n},n\in\mathbb{N}$, which implies that the family of fiber spaces $(I_{\gamma})_{\gamma\in\varGamma_{n},n\in\mathbb{N}}$ of $P_{d}(X_{n})$ has equi-Property $A$. It follows from \cite{SW13} that the canonical homomorphism
\[\lambda:A^{*}_{\max,\infty}((I_{\gamma};\gamma\in\varGamma_{n})_{n\in\mathbb{N}})\rightarrow A^{*}_{\infty}((I_{\gamma};\gamma\in\varGamma_{n})_{n\in\mathbb{N}})\]
is an isomorphism.

\begin{definition}\label{def. 6.14}
Define 
$A_{L,\infty}[(I_{\gamma};\gamma\in\varGamma_{n})_{n\in\mathbb{N}}]$ to be the $\ast$-subalgebra of 
all bounded and uniformly norm-continuous functions
\[f:\mathbb{R}_{+}\rightarrow A_{\infty}[(I_{\gamma};\gamma\in\varGamma_{n})_{n\in\mathbb{N}}],\]
where $f(t)$ is of the form $f(t)=[({f}^{(0)}(t),\dots,{f}^{(n)}(t),\dots)]$ for all $t\in\mathbb{R}_{+}$, and
\[{f}^{(n)}(t)=\bigoplus_{\gamma\in\varGamma_{n}}{f}^{(n)}_{\gamma}(t)\]
such that the family of functions 
 $({f}^{(n)}_{\gamma}(t))_{\gamma\in\varGamma_{n},n\in\mathbb{N}, t\geqslant 0}$ satisfies the conditions in Definition \ref{def. 5.5}, and there exists a bounded function $R(t):\mathbb{R}_{+}\rightarrow\mathbb{R}_{+}$ with $\lim_{t\rightarrow\infty} R(t)=0$ such that
$(({f}^{(n)}_{\gamma}(t))(x,x')=0$ whenever $d(x,x')>R(t)$
for all $x,x'\in Z_{d,n}\cap I_{\gamma}, d\geqslant 0,\gamma\in\varGamma_{n}, n\in\mathbb{N}$ and $ t\in\mathbb{R}_{+}$.\end{definition}
\begin{definition}\label{def. 6.15}
Define $A^{*}_{L,\infty}((I_{\gamma};\gamma\in\varGamma_{n})_{n\in\mathbb{N}})$ to be the completion of 
$A_{L,\infty}[(I_{\gamma};\gamma\in\varGamma_{n})_{n\in\mathbb{N}}]$ with respect to the norm
\[\|f\|_{\max}:=\sup_{t\in\mathbb{R}_{+}}\|f(t)\|_{\max}.\]
(We remark that $f(t)$ is defined with the same norm closure in $C^{*}$-algebras  $A^{*}_{\max,\infty}((I_{\gamma};\gamma\in\varGamma_{n})_{n\in\mathbb{N}})$ and $A^{*}_{\infty}((I_{\gamma};\gamma\in\varGamma_{n})_{n\in\mathbb{N}})$, 
so we no longer distinguish the notation, but the uniform notation $\|f(t)\|_{\max}$. )\end{definition}

Recall that $I_{\gamma}$ is a \emph{cylinder}, i.e.
\[I_{\gamma}=B_{P_{d}(X_{n})}(p^{-1}_{n}(\gamma),S):=\{x\in P_{d}(X_{n})\mid d(x,p^{-1}_{n}(\gamma))\leqslant S\}\] for some common $S\geqslant 0$ for all $\gamma\in\varGamma_{n}, n\in\mathbb{N}$, which provides a specific characterization of the algebraic structure of the $C^{*}$-algebras $C^{*}_{u,  \phi,\infty}((P_{d}(X_{n}),\mathcal{A}(V_{n}))_{n\in\mathbb{N}})_{O}$ and $C^{*}_{u, L, \phi,\infty}((P_{d}(X_{n}),\mathcal{A}(V_{n}))_{n\in\mathbb{N}})_{O}$ as follows.

\begin{proposition}\label{prop. 6.16}
Suppose that $O:=(O_{n,y})_{y\in Y_{n},n\in\mathbb{N}}$ is a coherent system of open subsets of $\mathbb{R}_{+}\times V_{n}$, $n\in\mathbb{N}$, which is  $(\varGamma,r)$-separate for some
$\varGamma:=(\varGamma_{n})_{n\in\mathbb{N}}$ and $r>0$. Then
\begin{enumerate}
   \item [(1)] \hspace{0pt}$C^{*}_{u,  \phi,\infty}((P_{d}(X_{n}),\mathcal{A}(V_{n}))_{n\in\mathbb{N}})_{O}\cong \lim_{s\rightarrow\infty}A^{*}_{\infty}((B_{P_{d}(X_{n})}(p^{-1}_{n}(\gamma),S);\gamma\in\varGamma_{n})_{n\in\mathbb{N}})
;$ 
\item [(2)] \hspace{0pt}$C^{*}_{u, L, \phi,\infty}((P_{d}(X_{n}),\mathcal{A}(V_{n}))_{n\in\mathbb{N}})_{O}\cong \lim_{s\rightarrow\infty}A^{*}_{L,\infty}((B_{P_{d}(X_{n})}(p^{-1}_{n}(\gamma),S);\gamma\in\varGamma_{n})_{n\in\mathbb{N}}).$
\end{enumerate}
\end{proposition}
\begin{proof}[\rm {\textbf{Proof.}}]
Let
$T=[(T^{(0)},\dots,T^{(n)},\dots)]$ be an arbitrary element in $\mathbb{C}_{u,L,\infty}
[(P_{d}(X_{n}),\mathcal{A}(V_{n}))_{n\in\mathbb{N}}]_{O}.$ Since the coherent system $O:=(O_{n,y})_{y\in Y_{n},n\in\mathbb{N}}$ is $(\varGamma,r)$-separate and the map $p_{n}$ for each $n\in\mathbb{N}$ is surjective, there exist a sequence of open subsets 
$(O_{n,p_{n}(x),\gamma})_{\gamma\in \varGamma_{n}\cap B_{Y_{n}}(p_{n}(x), l_{n})}$ of $\mathbb{R}_{+}\times V_{n}$ for all $y\in Y_{n},n\in\mathbb{N}$ such that
\[\text{Supp}(T^{(n)}(x,x'))\subseteq O_{n,p_{n}(x)}=\bigsqcup_{\gamma\in \varGamma_{n}\cap B_{Y_{n}}(p_{n}(x),l_{n})}O_{n,p_{n}(x),\gamma}\]
for all $x,x'\in Z_{d,n},n\in\mathbb{N}$, and each open subset $O_{n,p_{n}(x),\gamma}$ is contained in $B_{\mathbb{R}_{+}\times V_{n}}(t_{p_{n}(x)}(\gamma)(s(\gamma)),r)$ for some $r>0$.
Then we have the decomposition
\[T^{(n)}(x,x')=\bigoplus_{\gamma\in \varGamma_{n}\cap B_{Y_{n}}(p_{n}(x),l_{n})}T_{\gamma}^{(n)}(x,x'),
\]
where 
\[T_{\gamma}^{(n)}(x,x')=T^{(n)}(x,x')|_{O_{n,p_{n}(x),\gamma}}
\in \mathcal{A}(O_{n,p_{n}(x),\gamma})\hat\otimes\mathcal{K}.\]
On the other hand, it follows from the support condition (5) in Definition \ref{def. 5.5} that there exists another $\tilde{r}>0$ such that
$\text{Supp}(T^{(n)}(x,x'))\subseteq B_{\mathbb{R}_{+}\times V_{n}}(t_{p_{n}(x)}(p_{n}(x))(s(p_{n}(x))),\tilde{r}).$
Therefore,
\[T_{\gamma}^{(n)}(x,x')=0\quad\text{whenever}\quad
d(t_{p_{n}(x)}(\gamma)(s(\gamma)),t_{p_{n}(x)}(p_{n}(x))(s(p_{n}(x))))>r+\tilde{r}\] for each $\gamma\in \varGamma_{n}\cap B_{Y_{n}}(p_{n}(x),l_{n})$.
Then there exist $S,\tilde{S}>0$ such that
\[d(t_{p_{n}(x)}(\gamma)(s(\gamma)),t_{p_{n}(x)}(p_{n}(x))(s(p_{n}(x))))\leqslant\bar{r}+r\Rightarrow 
d(p_{n}(x),\gamma)\leqslant\tilde{S}\Rightarrow d(x,p^{-1}_{n}(\gamma))\leqslant S,\]
where the first derivation follows because 
the section $s$ is locally a coarse embedding and the second is because the map $p_{n}$ is uniformly expansive for each $n\in\mathbb{N}$. Consequently, 
\[T_{\gamma}^{(n)}(x,x')=0\quad\text{ whenever}\quad d(x,p^{-1}_{n}(\gamma))\leqslant S\] for all $x,x'\in Z_{d,n}, n\in\mathbb{N}$ and $\gamma\in \varGamma_{n}\cap B_{Y_{n}}(p_{n}(x),l_{n})$.



Now, we take $N_{R}\in\mathbb{N}$ large enough depending only on the finite propagation $R$ of $T$, and define a equivalent class $U=[(U^{(0)},\dots,U^{(n)},\dots)]$ of sequence $(U^{(0)},\dots,U^{(n)},\dots)$ satisfying 
\begin{equation*}
	U^{(n)}=\left\{
	\begin{array}{ccc}
	&0  & \text{if}\ n<N_{R};\\
	&\oplus_{\gamma\in\varGamma_{n}}{U^{(n)}_{\gamma}}&\text{if}\ n\geqslant N_{R},
	\end{array}\right.
\end{equation*}
 and 
\[U^{(n)}_{\gamma}(x,x')=(t_{\gamma\ p_{n}(x)})_{*}T_{\gamma}^{(n)}(x,x')\in \mathcal{A}(O_{n,\gamma,\gamma})\hat\otimes\mathcal{K}\]for any $x,x'\in Z_{d,n}\cap B_{P_{d}(X_{n})}(p^{-1}_{n}(\gamma),S)$ with
$\gamma\in \varGamma_{n}\cap B_{Y_{n}}(p_{n}(x),l_{n}),n\in\mathbb{N}$.
It is clear that $U^{(n)}_{\gamma}$
 is well-defined, and we have that
 \begin{align*}
    U^{(n)}_{\gamma}&=[(U^{(n,0)}_{\gamma},U^{(n,1)}_{\gamma},\dots,U^{(n,n)}_{\gamma},\dots)]\\
&\in\oplus_{\gamma\in\varGamma_{n}}\mathbb{C}[B_{P_{d}(X_{n})}(p^{-1}_{n}(\gamma),S)]\hat\otimes\mathcal{A}(O_{n,\gamma,\gamma})\\
    &\subseteq \oplus_{\gamma\in\varGamma_{n}}C^{*}_{\max}(B_{P_{d}(X_{n})}(p^{-1}_{n}(\gamma),S))\hat\otimes\mathcal{A}(O_{n,\gamma,\gamma}).
 \end{align*} 
Hence, 
\[U\in A_{\infty}[(B_{P_{d}(X_{n})}(p^{-1}_{n}(\gamma),S);\gamma\in \varGamma_{n})_{n\in\mathbb{N}}],\]
which induces the following correspondence 
\begin{equation}\label{equ. 4}
\begin{split}
	\mathbb{C}_{u,\infty}
	[(P_{d}(X_{n}),\mathcal{A}(V_{n}))_{n\in\mathbb{N}}]_{O}&\rightarrow A_{\infty}[(B_{P_{d}(X_{n})}(p^{-1}_{n}(\gamma),S);\gamma\in \varGamma_{n})_{n\in\mathbb{N}}]\\
	T\quad\quad \quad\quad &\mapsto\quad\quad \quad\quad U
\end{split}\end{equation}

If we take the completion of $\mathbb{C}_{u,\infty}
	[(P_{d}(X_{n}),\mathcal{A}(V_{n}))_{n\in\mathbb{N}}]_{O}$ 
, then the map \eqref{equ. 4} can be extended to the following two $C^{*}$-isomorphisms
$$C^{*}_{u, \max,\infty}((P_{d}(X_{n}),\mathcal{A}(V_{n}))_{n\in\mathbb{N}})_{O}
\to  A^{*}_{\max,\infty}((B_{P_{d}(X_{n})}(p^{-1}_{n}(\gamma),S);\gamma\in \varGamma_{n})_{n\in\mathbb{N}})$$
and
$$C^{*}_{u,\infty}((P_{d}(X_{n}),\mathcal{A}(V_{n}))_{n\in\mathbb{N}})_{O}
\to  A^{*}_{\infty}((B_{P_{d}(X_{n})}(p^{-1}_{n}(\gamma),S);\gamma\in \varGamma_{n})_{n\in\mathbb{N}}).$$
It follows from Proposition 1.3 in \cite{SW13} that the $C^{*}$-algebra $ A^{*}_{\max,\infty}((B_{P_{d}(X_{n})}(p^{-1}_{n}(\gamma),S);\gamma\in \varGamma_{n})_{n\in\mathbb{N}})$ is isomorphic to $C^{*}$-algebra $ A^{*}_{\infty}((B_{P_{d}(X_{n})}(p^{-1}_{n}(\gamma),S);\gamma\in \varGamma_{n})_{n\in\mathbb{N}})$ since the collection $(B_{P_{d}(X_{n})}(p^{-1}_{n}(\gamma),S);\gamma\in \varGamma_{n})_{n\in\mathbb{N}}$ is uniformly coarsely equivalent to $\{p^{-1}_{n}(\gamma)\}_{\gamma\in \varGamma_{n},n\in\mathbb{N}}$ and the sequence $\{p^{-1}_{n}(\gamma)\}_{\gamma\in \varGamma_{n},n\in\mathbb{N}}$ has equi-Property A by Definition \ref{def. 2.4}. We finally obtain that $C^{*}_{u, \max,\infty}((P_{d}(X_{n}),\mathcal{A}(V_{n}))_{n\in\mathbb{N}})_{O}$ is isomorphic to $C^{*}_{u,\infty}((P_{d}(X_{n}),\mathcal{A}(V_{n}))_{n\in\mathbb{N}})_{O}$. Since $\|\cdot\|_{\phi}$ is an intermediate norm between the maximal norm $\|\cdot\|_{\max}$ and the reduced norm $\|\cdot\|$. As a result, the map \eqref{equ. 4} extends to an $\ast$-isomorphism from $C^{*}$-algebra
$C^{*}_{u, \phi,\infty}((P_{d}(X_{n}),\mathcal{A}(V_{n}))_{n\in\mathbb{N}})_{O}$
to $C^{*}$-algebra $ A^{*}_{\infty}((B_{P_{d}(X_{n})}(p^{-1}_{n}(\gamma),S);\gamma\in \varGamma_{n})_{n\in\mathbb{N}})$.\end{proof}

\begin{proof}[\rm\textbf{Proof of Lemma \ref{lem. 6.8}}]
It follows directly from Proposition \ref{prop. 6.16}.
\end{proof}

Let $e$ be the evaluation homomorphism from 
$A^{*}_{L,\infty}((I_{\gamma};\gamma\in\varGamma_{n})_{n\in\mathbb{N}})$ to $A^{*}_{\infty}((I_{\gamma};\gamma\in\varGamma_{n})_{n\in\mathbb{N}})$ given by $e(f)=f(0)$. Then we have the following result.

\begin{proposition}\label{prop. 6.17}
 For any $d\geqslant 0$. The map
\[e_{*}:\lim_{S\rightarrow\infty}K_{*}(A^{*}_{L,\infty}(
(B_{P_{d}(X_{n})}(p^{-1}_{n}(\gamma),S):\gamma\in\varGamma_{n})_{n\in\mathbb{N}}))\rightarrow \lim_{S\rightarrow\infty}K_{*}(A^{*}_{\infty}(
(B_{P_{d}(X_{n})}(p^{-1}_{n}(\gamma),S):\gamma\in\varGamma_{n})_{n\in\mathbb{N}}))\]
induced by the evaluation-at-zero map on $K$-theory is an isomorphism.
\end{proposition} 
\begin{proof}[\rm \textbf{{Proof.}}] Fix $S > 0$. From Definition \ref{def. 2.4} we have that
the collection $\{B_{P_{d}(X_{n})}(p^{-1}_{n}(\gamma),S)\}_{\gamma\in\varGamma_{n},n\in\mathbb{N}}$
is uniformly coarsely equivalent to the sequence of Rips complexes of fibers $\{P_{d}(p^{-1}_{n}(\gamma))\}_{\gamma\in\varGamma_{n},n\in\mathbb{N}}$ with scale $d\geqslant0$. Since the coarse assembly map is
coarsely invariant (see \cite{WY20}), it suffices to show
\[e_{*}:\lim_{S\rightarrow\infty}K_{*}(A^{*}_{L,\infty}(
(P_{d}(p^{-1}_{n}(\gamma)):\gamma\in\varGamma_{n})_{n\in\mathbb{N}}))\rightarrow \lim_{S\rightarrow\infty}K_{*}(A^{*}_{\infty}(
(P_{d}(p^{-1}_{n}(\gamma)):\gamma\in\varGamma_{n})_{n\in\mathbb{N}}))\] is an isomorphism, which follows from the fact that fiber spaces $\{p_{n}^{-1}(y)\}_{y\in Y_{n},n\in\mathbb{N}}$ has equi-Property A and Yu's celebrated theorem \cite{Yu00}. This finishes the proof. \end{proof}
\begin{proof}[\rm\textbf{Proof of Theorem \ref{the. 6.12}}]
By Definition \ref{def. 6.6}, \ref{def. 6.10} and Proposition \ref{prop. 6.16}, 
 \ref{prop. 6.17}, we have the following commutative diagram
\begin{center}
	\begin{tikzcd}
		\lim\limits_{d\rightarrow\infty}C^{*}_{u, L, \max,\infty}((P_{d}(X_{n}),\mathcal{A}(V_{n}))_{n\in\mathbb{N}})_{O} \arrow[d,"{\lambda_{L,\max,\phi}} "] \arrow[r,"e_{*}"] & \lim\limits_{d\rightarrow\infty}C^{*}_{u,  \max,\infty}((P_{d}(X_{n}),\mathcal{A}(V_{n}))_{n\in\mathbb{N}})_{O}\arrow[d,"{\lambda_{\max,\phi}}"] \\
		\lim\limits_{d\rightarrow\infty}C^{*}_{u, L, \phi,\infty}((P_{d}(X_{n}),\mathcal{A}(V_{n}))_{n\in\mathbb{N}})_{O} \arrow[d,"\cong"] \arrow[r,"e_{*}"] & \lim\limits_{d\rightarrow\infty}C^{*}_{u,  \phi,\infty}((P_{d}(X_{n}),\mathcal{A}(V_{n}))_{n\in\mathbb{N}})_{O}\arrow[d,"\cong"] \\
	\lim\limits_{d\rightarrow\infty}\lim\limits_{S\rightarrow\infty}A^{*}_{L,\infty}((B_{P_{d}(X_{n})}(p^{-1}_{n}(\gamma),S);\gamma\in\varGamma_{n})_{n\in\mathbb{N}})\arrow[r]& \lim\limits_{d\rightarrow\infty}\lim\limits_{S\rightarrow\infty}A^{*}_{\infty}((B_{P_{d}(X_{n})}(p^{-1}_{n}(\gamma),S);\gamma\in\varGamma_{n})_{n\in\mathbb{N}})
	\end{tikzcd}
\end{center}
which induces the following commutative diagram on $K$-theory
\begin{center}
	\begin{tikzcd}
	\lim\limits_{d\rightarrow\infty}K_{*}(C^{*}_{u, L, \max,\infty}((P_{d}(X_{n}),\mathcal{A}(V_{n}))_{n\in\mathbb{N}})_{O}) \arrow[d,"{(\lambda_{L,\max,\phi})}_{*} "] \arrow[r,"e_{*}"] & \lim\limits_{d\rightarrow\infty}K_{*}(C^{*}_{u,  \max,\infty}((P_{d}(X_{n}),\mathcal{A}(V_{n}))_{n\in\mathbb{N}})_{O})\arrow[d,"{(\lambda_{\max,\phi})}_{*}"] \\
	\lim\limits_{d\rightarrow\infty}K_{*}(C^{*}_{u, L, \phi,\infty}((P_{d}(X_{n}),\mathcal{A}(V_{n}))_{n\in\mathbb{N}})_{O}) \arrow[d,"\cong"] \arrow[r,"e_{*}"] & \lim\limits_{d\rightarrow\infty}K_{*}(C^{*}_{u,  \phi,\infty}((P_{d}(X_{n}),\mathcal{A}(V_{n}))_{n\in\mathbb{N}})_{O})\arrow[d,"\cong"] \\
	\lim\limits_{d\rightarrow\infty}\lim\limits_{S\rightarrow\infty}K_{*}(A^{*}_{L,\infty}((B_{P_{d}(X_{n})}(p^{-1}_{n}(\gamma),S);\gamma\in\varGamma_{n})_{n\in\mathbb{N}}))\arrow[r,"\cong"]& \lim\limits_{d\rightarrow\infty}\lim\limits_{S\rightarrow\infty}K_{*}(A^{*}_{\infty}((B_{P_{d}(X_{n})}(p^{-1}_{n}(\gamma),S);\gamma\in\varGamma_{n})_{n\in\mathbb{N}}))
\end{tikzcd}
\end{center}
\normalsize
Then Theorem \ref{the. 6.12} immediately follows from Lemma \ref{lem. 6.8}, \ref{lem. 6.11} and Proposition \ref{prop. 6.17}.
\end{proof}

\subsection{Gluing together the pieces}
In this subsection,  we glue together the smaller ideals to prove, our main result of this section, Theorem \ref{the. 6.1}.
\begin{proof}[\rm\textbf{Proof of Theorem 6.1.}] Let $r > 0$. Define
\[O^{(r)}_{n,y}:=\bigcup_{\gamma\in B_{Y_{n}}(y,{l_{n}})}B_{\mathbb{R}_{+}\times V_{n}}(t_{y}(\gamma)(s(\gamma)),r),\]
for all $y\in Y_{n},n\in\mathbb{N}$. Then $O^{(r)}:=(O^{(r)}_{n,y})_{y\in Y_{n},n\in\mathbb{N}}$
is a coherent system of open subsets of $\mathbb{R}_{+}\times V_{n}$.
For any $d\geqslant 0$, if $r<r'$ then $O^{(r)}_{n,y}\subseteq O^{(r')}_{n,y}$ so that 

\begin{align*}
	C^{*}_{u,  \max,\infty}((P_{d}(X_{n}),\mathcal{A}(V_{n}))_{n\in\mathbb{N}})_{O^{(r)}}&\subseteq C^{*}_{u,  \max,\infty}((P_{d}(X_{n}),\mathcal{A}(V_{n}))_{n\in\mathbb{N}})_{O^{(r')}},
	\\C^{*}_{u, L, \max,\infty}((P_{d}(X_{n}),\mathcal{A}(V_{n}))_{n\in\mathbb{N}})_{O^{(r)}}&\subseteq C^{*}_{u, L, \max,\infty}((P_{d}(X_{n}),\mathcal{A}(V_{n}))_{n\in\mathbb{N}})_{O^{(r')}}.
\end{align*}
By definitions, the maximal twisted algebra at infinity and its localization counterpart are the directed limits of the corresponding ideals, i.e.
\begin{align*}
	C^{*}_{u,  \max,\infty}((P_{d}(X_{n}),\mathcal{A}(V_{n}))_{n\in\mathbb{N}})&=\lim_{r\rightarrow\infty} C^{*}_{u,  \max,\infty}((P_{d}(X_{n}),\mathcal{A}(V_{n}))_{n\in\mathbb{N}})_{O^{(r)}},
	\\C^{*}_{u, L, \max,\infty}((P_{d}(X_{n}),\mathcal{A}(V_{n}))_{n\in\mathbb{N}})&=\lim_{r\rightarrow\infty}C^{*}_{u, L, \max,\infty}((P_{d}(X_{n}),\mathcal{A}(V_{n}))_{n\in\mathbb{N}})_{O^{(r)}}.
\end{align*}
So, it suffices to show that for any $r_{0}>0$ the following map
	\[e_{*}:\lim_{d\rightarrow\infty}K_{*}(\lim_{r<r_{0},r\rightarrow r_{0}}C^{*}_{u, L, \max,\infty}((P_{d}(X_{n}),\mathcal{A}(V_{n}))_{n\in\mathbb{N}})_{O^{(r)}}) \rightarrow 
	\lim_{d\rightarrow\infty}K_{*}(\lim_{r<r_{0},r\rightarrow r_{0}}C^{*}_{u,  \max,\infty}((P_{d}(X_{n}),\mathcal{A}(V_{n}))_{n\in\mathbb{N}})_{O^{(r)}})
\]
is an isomorphism.\par

Fix $y_{0}\in Y_{n}$, and take $y_{1}$ ouside the ball $B(y_{0},r_{0})$. Then, choose $y_{2}$ ouside the ball $B(y_{1},r_{0})$ with $y_{2}\neq y_{0}$. This process give rise to a countable sequence $\{y_{0},y_{1},y_{2},\dots\}$, denoted as $\varGamma^{(1)}_{n}$. Since the collection $(Y_{n})_{n\in\mathbb{N}}$ has uniform bounded geometry, there exists $N>0$ such that $\#B(y_{0},r_{0})<N$. It follows that there exists a positive integer $J_{r_{0}}>0$ independent of $n$ such that
\begin{enumerate}
   \item [(1)] $Y_{n}=\bigsqcup_{j=1}^{J_{r_{0}}}\varGamma_{n}^{(j)}$ for $J_{r_{0}}$ subspaces $\varGamma^{(j)}_{n}$ of $Y_{n}$;
   \item [(2)]  for any $y\in Y_{n}$ and any distinct element $\gamma,\gamma'\in\varGamma_{n}^{(j)}\cap 
B_{Y_{n}}(y,{l_{n}})$, \[d(t_{y}(\gamma)(s(\gamma)),t_{y}(\gamma')(s(\gamma')))>2r_{0}.\]
\end{enumerate}

For any $0<r<r_{0}$ and each $j\in\{1,2,\dots,J_{r_{0}}\}$, we define 
\[O^{(r,j)}_{n,y}:=\bigcup_{\gamma\in \varGamma_{n}^{(j)}\cap B_{Y_{n}}(y,{l_{n}})}B_{\mathbb{R}_{+}\times V_{n}}(t_{y}(\gamma)(s(\gamma)),r)
\]
for all $y\in Y_{n},l_{n}\geqslant 0,n\in\mathbb{N}$.  Then the collection
\[O^{(r,j)}:=(O^{(r,j)}_{n,y})_{y\in Y_{n},n\in\mathbb{N}}
\] is a coherent system for each $j\in\{1,2,\dots,J_{r_{0}}\}$ and we have that $O^{(r)}=\cup^{J_{r_{0}}}_{j=1}O^{(r,j)}$. Now, let $\varGamma^{(j)}=(\varGamma^{(j)}_{n})_{n\in\mathbb{N}}$, then the family of open subsets
$O^{(r,j)}$ or $O^{(r,j)}\cap(\cup_{i=1}^{j-1}O^{(r,j)})$ becomes $(\varGamma^{(j)},r)$-separate for each $j\in\{1,2,\dots,J_{r_{0}}\}$. Using the same argument 
in \cite{Yu00}, we have that 
\begin{align*}
&\lim_{r<r_{0},r\rightarrow r_{0}}C^{*}_{u, \max,\infty}((P_{d}(X_{n}),\mathcal{A}(V_{n}))_{n\in\mathbb{N}})_{O^{(r,j)}}+\lim_{r<r_{0},r\rightarrow r_{0}}C^{*}_{u, \max,\infty}((P_{d}(X_{n}),\mathcal{A}(V_{n}))_{n\in\mathbb{N}})_{\cup_{i=1}^{j-1}O^{(r,i)}}\\
&=\lim_{r<r_{0},r\rightarrow r_{0}}C^{*}_{u, \max,\infty}((P_{d}(X_{n}),\mathcal{A}(V_{n}))_{n\in\mathbb{N}})_{\cup_{i=1}^{j}O^{(r,i)}};\\
&\lim_{r<r_{0},r\rightarrow r_{0}}C^{*}_{u, \max,\infty}((P_{d}(X_{n}),\mathcal{A}(V_{n}))_{n\in\mathbb{N}})_{O^{(r,j)}}\cap\lim_{r<r_{0},r\rightarrow r_{0}}C^{*}_{u, \max,\infty}((P_{d}(X_{n}),\mathcal{A}(V_{n}))_{n\in\mathbb{N}})_{\cup_{i=1}^{j-1}O^{(r,i)}}\\
&=\lim_{r<r_{0},r\rightarrow r_{0}}C^{*}_{u, \max,\infty}((P_{d}(X_{n}),\mathcal{A}(V_{n}))_{n\in\mathbb{N}})_{O^{(r,j)}\cap(\cup_{i=1}^{j-1}O^{(r,i)})};\\
&\lim_{r<r_{0},r\rightarrow r_{0}}C^{*}_{u,L, \max,\infty}((P_{d}(X_{n}),\mathcal{A}(V_{n}))_{n\in\mathbb{N}})_{O^{(r,j)}}+\lim_{r<r_{0},r\rightarrow r_{0}}C^{*}_{u, \max,\infty}((P_{d}(X_{n}),\mathcal{A}(V_{n}))_{n\in\mathbb{N}})_{\cup_{i=1}^{j-1}O^{(r,i)}}\\
&=\lim_{r<r_{0},r\rightarrow r_{0}}C^{*}_{u,L, \max,\infty}((P_{d}(X_{n}),\mathcal{A}(V_{n}))_{n\in\mathbb{N}})_{\cup_{i=1}^{j}O^{(r,i)}}\\
&\lim_{r<r_{0},r\rightarrow r_{0}}C^{*}_{u,L, \max,\infty}((P_{d}(X_{n}),\mathcal{A}(V_{n}))_{n\in\mathbb{N}})_{O^{(r,j)}}\cap\lim_{r<r_{0},r\rightarrow r_{0}}C^{*}_{u, \max,\infty}((P_{d}(X_{n}),\mathcal{A}(V_{n}))_{n\in\mathbb{N}})_{\cup_{i=1}^{j-1}O^{(r,i)}}\\
&=\lim_{r<r_{0},r\rightarrow r_{0}}C^{*}_{u,L, \max,\infty}((P_{d}(X_{n}),\mathcal{A}(V_{n}))_{n\in\mathbb{N}})_{O^{(r,j)}\cap(\cup_{i=1}^{j-1}O^{(r,i)})}.
\end{align*}
\normalsize
Together with Lemma 2.4 in \cite{HRY93}, two Mayer-Vietoris sequences are derived. By connecting them with the associated evaluation maps, we have the following commutative diagram on $K$-theory 
\begin{center}
\quad\quad	\xymatrix{
&AL_{0}\ar[dd]\ar[rr]&&BL_{0}\ar[rr]\ar[dd]&&CL_{0}\ar[dd]\ar[ld]\\
CL_{1}\ar[dd]\ar[ru]&&BL_{1}\ar[ll]\ar[dd]&&AL_{1}\ar[dd]\ar[ll]\\&A_{0}\ar[rr]&&B_{0}\ar[rr]&&C_{0}\ar[ld]\\
C_{1}\ar[ru]&&B_{1}\ar[ll]&&A_{1}\ar[ll]}
\end{center}
where, for $\ast=0,1$,
\begin{align*}
		A_{*}&=K_{*}(C^{*}_{u, \max,\infty}((P_{d}(X_{n}),\mathcal{A}(V_{n}))_{n\in\mathbb{N}})_{O^{(r,j)}\cap(\cup_{i=1}^{j-1}O^{(r,i)})})\\
	B_{*}&=K_{*}(C^{*}_{u, \max,\infty}((P_{d}(X_{n}),\mathcal{A}(V_{n}))_{n\in\mathbb{N}})_{O^{(r,j)}})\oplus
	K_{*}(C^{*}_{u, \max,\infty}((P_{d}(X_{n}),\mathcal{A}(V_{n}))_{n\in\mathbb{N}})_{\cup_{i=1}^{j-1}O^{(r,i)}})\\
	C_{*}&=K_{*}(C^{*}_{u,\max,\infty}((P_{d}(X_{n}),\mathcal{A}(V_{n}))_{n\in\mathbb{N}})_{\cup_{i=1}^{j}O^{(r,i)}})\\
	AL_{*}&=K_{*}(C^{*}_{u,L, \max,\infty}((P_{d}(X_{n}),\mathcal{A}(V_{n}))_{n\in\mathbb{N}})_{O^{(r,j)}\cap(\cup_{i=1}^{j-1}O^{(r,i)})})\\
	BL_{*}&=K_{*}(C^{*}_{u,L, \max,\infty}((P_{d}(X_{n}),\mathcal{A}(V_{n}))_{n\in\mathbb{N}})_{O^{(r,j)}})\oplus
	K_{*}(C^{*}_{u,L, \max,\infty}((P_{d}(X_{n}),\mathcal{A}(V_{n}))_{n\in\mathbb{N}})_{\cup_{i=1}^{j-1}O^{(r,i)}})\\
	CL_{*}&=K_{*}(C^{*}_{u,L, \max,\infty}((P_{d}(X_{n}),\mathcal{A}(V_{n}))_{n\in\mathbb{N}})_{\cup_{i=1}^{j}O^{(r,i)}})
\end{align*}
 Now, Theorem \ref{the. 6.1} follows from Theorem \ref{the. 6.12} and the five lemma.\end{proof}

\section{Proof of the main theorem}\label{section 7}
In this section, we shall complete the proof of Theorem \ref{the. 1.1}. To do so, we will construct the Dirac maps $\alpha$ and $\alpha_{L}$, to establish the following commutative diagram. 

\begin{center}
	\begin{tikzcd}
\lim\limits_{d\rightarrow\infty}K_{*+1}(C^{*}_{u, L, \max,\infty}(P_{d}(X_{n}))_{n\in\mathbb{N}}) \arrow[d,"(\beta_{L})_{*}"] \arrow[r,"e_{*}"] & \lim\limits_{d\rightarrow\infty}K_{*+1}(C^{*}_{u, \max,\infty}(P_{d}(X_{n}))_{n\in\mathbb{N}})\arrow[d,"\beta_{*}"] \\
	\lim\limits_{d\rightarrow\infty}K_{*}(C^{*}_{u, L, \max,\infty}((P_{d}(X_{n}),\mathcal{A}(V_{n}))_{n\in\mathbb{N}})) \arrow[d,"(\alpha_{L})_{*}"] \arrow[r,"e^{\mathcal{A}}_{*}"] & 	\lim\limits_{d\rightarrow\infty}K_{*}(C^{*}_{u, \max,\infty}((P_{d}(X_{n}),\mathcal{A}(V_{n}))_{n\in\mathbb{N}}))\arrow[d,"\alpha_{*}"] \\
\lim\limits_{d\rightarrow\infty}K_{*}(C^{*}_{u, L, \max,\infty}((P_{d}(X_{n}),\mathcal{SK}(\mathcal{L}^{2}(E_{n})))_{n\in\mathbb{N}})) \arrow[r,"e_{*}"] & 	\lim\limits_{d\rightarrow\infty}K_{*}(C^{*}_{u, \max,\infty}((P_{d}(X_{n}),\mathcal{SK}(\mathcal{L}^{2}(E_{n})))_{n\in\mathbb{N}})).
	\end{tikzcd}
\end{center}
In subsection \ref{subsection. 7.2}, we establish a geometric analogue of the finite-dimensional Bott periodicity introduced by Higson, Kasparov, and Trout in \cite{HKT98}. Theorem \ref{the. 1.1} is derived from the reduction of the maximal coarse Baum-Connes conjecture to the maximal twisted coarse Baum-Connes conjecture for a sequence of metric spaces $(X_{n})_{n\in\mathbb{N}}$ with uniform bounded geometry which admits an A-by-FCE coarse fibration structure.

\subsection{Constructions of the Dirac maps $\alpha$ and $\alpha_{L}$}\label{subsection. 7.1}
Recall first that 
$(X_{n})_{n\in\mathbb{N}}$ is a sequence of metric spaces with uniform bounded geometry which admit an A-by-FCE coarse fibration structure. For each $d\geqslant 0$, $F_{d,n}$ denotes a countable dense subset of $P_{d}(Y_{n})$, while $Z_{d,n}$ is the corresponding countable dense subset of $P_{d}(X_{n})$. For any $n\in\mathbb{N}$, $V_{n}$ is defined to be the finite-dimensional affine subspace of $H$ spanned by
$t_{y}(z)(s(z))$ for all $z\in B_{Y_{n}}(y,l_{n})$ with $y\in Y_{n}$.
Take a linear span of $V_{n}$ and $0$ we obtain a finite-dimensional linear space of Hilbert space $H$ containing $V_{n}$, i.e.
\[E_{n}:=\text{linear-span}\{V_{n},0\}.\]
Denote by
\[\mathcal{L}^{2}_{n}:=L^{2}(E_{n},\text{Cliff}{(E_{n}))}\]
the $\mathbb{Z}/2$-graded complex Hilbert space of square-integrable functions from $E_{n}$ into Cliff$(E_{n})$, where $E_{n}$ is endowed with the Lebesgue measure induced from the inner product on $E_{n}$, and 
the grading of $\mathcal{L}^{2}_{n}$ is inherited from the Clifford algebra $\text{Cliff}{(E_{n})}$.
Denote by $s(E_{n})$ the Schwartz subspace of $\mathcal{L}^{2}_{n}$.

$\bullet$ The Dirac operator $D_{E_{n}}$ is the unbounded operator on $\mathcal{L}^{2}_{n}$ with domain $s(E_{n})$, defined by the formula
\[(D_{E_{n}}\xi )(v)=\sum^{n}_{i=1}(-1)^{degree{(\xi)}}\frac{\partial\xi}{\partial x_{i}}(v)\cdot e_{i},\]
where $\xi\in s(E_{n})\subseteq\mathcal{L}^{2}_{n}$,  $v\in E_{n}$,  $\{e_{1},\dots,e_{n}\}$ is an orthogonal basis for $E_{n}$, and $\{x_{1},\dots,x_{n}\}$ is the dual coordinate system on $E_{n}$. Denote by $\mathcal{K}(\mathcal{L}^{2}_{n})$ the graded $C^{*}$-algebra of compact operators on $\mathcal{L}^{2}_{n}$ for all $n\in\mathbb{N}$. Then the Dirac operator can be viewed as a grading-degree one, essentially a self-adjoint multiplier of $\mathcal{K}(\mathcal{L}^{2}_{n})$, with domain those compact operators which map $\mathcal{L}^{2}_{n}$
 into the Schwartz space.
 
$\bullet$ The Clifford operator $C_{E_{n},v_{0}}$ of $E_{n}$ at basis point $v_{0}\in E_{n}$ is an unbounded operator on $\mathcal{L}^{2}_{n}$ with domain $s(E_{n})$,
defined by the formula
\[(C_{E_{n},v_{0}}\xi)(v)=(v-v_{0})\cdot\xi(v)\]
for all $\xi\in \mathcal{L}^{2}_{n}$, $v\in E_{n}$, where the multiplication $\cdot$ is the Clifford multiplication. Note that the Clifford operator $C_{E_{n},v_{0}}$ depends on the choice of the base point $v_{0}$. Sometimes we abbreviate the Clifford operator $C_{E_{n},0}$ at base point $0$ as $C_{E_{n}}$.

$\bullet$ The Bott-Dirac operator $B_{E_{n},v_{0}}$ of $E_{n}$ at basis point $v_{0}\in E_{n}$ is the unbounded symmetric operator on $\mathcal{L}^{2}_{n}$ with domain again $s(E_{n})$, defined by the formula
\[B_{E_{n},v_{0}}=D_{E_{n}}+C_{E_{n},v_{0}}.\]
One can also show that $B_{E_{n},v_{0}}$ is essentially self-adjoint, and has compact
resolvent. In addition, the square of $B_{E_{n},v_{0}}$ admits an orthonormal eigenbasis of Schwartz-class
functions, with eigenvalues $2n (n=0, 1, \dots)$, each of finite multiplicity. Hence
$B_{E_{n},v_{0}}$ admits an orthonormal eigenbasis of Schwartz-class functions, with eigenvalues $\pm\sqrt{2n}$, each of finite multiplicity. The kernel of $B_{E_{n},v_{0}}$ is spanned by $\exp(\frac{1}{2}\|e-v_{0}\|^{2})\cdot 1$. This fact ensures the constructions in the proof of Theorem \ref{the. 7.8}. Consult \cite{WY20} for more related results.

\begin{definition}\label{def. 7.1}
For any $d\geqslant 0$, define
$\mathbb{C}_{u,\infty}
[(P_{d}(X_{n}),\mathcal{SK}(\mathcal{L}^{2}_{n}))_{n\in\mathbb{N}}]$ to be the set of all equivalence classes $T=[(T^{(0)},\dots,T^{(n)},\dots)]$ of sequences $(T^{(0)},\dots,T^{(n)},\dots)$ described as follows
\begin{enumerate}
	\item [(1)]\hspace{0pt}for each $n\in\mathbb{N}$, 
$T^{(n)}$ is a bounded function from $Z_{d,n}\times Z_{d,n}$ to $\mathcal{SK}(\mathcal{L}^{2}_{n})\hat{\otimes}\mathcal{K}$ such that 
\[\sup_{n\in\mathbb{N}} \sup_{x,x'\in Z_{d,n}}\|T^{(n)}(x,x')\|_{\mathcal{SK}(\mathcal{L}^{2}_{n})\hat{\otimes}\mathcal{K}}<\infty\]
 and the collection $(T^{(n)}(x,x'))_{x,x'\in Z_{d,n},n\in\mathbb{N}}$ is equi-continuous from $\mathbb{R}$ to $\mathcal{K}(\mathcal{L}^{2}_{n})\hat{\otimes}\mathcal{K}$;
\item [(2)]\hspace{0pt}for each $n\in\mathbb{N}$ and any bounded subset $B\subset P_{d}(X_{n})$, the set
\[\{(x,x')\in (B\times B)\cap (Z_{d,n}\times Z_{d,n})\mid T^{(n)}(x,x')\neq 0 \}<\infty\]
is finite; 
\item [(3)]\hspace{0pt}there exists $L>0$ such that 
\[\#\{x'\in Z_{d,n}\mid T^{(n)}(x,x')\neq 0\}<L \quad\text{and}\quad 
\#\{x'\in Z_{d,n}\mid T^{(n)}(x',x)\neq 0\}<L\]
for all $x\in Z_{d,n},n\in\mathbb{N};$
\item [(4)]\hspace{0pt}there exists $R>0$ such that $T^{(n)}(x,x')=0$ whenever $d(x,x')>R$ for $x,x'\in Z_{d,n},n\in\mathbb{N}.$
\end{enumerate}
\end{definition}
 The equivalence relation $\thicksim$ on these sequences is defined by 
\[(T^{(0)},\dots,T^{(n)},\dots)\thicksim(S^{(0)},\dots,S^{(n)},\dots)\] if and only if
\[\lim_{n\rightarrow \infty}\sup_{x,x'\in Z_{d,n}}{\|T^{(n)}(x,x')-S^{(n)}(x,x')\|}_{\mathcal{SK}(\mathcal{L}^{2}_{n})\hat{\otimes}\mathcal{K}}=0.\]

 Take two arbitrary classes ${[(T^{(0)},\dots,T^{(n)},\dots)]}$ and $[(S^{(0)},\dots,S^{(n)},\dots)]$ in 
$\mathbb{C}_{u,\infty}
[(P_{d}(X_{n}),\mathcal{SK}(\mathcal{L}^{2}_{n})_{n\in\mathbb{N}}],$ 
their product is defined to be
\[[(T^{(0)},\dots,T^{(n)},\dots)][(S^{(0)},\dots,S^{(n)},\dots)]=[((TS)^{(0)},\dots,
(TS)^{(n)},\dots)],\]
where there exists a sufficiently large $N\in\mathbb{N}$ depending on the propagations of these two sequences satisfying that $(TS)^{(n)}=0$ for all $n<N$ and 
\[(TS)^{(n)}(x,x')=\sum_{z\in Z_{d,n}}
(T^{(n)}(x,z))\cdot ((t_{p_{n}(x)p_{n}(z)})_{*}(S^{(n)}(z,x')))\]
for all $x,x'\in Z_{d,n}$ and $n\geqslant N$.

\begin{remark} Here are some explanations of the isomorphism $(t_{p_{n}(x)p_{n}(z)})_{*}$ from $\mathcal{SK}(\mathcal{L}_{n}^{2})\hat\otimes\mathcal{K}$ to $\mathcal{SK}(\mathcal{L}_{n}^{2})\hat\otimes\mathcal{K}$. We start with the isometry between two affine subspaces.\par 
Let $H$ be an infinite-dimensional separable Hilbert space. Denote by $V_{a}$ and $ V_{b}$ the finite-dimensional affine subspaces of $H$. Let $V_{a}^{0}$ and $V_{b}^{0}$ be the linear subspaces of $H$ consisting of differences of elements of $V_{a}$ and $V_{b}$, respectively. 
If $t_{ba}:V_{a}\rightarrow V_{b}$ is an isometric bijection, there exists a canonical unitary 
\[U_{t_{ba}}:L^{2}(V_{a},\text{Cliff}(V_{a}^{0}))\rightarrow L^{2}(V_{b},\text{Cliff}(V_{b}^{0}))\]
defined by the formula 
\[(U_{t_{ba}}\xi)(v_{a})=(t_{0})_{*}(\xi(t_{ba}^{-1}v_{a})),\]
where $(t_{0})_{*}:\text{Cliff}(V_{a}^{0})\rightarrow\text{Cliff}(V_{b}^{0})$ is induced by isometric bijection between linear subspaces:
\[t_{0}:=t_{ba}|_{V_{a}^{0}}:V_{a}^{0}\rightarrow V_{b}^{0}.\] 
Moreover, the unitary $U_{t_{ba}}$ further induces a $*$-isomorphism 
\[(t_{ba})_{*}:\mathcal{K}(L^{2}(V_{a},\text{Cliff}(V_{a}^{0})))\rightarrow \mathcal{K}(L^{2}(V_{b},\text{Cliff}(V_{b}^{0})))\] defined by conjugation with the unitary. In the case when $t_{ba}$ is not bijective, there exists an affine isometry (not necessarily unitary) \[T_{ba}:L^{2}(V_{a},\text{Cliff}(V_{a}^{0}))\rightarrow L^{2}(V_{b},\text{Cliff}(V_{b}^{0}))\cong L^{2}(V_{ba}^{0},\text{Cliff}(V_{ba}^{0})\hat\otimes L^{2}(V_{a},\text{Cliff}(V_{a}^{0}) \]
given by the formula
\[(T_{ba}\xi)(v_{ba}+v_{a})=\pi^{-\frac{n}{4}}\cdot e^{-\frac{\|v_{ba}\|^{2}}{2}}\cdot\xi(v_{a})\]
where $v_{a}\in V_{a}$, $v_{ba}\in V_{ba}^{0}$, and $n$ is the difference of the dimensions of $V_{b}$ and $V_{a}$ (Cf. \cite{HKT98, WY20}). 

We now come back to the case of isometry $(t_{p_{n}(x)p_{n}(z)})_{*}$, with the assumption in Definition \ref{def. 7.1} above, we denote
\begin{align*}
	W_{p_{n}(x)}&=\text{affine-span}\{t_{p_{n}(x)}(w)(s(w))\mid w\in B(p_{n}(x),l_{n})\cap B(p_{n}(z),l_{n})\}=t_{p_{n}(x)p_{n}(z)}
	(W_{p_{n}(z)}),\\
		W_{p_{n}(z)}&=\text{affine-span}\{t_{p_{n}(z)}(w)(s(w))\mid w\in B(p_{n}(x),l_{n})\cap B(p_{n}(z),l_{n})\}=t_{p_{n}(z)p_{n}(x)}
		(W_{p_{n}(x)}),
\end{align*}
where the affine isometry $t_{p_{n}(x)p_{n}(z)}=t_{p_{n}(x)}(w)\circ t^{-1}_{p_{n}(z)}(w)$ for all $w\in B(p_{n}(x),l_{n})\cap B(p_{n}(z),l_{n})$
maps the affine subspaces $W_{p_{n}(z)}$ onto $W_{p_{n}(x)}$. We have the direct sum decomposition 
\[E_{n}=W_{p_{n}(x)}^{\bot}\oplus W_{p_{n}(x)}=W_{p_{n}(z)}^{\bot}\oplus W_{p_{n}(z)},\] 
where $W_{p_{n}(x)}^{\bot}$ and $W_{p_{n}(z)}^{\bot}$ represent the linear orthogonal completments of $W_{p_{n}(x)}$ and $W_{p_{n}(z)}$ in $E_{n}$, respectively. Accordingly, we have that
\[\mathcal{L}^{2}_{n}=L^{2}(W_{{p_{n}(x)}},\text{Cliff}(W_{p_{n}(x)}^{0}))\hat\otimes L^{2}(W_{{p_{n}(x)}^{\perp}},\text{Cliff}(W_{{p_{n}(x)}^{\perp}}^{0})).\]
Since the isometric bijection $t_{p_{n}(x)p_{n}(z)}$ canonically induces a unitary $U_{p_{n}(x)}$ on $L^{2}(W_{{p_{n}(x)}},\text{Cliff}(W_{p_{n}(x)}^{0}))$, we only need to choose a unitary operator
\[U_{p_{n}(x)p_{n}(z)}:W_{p_{n}(z)}^{\bot}\rightarrow W_{p_{n}(x)}^{\bot}.\]
Then
$U_{p_{n}(x)p_{n}(z)}\oplus t_{p_{n}(x)p_{n}(z)}$ 
becomes an affine isometry from $E_{n}$ onto $E_{n}$. By conjugation with the unitary, we finally obtain the $*$-isomorphism
\[(t_{p_{n}(x)p_{n}(z)})_{*}:=(U_{p_{n}(x)p_{n}(z)}\oplus t_{p_{n}(x)p_{n}(z)})_{*}\hat\otimes 1:\mathcal{SK}(\mathcal{L}^{2}_{n})\hat\otimes\mathcal{K}\rightarrow\mathcal{SK}(\mathcal{L}^{2}_{n})\hat\otimes\mathcal{K}.\]
Moreover, in Rips complexes for each $d\geqslant 0$, if $p_{n}(x),p_{n}(z)\in F_{d,n}\subseteq P_{d}(Y_{n}),n\in\mathbb{N}$, with $p_{n}(x)\in\text{Star}(p_{n}(\bar{x}))$ and $p_{n}(z)\in\text{Star}(p_{n}(\bar{z}))$, we define $(t_{p_{n}(x)p_{n}(z)})_{*}=(t_{p_{n}(\bar{x})p_{n}(\bar{z})})_{*}$.

\end{remark}

The $\ast$-structure for $\mathbb{C}_{u,\infty}
[(P_{d}(X_{n}),\mathcal{SK}(\mathcal{L}^{2}_{n}))_{n\in\mathbb{N}}]$ is defined by 
\[{[(T^{(0)},\dots,T^{(n)},\dots)]}^{*}=[({(T^{*})}^{(0)},\dots,{(T^{*})}^{(n)},\dots)],\]
where 
\[{(T^{*})}^{(n)}(x,x')=(t_{p_{n}(x)p_{n}(x')})_{*}
((T^{(n)}(x',x))^{*})\]
for all but finitely many $n$, and $0$ otherwise.

Now, $\mathbb{C}_{u,\infty}
[(P_{d}(X_{n}),\mathcal{SK}(\mathcal{L}^{2}_{n}))_{n\in\mathbb{N}}]$ is made into a $\ast$-algebra by using the additional usual matrix operations. 
Define
\[C^{*}_{u,\max,\infty}((P_{d}(X_{n}),\mathcal{SK}(\mathcal{L}^{2}_{n}))_{n\in\mathbb{N}})\]
to be the completion of $\mathbb{C}_{u,\infty}
[(P_{d}(X_{n}),\mathcal{SK}(\mathcal{L}^{2}_{n}))_{n\in\mathbb{N}}]$ with respect to the maximal norm 
\[\|T\|_{\max}:=\sup\{\|\phi(T)\|_{\mathcal{B}(H_{\phi})}	\mid\phi:	\mathbb{C}_{u,\infty}
[(P_{d}(X_{n}),\mathcal{SK}(\mathcal{L}^{2}_{n}))_{n\in\mathbb{N}}]\rightarrow\mathcal{B}(H_{\phi}), \text{a}\ast\text{-representation}\}.\]

\begin{definition}\label{{def. 7.3}}
 For each $d\geqslant 0$, define $\mathbb{C}_{u,L,\infty}[(P_{d}(X_{n}),\mathcal{SK}
(\mathcal{L}^{2}_{n}))_{n\in\mathbb{N}}]$ to be the
$\ast$-algebra of all bounded and uniformly norm-continuous functions
\[f:\mathbb{R}_{+}\rightarrow\mathbb{C}_{u,\infty}[(P_{d}(X_{n}),\mathcal{SK}
(\mathcal{L}^{2}_{n}))_{n\in\mathbb{N}}] \]
such that $f(t)$ is of the form $f(t)=[({f}^{(0)}(t),\dots,{f}^{(n)}(t),\dots)]$ for all $t\in\mathbb{R}_{+}$, where the family of functions $({f}^{(n)}(t))_{n\in\mathbb{N}, t\geqslant 0}$ satisfy the conditions in Definition \ref{def. 7.1} with uniform constants,
and there exists a bounded function $R(t):\mathbb{R}_{+}\rightarrow\mathbb{R}_{+}$ with $\lim_{t\rightarrow\infty} R(t)=0$ such that
\[({f}^{(n)}(t))(x,x')=0\quad \text{whenever}\quad d(x,x')>R(t)\]
for all $x,x'\in Z_{d,n},n\in\mathbb{N}, t\in\mathbb{R}_{+}$.
Define \[C^{*}_{u,L,\max,\infty}((P_{d}(X_{n}),\mathcal{SK}
(\mathcal{L}^{2}_{n}))_{n\in\mathbb{N}})\]to be the completion of $\mathbb{C}_{u,L,\infty}[(P_{d}(X_{n}),\mathcal{SK}
(\mathcal{L}^{2}_{n}))_{n\in\mathbb{N}}]$
with respect to the norm
\[\|f\|_{\max}:=\sup_{t\in\mathbb{R}_{+}}\|f(t)\|_{\max}.\]
\end{definition}

Recall that for any $y\in Y_{n}$ and $k\geqslant 0$, 
\[W_{k}(y):=\text{affine-span}\{t_{y}(z)(s(z))\mid
z\in F_{d,n}\cap
B_{P_{d}(Y_{n})}(y,k)\}\]
is an affine subspace of $E_{n}$ for some suitable large $n\in\mathbb{N}$. Denote by $(W_{k}(y))^{\bot}=E_{n}\ominus W_{k}(y)$ the orthogonal complement of $W_{k}(y)$ in $E_{n}$. Then the Bott-Dirac operator on $(W_{k}(y))^{\bot}$ is defined to be
\[B_{(W_{k}(y))^{\bot}}=C_{(W_{k}(y))^{\bot}}+D_{(W_{k}(y))^{\bot}},\]
where $C_{(W_{k}(y))^{\bot}}$ is the Clifford operator of $(W_{k}(y))^{\bot}$ with base point $0$, while $D_{(W_{k}(y))^{\bot}}$ is the Dirac operator on $(W_{k}(y))^{\bot}$.
For every $k\geqslant 0$ and  $t\geqslant 1$, we define
\[\theta^{k}_{t}(y):\mathcal{A}(W_{k}(y))\hat\otimes\mathcal{K}\rightarrow\mathcal{SK}(\mathcal{L}^{2}_{n})\hat\otimes\mathcal{K}\]
by the formula
\[(\theta^{k}_{t}(y))(g\hat\otimes h\hat\otimes b)=g_{t}(X\hat\otimes 1+1\hat\otimes (B_{(W_{k}(y))^{\bot}}+D_{W_{k}(y)}))(1\hat\otimes M_{h_{t}})\hat\otimes b
\]
for any $g\in\mathcal{S}$, $h\in\mathcal{C}(W_{k}(y))$, $b\in\mathcal{K}$, where $g_{t}(s)=g(t^{-1}s)$
for all $t\geqslant 1$ and $s\in\mathbb{R}$, the function $h_{t}\in\mathcal{C}(W_{k}(y)) $ is defined by 
\[h_{t}(v)=h(t_{y}(y)s(y)+t^{-1}(v-t_{y}(y)s(y)))\] for any $t\geqslant 1$ and $v\in W_{k}(y)$. In addition,
 $M_{h_{t}}$ acts on $L^{2}(W_{k}(y),\text{Cliff}(W^{0}_{k}(y)))$ by pointwise multiplication, 
 we define
\[(M_{h_{t}}\xi )(v)=h_{t}(v)\cdot\xi(v)
\] for all $\xi\in L^{2}(W_{k}(y),\text{Cliff}(W^{0}_{k}(y)))$ and $v\in W_{k}(y)$. Note that $\theta^{k}_{t}(y)$
is defined on
a dense subalgebra of $A(W_{k}(y))\hat\otimes \mathcal{K}$.

\begin{definition}[The Dirac map]\label{{def. 7.4}} Let $d\geqslant 0$. For each $t\in[1,\infty)$,  define a map
\[\alpha_{t}:\mathbb{C}_{u,\infty}[(P_{d}(X_{n}),\mathcal{A}(V_{n}))_{n\in\mathbb{N}}]\rightarrow\mathbb{C}_{u,\infty}[(P_{d}(X_{n}),\mathcal{SK}(\mathcal{L}^{2}_{n}))_{n\in\mathbb{N}}]\]
by the formula
\[\alpha_{t}(T)=[((\alpha_{t}(T))^{(0)},\dots,(\alpha_{t}(T))^{(n)},\dots)]\]
for every $T=[(T^{(0)},\dots,T^{(n)},\dots)]\in \mathbb{C}_{u,\infty}[(P_{d}(X_{n}),\mathcal{A}(V_{n}))_{n\in\mathbb{N}}]$ with \[(\alpha_{t}(T))^{n}(x,x')=({\theta_{t}^{k_{0}}}(x))(T_{1}^{(n)}(x,x')),\] where $k_{0}$ is a non-negative number such that, for any pair $x$ and $x'$ in $Z_{d,n}$ with $n\in\mathbb{N}$, there exists $T_{1}(x,x')\in\mathcal{A}(W_{k_{0}}(y))\hat\otimes\mathcal{K}$ satisfying \[T^{(n)}(x,x')=(\beta_{V_{n},W_{k_{0}}(y)}\hat{\otimes}1)({T_{1}}^{(n)}(x,x')).\]\end{definition}

\begin{remark}\label{rem. 7.5} 
(1) For each $t\in[1,\infty)$, the map $\alpha_{t}$ is well-defined as a consequence of the Rellich Lemma and ellipticity of Dirac operators.

(2)  Based on the construction of the Bott map  in section \ref{section 5} and the definition of the Dirac map mentioned above, we have the following asymptotically commutative diagram
\small
\begin{center}
\begin{tikzcd}
\mathcal{S}\hat\otimes\mathcal{K}\cong\mathcal{A}(W_{{0}}(y))\hat\otimes\mathcal{K} \arrow[rr,"\beta_{W_{k_{1}}(y),W_{0}(y)}\hat\otimes 1"] \arrow[rrrrrrdd,"\theta^{0}_{t}(y)"description] &  & \mathcal{A}(W_{{k_{1}}}(y))\hat\otimes\mathcal{K}  \arrow[rr,"\beta_{W_{k_{2}}(y),W_{k_{1}}(y)}\hat\otimes 1"] \arrow[rrrrdd, "\theta^{k_{1}}_{t}(y)"description] &  & \mathcal{A}(W_{{k_{2}}}(y))\hat\otimes\mathcal{K}  \arrow[r] \arrow[rrdd,"\theta^{k_{2}}_{t}(y)"description] & \dots \arrow[r] \arrow[rdd, dashed] & \mathcal{A}(V_{n})\hat\otimes\mathcal{K}  \arrow[dd] \\
	&  &                            &  &                         &                         	& \\ &                            &  &                         &             &       &\mathcal{SK}(\mathcal{L}^{2}_{n})\hat\otimes\mathcal{K}          
\end{tikzcd}
\end{center}
\normalsize
Note that the asymptotic morphism
\[\beta_{W_{k_{j}}(y),W_{k_{j-1}}(y)}\hat\otimes 1:\mathcal{A}(W_{{k_{j-1}}}(y))\hat\otimes\mathcal{K}\rightarrow \mathcal{A}(W_{{k_{j}}}(y))\hat\otimes\mathcal{K}\] in the horizontal direction is defined by the formula
\[
g\hat\otimes h\hat\otimes b\mapsto g_{t}(X\hat\otimes 1+1\hat\otimes C_{W_{k_{j-1}}(y),W_{k_{j}}(y)})(1\hat\otimes M_{h_{t}})\hat\otimes b\]
for each $j\in\mathbb{N}$ and $t\in [1,\infty)$, where $g_{t}$ performs a functional calculus on the Clifford operator $C_{W_{k_{j-1}}(y),W_{k_{j}}(y)}$ causing the function originally defined on $\mathbb{R}_{+}\times W_{k_{j-1}}(y) $ to rotate into a function defined on $\mathbb{R}_{+}\times W_{k_{j}}(y)$. As the parameter $j$ is taken from $k_{0}=0$ to $k_{n}$, the 0-dimensional space $\{t_{y}(y)(s(y))\}$ is elevated to finite-dimensional affine subspace $V_{n}$ of $E_{n}$.
On the other hand,
$\theta_{t}^{k_{j}}$ 
below the slope is defined by the formula 
\[\theta_{t}^{k_{j}}(g\hat\otimes h\hat\otimes b)=g_{t}(X\hat\otimes 1+1\hat\otimes (C_{{W_{k_{j}}(y)}^{\perp}}+D_{E_{n}}))(1\hat\otimes M_{h_{t}})\hat\otimes b
\]from $\mathcal{A}(W_{{k_{j}}}(y))\hat\otimes\mathcal{K}$ to $\mathcal{SK}(\mathcal{L}^{2}_{n})\hat\otimes \mathcal{K}$,
performs a functional calculus on 
the Clifford operator of the orthogonal complement of $W_{k_{j}}(y)$ in $E_{n}$, and a functional calculus on the Dirac operator of the entire linear space $E_{n}$. It is clear that the map $\alpha_{t}$ does not asymptotically depend on the choice of $k$ (Cf. Proposition 2 in Section 4 of \cite{HKT98}, or Lemma 3.5 in \cite{HG04}). Moreover, the composition of all the Bott maps and the Dirac maps is essentially a functional calculation of the Bott-Dirac operator of the entire space $E_{n}$. 
This fact enables us to establish a geometric analogue of the Bott periodicity in finite dimensions, see Theorem \ref{the. 7.8} later.
\end{remark}

\begin{definition}\label{def. 7.6} Let $d\geqslant 0$. For any $t\in [1,\infty)$, define a map
\[(\alpha_{L})_{t}:\mathbb{C}_{u,L,\infty}[(P_{d}(X_{n}),\mathcal{A}(V_{n}))_{n\in\mathbb{N}}]\rightarrow\mathbb{C}_{u,L,\infty}[(P_{d}(X_{n}),\mathcal{SK}(\mathcal{L}^{2}_{n}))_{n\in\mathbb{N}}]\]
by the formula
\[((\alpha_{L})_{t}(f))(s)=(\alpha_{t})(f(s))\]
for all $s\in \mathbb{R}_{+}$.\end{definition}

\begin{theorem}\label{the. 7.7}
For each $d\geqslant 0$, the maps $(\alpha_{t})_{t\geqslant 1}$ and $((\alpha_{L})_{t})_{t\geqslant 1}$
extend to asymptotic morphisms
\begin{align*}
	\alpha_{t}:C^{*}_{u,\max,\infty}((P_{d}(X_{n}),\mathcal{A}(V_{n}))_{n\in\mathbb{N}})&\rightarrow C^{*}_{u,\max,\infty}((P_{d}(X_{n}),\mathcal{SK}
	(\mathcal{L}^{2}_{n}))_{n\in\mathbb{N}}),\\
	(\alpha_{L})_{t}:C^{*}_{u,L,\max, \infty}((P_{d}(X_{n}),\mathcal{A}(V_{n}))_{n\in\mathbb{N}})&\rightarrow C^{*}_{u,L,\max\infty}((P_{d}(X_{n}),\mathcal{SK}
	(\mathcal{L}^{2}_{n}))_{n\in\mathbb{N}}).
\end{align*}\end{theorem} 
\begin{proof}[\rm{\textbf{Proof.}}]
\textbf{Step 1.} We first fix some notations. Let $l\geqslant 0$ large enough, for any subset 
$C\subseteq B_{P_{d}(Y_{n})}(y,l)\cap B_{P_{d}(Y_{n})}(y',l)$ where $y,y'\in F_{d,n}$ (for all but finitely many $n$) with $d(y,y')<l$, denote by
\begin{align*}
	W_{C}(y)&=\text{affine-span}\{t_{y}(w)(s(w))\mid w\in C\}=t_{yy'}(W_{C}(y')),\\
	W_{C}(y')&=\text{affine-span}\{t_{y'}(w)(s(w))\mid w\in C\}=t_{y'y}(W_{C}(y)).
\end{align*}
For any affine subspace $W$ with $W_{C}(y)\subseteq W\subseteq E_{n}$, we identify
$\mathcal{A}(W_{C}(y))$ with a subalgebra of $\mathcal{A}(W)$ via the map $\beta_{W,W_{C}(x)}$ defined in Definition \ref{def. 5.1}. For any $i_{0}>0$, $r>0$, $c>0$, denote by
$[\mathcal{A}(W_{C}(y))\hat\otimes \mathcal{K}]_{i_{0},r,c}$
the subset of $\mathcal{A}(W_{C}(y))\hat\otimes \mathcal{K}$ consisting of those elements 
of the form $\sum_{i=1}^{i_{0}} g_{i}\hat{\otimes}h_{i}\hat{\otimes}{k_{i}}$, where $g_{i}\in\mathcal{S}$, $h_{i}\in\mathcal{C}(W_{k}(y))$, $k_{i}\in\mathcal{K}$, $1\leqslant i\leqslant i_{0},$ satisfying
\begin{enumerate}
    \item [(1)]\hspace{0pt}$\text{Supp}(g_{i})\subseteq[-r,r]$ for each $1\leqslant i\leqslant i_{0}$;
    \item [(2)]\hspace{0pt}both $g_{i}$ and
$h_{i}$ are continuously differentiable and their derivatives satisfy $\|g'_{i}\|\leqslant c$ and 
$\|\bigtriangledown_{v}h_{i}\|\leqslant c$ for all $v\in W^{0}_{C}(y)$ such that $\|v-t_{y}(y)(s(y))\|\leqslant 1$.
\end{enumerate}
Then it follows from Lemma 7.5 in \cite{Yu00} and Lemma 2.9 in \cite{HKT98} that for any $\epsilon>0$ there exists $t_{0}>1$
such that for all $t\geqslant t_{0}$ and all $a,b\in[\mathcal{A}(W_{C}(y))\hat\otimes \mathcal{K}]_{i_{0},r,c}$, we have 
\begin{align*}
	\begin{split}
		\lim_{t\rightarrow\infty}\left\{
		\begin{array}{c}
	\theta^{l}_{t}(x)(ab)-(\theta^{l}_{t}(x)(a))(\theta^{l}_{t}(x)(b))\\
		\theta^{l}_{t}(x)(a+b)-(\theta^{l}_{t}(x)(a)+\theta^{l}_{t}(x)(b))\\
			\lambda(\theta^{l}_{t}(x)(a))-\theta^{l}_{t}(x)(\lambda a)\\
			\theta^{l}_{t}(x)(a^{*})-\theta^{l}_{t}(x)(a)^{*}
		\end{array}
		\right\}=0,
	\end{split}
\end{align*}
so that the maps
$(\theta^{l}_{t})_{t\geqslant 1}$ defines an asymptotic morphism from $\mathcal{A}(W_{C}(y))\hat\otimes\mathcal{K}$ to $\mathcal{SK}(\mathcal{L}^{2}_{n})\hat\otimes\mathcal{K}
$. 

\textbf{Step 2.} We next prove that the maps $(\alpha_{t})_{t\geqslant 1}$ is a well-defined asymptotic morphism from  $\mathbb{C}_{u,\infty}[(P_{d}(X_{n}),\mathcal{A}(V_{n}))_{n\in\mathbb{N}}]$ to 
$C^{*}_{u,\max,\infty}((P_{d}(X_{n}),\mathcal{SK}
(\mathcal{L}^{2}_{n}))_{n\in\mathbb{N}})$. It suffices to verify that
$\alpha_{t}(T)$ is norm continuous in $t\in [1,\infty)$, and 
\begin{align*}
	\begin{split}
		\lim_{t\rightarrow\infty}\left\{
		\begin{array}{c}
			\alpha_{t}(TS)-\alpha_{t}(T)\alpha_{t}(S)\\
		\alpha_{t}(T+S)-(\alpha_{t}(T)+\alpha_{t}(S))\\
			\lambda\alpha_{t}(T)-\alpha_{t}(\lambda T)\\
			\alpha_{t}(T^{*})-(\alpha_{t}(T))^{*}
		\end{array}
		\right\}=0
	\end{split}
\end{align*}
for all 
$T,S\in\mathbb{C}_{u,\infty}[(P_{d}(X_{n}),\mathcal{A}(V_{n}))_{n\in\mathbb{N}}]$ and $\lambda\in\mathbb{C}$. Below we shall only prove the first equality, others can be checked similarly. By Definition \ref{def. 5.5}, we calculate that
\small
\begin{equation}\label{equ. 5}
\begin{split}
	&\sup_{x,x'\in Z_{d,n},n\in\mathbb{N}}\|
		(\alpha_{t}(TS))^{(n)}(x,x')-(\alpha_{t}(T))^{(n)}\cdot(\alpha_{t}(S))^{(n)}(x,x')\|\\
	&=\sup_{x,x'\in Z_{d,n},n\in\mathbb{N}}\|
		\theta^{k}_{t}(y)((T_{1}S_{1})^{(n)}(x,x'))-
		\sum_{\xi\in Z_{d,n}}(\theta^{k}_{t}(y)(T_{1}^{(n)}(x,\xi)))\cdot ((t_{yy'})_{*}
		\theta^{k}_{t}(y')(S_{1}^{(n)}(\xi,x'))) \|\\
	&=\sup_{x,x'\in Z_{d,n},n\in\mathbb{N}}\|
		\theta^{k}_{t}(y)(\sum_{\xi\in Z_{d,n}}T^{(n)}_{1}(x,\xi)\cdot ((t_{yy'})_{*}S^{(n)}_{1}(\xi,x')))-
		\sum_{\xi\in Z_{d,n}}\theta^{k}_{t}(y)(T_{1}^{(n)}(x,\xi))\cdot (t_{yy'})_{*}
		(\theta^{k}_{t}(y')(S_{1}^{(n)}(\xi,x'))) \|\\
	&=\sup_{x,x'\in Z_{d,n},n\in\mathbb{N}}\|
		\sum_{\xi\in Z_{d,n}}\theta^{k}_{t}(y)T^{(n)}_{1}(x,\xi)\cdot (\theta^{k}_{t}(y)(t_{yy'})_{*}S^{(n)}_{1}(\xi,x')) -
		\sum_{\xi\in Z_{d,n}}\theta^{k}_{t}(y)T^{(n)}_{1}(x,\xi)\cdot (t_{yy'})_{*}
		\theta^{k}_{t}(y')S^{(n)}_{1}(\xi,x') \|\\
	&=\sup_{x,x'\in Z_{d,n},n\in\mathbb{N}}\|
		\sum_{\xi\in Z_{d,n}}\theta^{k}_{t}(y)T^{(n)}_{1}(x,\xi)\cdot[ \theta^{k}_{t}(y)(t_{yy'})_{*}S^{(n)}_{1}(\xi,x')- (t_{yy'})_{*}
		\theta^{k}_{t}(y')S^{(n)}_{1}(\xi,x')]\|
  \end{split}
\end{equation}
\normalsize
where $y=p_{n}(x),y'=p_{n}(\xi)$, and $k$ is a non-negative number such that, for any pair $x$ and $x'$ in $Z_{d,n},n\in\mathbb{N}$, there exists $T_{1}(x,x'),S_{1}(x,x')\in\mathcal{A}(W_{k}(y))\hat\otimes\mathcal{K}$ satisfying
\begin{align*}
	T^{(n)}(x,x')&=(\beta_{V_{n},W_{k}(y)}\hat{\otimes}1)({T_{1}}^{(n)}(x,x')),\\
	S^{(n)}(x,x')&=(\beta_{V_{n},W_{k}(y)}\hat{\otimes}1)({S_{1}}^{(n)}(x,x')).
\end{align*}
Let $y,y'\in F_{d,n}$ be such that $d(y,y')\leqslant R$
for some $R>0$. If $b\in[\mathcal{A}(W_{k}(y'))\hat\otimes \mathcal{K}]_{i_{0},r,c}$, viewing $\mathcal{A}(t_{yy'}(W_{k}(y')))$ as a subalgebra of $\mathcal{A}(W_{R+k}(y))$ via the map $\beta_{W_{R+k}(y),t_{yy'}(W_{k}(y'))}$, then we have
\[(t_{yy'})_{*}(b)\in[\mathcal{A}(t_{yy'}(W_{k}(y')))\hat\otimes \mathcal{K}]_{i_{0},r,c}\subseteq [\mathcal{A}(W_{R+k}(y))\hat\otimes\mathcal{K}]_{i_{0},r,c} .\]
By Proposition 4.2 in \cite{HKT98} one can verify that the following diagram
\begin{center}
\begin{tikzcd}
	{\mathcal{A}(W_{k}(y'))\hat\otimes\mathcal{K}} \arrow[rr,"(t_{yy'})_{*}"] \arrow[d,"\theta^{k}_{t}(y')"] &  & {\mathcal{A}(t_{yy'}(W_{k}(y')))\hat\otimes \mathcal{K}}\arrow[rr,"\beta_{W_{R+k}(y),t_{yy'}(W_{k}(y'))}"] &  & {{\mathcal{A}(W_{R+k}(y))\hat\otimes\mathcal{K}}} \arrow[d,"\theta^{R+k}_{t}(y)"] \\
	{\mathcal{SK}(\mathcal{L}^{2}_{n})\hat\otimes\mathcal{K}  } \arrow[rrrr,"(t_{yy'})_{*}"]         &  &               &  & {\mathcal{SK}(\mathcal{L}^{2}_{n})\hat\otimes\mathcal{K}  }          
\end{tikzcd}
\end{center}
is asymptotically commutative.
Together with Lemma 7.3 and Lemma 7.4 in \cite{Yu00},
we have that
\[(t_{yy'})_{*}(\theta^{k}_{t}(y'))(b)-(\theta^{R+k}_{t}(y))(\beta_{W_{R+k}(y),t_{yy'}(W_{k}(y'))}((t_{yy'})_{*}(b))\]
converges uniformly to $0$ in norm on the $R$-strip
$\{(y,y')\in F_{d,n}\times F_{d,n}\mid d(y,y')\leqslant R\}$ as $t$ goes to $\infty$. In general case, for suitable large $l> R+k$ and non-empty subset $C\subseteq B_{P_{d}(Y_{n})}(y,l)\cap B_{P_{d}(Y_{n})}(y',l)$ with $n$ sufficient large, the affine isometry $t_{yy'}:W_{C}(y')\rightarrow W_{C}(y)$ induces the following asymptotically commutative diagram
\begin{center}
		\begin{tikzcd}
	\mathcal{A}(W_{C}(y'))\hat\otimes\mathcal{K} \arrow["\theta^{l}_{t}(y')",d] \arrow[r,"(t_{yy'})_{*}"] &\mathcal{A}(W_{C}(y))\hat\otimes\mathcal{K} \arrow[d,"\theta^{l}_{t}(y)"] \\
			\mathcal{SK}(\mathcal{L}^{2}_{n})\hat\otimes\mathcal{K}\arrow[r,"(t_{yy'})_{*}"] & \mathcal{SK}(\mathcal{L}^{2}_{n})\hat\otimes\mathcal{K}
		\end{tikzcd}
\end{center}
which implies that 
\[(t_{yy'})_{*}((\theta^{l}_{t}(y'))(b))-(\theta^{l}_{t}(y))((t_{yy'})_{*}(b))\]converges uniformly to $0$ in norm as $t$ goes to $\infty$,
so is (\ref{equ. 5}). Subsequently, by Lemma 3.4 in \cite{GWY08},  \[\alpha_{t}(TS)-\alpha_{t}(T)\alpha_{t}(S)\] converges uniformly to $0$ in norm for all 
$T,S\in\mathbb{C}_{u,\infty}[(P_{d}(X_{n}),\mathcal{A}(V_{n}))_{n\in\mathbb{N}}]$ as $t$ goes to $\infty$.

Consequently,
the maps $(\alpha_{t})_{t\geqslant 1}$ is a well-defined asymptotic morphism from $*$-algebra $\mathbb{C}_{u,\infty}[(P_{d}(X_{n}),\mathcal{A}(V_{n}))_{n\in\mathbb{N}}]$ to $C^{*}$-algebra
$C^{*}_{u,\max,\infty}((P_{d}(X_{n}),\mathcal{SK}
(\mathcal{L}^{2}_{n}))_{n\in\mathbb{N}})$. 
It follows that the maps $(\alpha_{t})_{t\geqslant 1}$ define a $*$-homomorphism $\alpha$ from $\mathbb{C}_{u,\infty}[(P_{d}(X_{n}),\mathcal{A}(V_{n}))_{n\in\mathbb{N}}]$ to the asymptotic $C^{*}$-algebra
\[\mathcal{U}(C^{*}_{u,\max,\infty}((P_{d}(X_{n}),\mathcal{SK}
(\mathcal{L}^{2}_{n}))_{n\in\mathbb{N}})):=\frac{C_{b}([1,\infty),C^{*}_{u,\max,\infty}((P_{d}(X_{n}),\mathcal{SK}
	(\mathcal{L}^{2}_{n}))_{n\in\mathbb{N}})}{C_{0}([1,\infty),C^{*}_{u,\max,\infty}((P_{d}(X_{n}),\mathcal{SK}
	(\mathcal{L}^{2}_{n}))_{n\in\mathbb{N}})}
\]satisfying $\|\alpha_{t}(T)\|\leqslant\|T\|$ for all $t\in[1,\infty)$ and $T\in\mathbb{C}_{u,\infty}[(P_{d}(X_{n}),\mathcal{A}(V_{n}))_{n\in\mathbb{N}}]$.

\textbf{Step 3.}
It follows from the
universality of the maximal norm that the maps $(\alpha_{t})_{t\geqslant 1}$ extends to an asymptotic morphism from $C^{*}_{u,\max\infty}((P_{d}(X_{n}),\mathcal{A}(V_{n}))_{n\in\mathbb{N}})$ to $C^{*}_{u,\max,\infty}((P_{d}(X_{n}),\mathcal{SK}
(\mathcal{L}^{2}_{n}))_{n\in\mathbb{N}})$.This finishes the proof.
\end{proof}

\subsection{A geometric analogue of the Bott periodicity in finite dimensions}\label{subsection. 7.2}
The following is a geometric analogue of the Bott periodicity in finite dimensions.
\begin{theorem}\label{the. 7.8}
For each $d\geqslant 0$, the compositions $\alpha_{*}\circ\beta_{*}$ and $(\alpha_{L})_{*}\circ(\beta_{L})_{*}$
are the identity.\end{theorem} 
\begin{proof}[\rm{\textbf{Proof.}}]
\textbf{Step 1.} Let $\gamma$ be the asymptotic morphism 
\[\gamma:\mathcal{S}\hat\otimes C^{*}_{u,\max\infty}((P_{d}(X_{n}))_{n\in\mathbb{N}})\rightarrow C^{*}_{u,\max,\infty}((P_{d}(X_{n}),\mathcal{SK}
(\mathcal{L}^{2}_{n}))_{n\in\mathbb{N}})\]
defined by the formula 
\[\gamma_{t}(g\hat\otimes T)=[((\gamma_{t}(g\hat\otimes T))^{(0)},\dots,(\gamma_{t}(g\hat\otimes T))^{(n)},\dots)]\]
for all $g\in\mathcal{S}$ and $T=[(T^{(0)},\dots,T^{(n)},\dots)]\in\mathbb{C}_{u,\infty}[(P_{d}(X_{n}))_{n\in\mathbb{N}}]$, where
\[(\gamma_{t}(g\hat\otimes T))^{(n)}(x,x')=g_{t^{2}}(X\hat\otimes 1+1\hat\otimes B_{E_{n},t_{p_{n}(x)}(p_{n}(x))(s(p_{n}(x)))})\hat\otimes T^{(n)}(x,x')\]
for any $t\in[1,\infty)$, $x,x'\in Z_{d,n}$ with $n\in\mathbb{N}$, and $B_{E_{n},t_{p_{n}(x)}(p_{n}(x))(s(p_{n}(x)))}$ is the Bott-Dirac operator of $E_{n}$ at base point $t_{p_{n}(x)}(p_{n}(x))(s(p_{n}(x)))$.
It follows from Proposition 2.13, Appendix in \cite{HKT98} and Proposition 7.7 in \cite{Yu00} that the composition $\alpha\circ\beta$ is asymptotically equivalent to $\gamma$. Hence, $\alpha^{*}\circ\beta^{*} =\gamma^{*}$.\par 
\textbf{Step 2.} Define $\delta$ to be the asymptotic morphism
\[\delta:\mathcal{S}\hat\otimes C^{*}_{u,\max\infty}((P_{d}(X_{n}))_{n\in\mathbb{N}})\rightarrow C^{*}_{u,\max,\infty}((P_{d}(X_{n}),\mathcal{SK}
(\mathcal{L}^{2}_{n}))_{n\in\mathbb{N}})\]
by the formula 
\[\delta_{t}(g\hat\otimes T)=[((\delta_{t}(g\hat\otimes T))^{(0)},\dots,(\delta_{t}(g\hat\otimes T))^{(n)},\dots)]\]
for all $g\in\mathcal{S}$ and $T=[(T^{(0)},\dots,T^{(n)},\dots)]\in\mathbb{C}_{u,\infty}[(P_{d}(X_{n}))_{n\in\mathbb{N}}]$, where
\[(\delta_{t}(g\hat\otimes T))^{(n)}(x,x')=g_{t^{2}}(X\hat\otimes 1+1\hat\otimes B_{E_{n},0})\hat\otimes T^{(n)}(x,x')\]
for all $t\geqslant 1$, $x,x'\in Z_{d,n}$, $n\in\mathbb{N}$, and $B_{E_{n},0}$ is the Bott-Dirac operator of $E_{n}$ at base point $0$.\par 
For each $x\in Z_{d,n}$, $n\in\mathbb{N}$, let $U_{p_{n}(x)}: \mathcal{L}^{2}_{n}\rightarrow \mathcal{L}^{2}_{n}$ 
be the unitary operator defined by
\[(U_{p_{n}(x)}(\xi))(v)=\xi(v-t_{p_{n}(x)}(p_{n}(x))(s(p_{n}(x))))
\]
for all $\xi\in\mathcal{L}^{2}_{n}$
and $v\in E_{n}$. Then
\[U^{-1}_{p_{n}(x)}B_{E_{n},t_{p_{n}(x)}(p_{n}(x))(s(p_{n}(x)))}U_{p_{n}(x)}=B_{E_{n},0}.\]
For each $s\in[0, 1]$, we define an asymptotic morphism
\[\varPhi^{(s)}:\mathcal{S}\hat\otimes C^{*}_{u,\max\infty}((P_{d}(X_{n}))_{n\in\mathbb{N}})\rightarrow C^{*}_{u,\max,\infty}((P_{d}(X_{n}),\mathcal{SK}
(\mathcal{L}^{2}_{n}))_{n\in\mathbb{N}})\hat\otimes \mathcal{M}_{2}(\mathbb{C})\]
by the formula 
\[(\varPhi^{(s)}_{t}(g\hat\otimes T))^{(n)}(x,x')=
(U^{(s)}_{p_{n}(x)})^{-1}
\left[
\begin{array}{cc}
	(\gamma_{t}(g\hat\otimes T))^{(n)}(x,x') & 0\\
	0&0\\
\end{array}\right]
 U^{(s)}_{p_{n}(x)}\]
for all $t\geqslant 1$, $x,x'\in Z_{d,n}$, $n\in\mathbb{N}$, where
 \[U^{(s)}_{y}=R(s)
\left[
\begin{array}{cc}
	U_{p_{n}(x)}\hat\otimes 1 & 0\\
	0&0\\
\end{array}\right]
{R(s)}^{-1},\quad\text{and}\quad
R(s)=\left[
\begin{array}{cc}
	\cos(\frac{\pi}{2}s) & \sin(\frac{\pi}{2}s)\\
	-\sin(\frac{\pi}{2}s)&\cos(\frac{\pi}{2}s)\\
\end{array}\right].\]
Then $\varPhi^{(s)}$ is a homotopy between the asymptotic morphisms $\gamma=\varPhi^{(1)}$ and $\delta=\varPhi^{(0)}$. Hence $\gamma_{*}=\delta_{*}$.\par 
\textbf{Step 3.} For each $n\in\mathbb{N}$, let $P^{(n)}$ be the projection of $E_{n}$ onto the 1-dimensional kernel of
the Bott-Dirac operator $B_{E_{n},0}$ at the origin $0$. For each $s\in[0, 1]$, define an asymptotic
morphism
\[\varPsi^{(s)}:\mathcal{S}\hat\otimes C^{*}_{u,\max\infty}((P_{d}(X_{n}))_{n\in\mathbb{N}})\rightarrow C^{*}_{u,\max,\infty}((P_{d}(X_{n}),\mathcal{SK}
(\mathcal{L}^{2}_{n}))_{n\in\mathbb{N}})\]
by the formula 
\begin{align*}
	\begin{split}
	(\varPsi^{s}_{t}(g\hat\otimes T))^{(n)}(x,x')=
		\left\{
		\begin{array}{c}
		g_{t^{2}}(\frac{1}{s}B_{E_{n},0})\hat\otimes T^{(n)}(x,x'),\quad \text{if}\ s\in(0,1];\\
			g(0)\cdot p^{(n)}\hat\otimes T^{(n)}(x,x'),\quad \text{if}\ s=0,\\
		\end{array}
  \right.
	\end{split}
\end{align*}
for all $t\geqslant 1$, $x,y\in Z_{d,n}$, $n\in\mathbb{N}$. Then $\varPsi^{(s)}$ is a homotopy between the asymptotic morphism $\delta$
and the $*$-homomorphism
\[\sigma:\mathcal{S}\hat\otimes C^{*}_{u,\max\infty}((P_{d}(X_{n}))_{n\in\mathbb{N}})\rightarrow C^{*}_{u,\max,\infty}((P_{d}(X_{n}),\mathcal{SK}
(\mathcal{L}^{2}_{n}))_{n\in\mathbb{N}})\]
defined by the formula 
\[\sigma(g\hat\otimes[(T^{(0)},\dots,T^{(n)},\dots)])=g(0)\cdot [(P^{(n)}\hat\otimes T^{(0)},\dots,P^{(n)}\hat\otimes T^{(0)},\dots)].\]
It is clear that $\sigma$ induces the identity on $K$-theory. Therefore, we conclude that
\[\alpha_{*}\circ\beta_{*}=\gamma_{*}=\delta_{*}=\text{identity}\]
on $K_{*}(C^{*}_{u,\max\infty}((P_{d}(X_{n}))_{n\in\mathbb{N}}))$. The case for $(\alpha_{L})_{*}\circ (\beta_{L})_{*}$ is similar.
\end{proof}
\begin{proof}[\rm\textbf{Proof of Theorem 1.1 }]By the constructions of Bott maps in Section \ref{section 5} and Dirac maps in Section \ref{section 7}, we obtain the following commutative diagram

\begin{center}
	\begin{tikzcd}
\lim\limits_{d\rightarrow\infty}K_{*+1}(C^{*}_{u, L, \max,\infty}(P_{d}(X_{n}))_{n\in\mathbb{N}}) \arrow[d,"(\beta_{L})_{*}"] \arrow[r,"e_{*}"] & \lim\limits_{d\rightarrow\infty}K_{*+1}(C^{*}_{u, \max,\infty}(P_{d}(X_{n}))_{n\in\mathbb{N}})\arrow[d,"\beta_{*}"] \\
	\lim\limits_{d\rightarrow\infty}K_{*}(C^{*}_{u, L, \max,\infty}((P_{d}(X_{n}),\mathcal{A}(V_{n}))_{n\in\mathbb{N}})) \arrow[d,"(\alpha_{L})_{*}"] \arrow[r,"e^{\mathcal{A}}_{*}"] & 	\lim\limits_{d\rightarrow\infty}K_{*}(C^{*}_{u, \max,\infty}((P_{d}(X_{n}),\mathcal{A}(V_{n}))_{n\in\mathbb{N}}))\arrow[d,"\alpha_{*}"] \\
\lim\limits_{d\rightarrow\infty}K_{*}(C^{*}_{u, L, \max,\infty}((P_{d}(X_{n}),\mathcal{SK}(\mathcal{L}^{2}(E_{n})))_{n\in\mathbb{N}})) \arrow[r,"e_{*}"] & 	\lim\limits_{d\rightarrow\infty}K_{*}(C^{*}_{u, \max,\infty}((P_{d}(X_{n}),\mathcal{SK}(\mathcal{L}^{2}(E_{n})))_{n\in\mathbb{N}})).
	\end{tikzcd}
\end{center}
It follows from Theorem \ref{the. 7.8}, together with an argument of diagram chasing that the evaluation map 
\[e_{*}:\lim\limits_{d\rightarrow\infty}K_{*}(C^{*}_{u, L, \max,\infty}(P_{d}(X_{n}))_{n\in\mathbb{N}}) \rightarrow \lim\limits_{d\rightarrow\infty}K_{*}(C^{*}_{u, \max,\infty}(P_{d}(X_{n}))_{n\in\mathbb{N}})\]
is an isomorphism. Theorem \ref{the. 1.1} is supported by Theorem \ref{the. 4.6}.\end{proof}

\bibliographystyle{alpha}
\bibliography{ref}

\end{document}